\documentclass[10pt,leqno]{amsart}

\PassOptionsToPackage{linktocpage}{hyperref}

\usepackage[french, english]{babel}
\usepackage[utf8,latin1]{inputenc}

\usepackage{amsmath,amscd,amssymb,amsthm,epsfig,array,hhline,longtable,float,eufrak,mathtools}
\usepackage{stmaryrd}
\restylefloat{table}
\counterwithin{table}{section}
\usepackage{bbm}
\usepackage{url}

\usepackage[margin=3cm]{geometry}

\usepackage{hyperref}

\usepackage{tikz-cd}

\makeatletter
\renewcommand\subsubsection{\@startsection{subsubsection}{3}{\z@}%
                                     {-3.25ex\@plus -1ex \@minus -.2ex}%
                                     {-1em}
                                     {\normalfont\normalsize\bfseries}}
\makeatother

\newtheorem{theorem}[subsubsection]{Theorem}
\newtheorem{lemm}[subsubsection]{Lemma}
\newtheorem{definition}[subsubsection]{Definition}
\newtheorem{remark}[subsubsection]{Remark}
\newtheorem*{remark*}{Remark}

\newtheorem{coro}[subsubsection]{Corollary}
\newtheorem{prop}[subsubsection]{Proposition}

\newtheorem{prop-def}[subsubsection]{Proposition-Definition}
\numberwithin{equation}{section}

\newtheorem{atheorem}{Theorem}[section]

\newtheorem{aremark}[atheorem]{Remark}

\newtheorem{acoro}[atheorem]{Corollary}
\newtheorem{aprop}[atheorem]{Proposition}

\newcommand{\overbar}{\bar}

\newcommand{\Hom}{\mathrm{Hom}}

\newcommand{\End}{\mathrm{End}}
\renewcommand{\d}{\mathrm{d}}

\renewcommand{\hat}{\widehat}
\newcommand{\vol}{\mathrm{vol}}

\newcommand{\Gal}{\mathrm{Gal}}
\newcommand{\Ind}{\mathrm{Ind}}
\newcommand{\Tr}{\mathrm{Tr}}

\newcommand{\ad}{\mathrm{Ad}}
\newcommand{\Ad}{\mathrm{Ad}}
\newcommand{\cInd}{\mathrm{c\text{-}Ind}}
\newcommand{\rk}{\mathrm{rk}}
\newcommand{\res}{\mathrm{res}}

\newcommand{\AAA}{\mathbb{A}}
\newcommand{\bbb}{\mathfrak{b}}
\renewcommand{\ggg}{\mathfrak{g}}
\newcommand{\hhh}{\mathfrak{h}}

\renewcommand{\lll}{\mathfrak{l}}
\newcommand{\nnn}{\mathfrak{n}}
\newcommand{\mmm}{\mathfrak{m}}
\newcommand{\ppp}{\mathfrak{p}}

\newcommand{\zzz}{\mathfrak{z}}
\newcommand{\ttt}{\mathfrak{t}}

\newcommand{\ooo}{\mathcal{O}}

\newcommand{\car}{\mathfrak{c}}
\newcommand{\rs}{\mathrm{reg}}
\newcommand{\ani}{\mathrm{ani}}
\newcommand{\ev}{{\mathrm{ev}}}

\newcommand{\ago}{\mathfrak{a}}

\DeclareMathOperator{\reg}{reg}

\DeclareMathOperator{\Spec}{Spec}

\renewcommand{\leq}{\leqslant}
\renewcommand{\geq}{\geqslant}
\newcommand{\htau}{\hat \tau}

\makeatletter
\newcommand*{\rom}[1]{\expandafter\@slowromancap\romannumeral #1@}
\makeatother

\begin{document}

\setcounter{secnumdepth}{3}
\setcounter{tocdepth}{2}

\title[Automorphic representations and Hitchin's moduli spaces]{Number of cuspidal automorphic representations and Hitchin's moduli spaces}

\author{Hongjie Yu}
\address{IST Austria\\ Am Campus 1\\ Klosterneuburg 3400\\ Austria}
\curraddr{Weizmann Institute of Science\\ Herzl St 234\\ Rehovot\\ Israel}
    \email{hongjie.yu@weizmann.ac.il}


\begin{abstract}

Let $F$ be the function field of a projective smooth geometrically connected curve $X$ defined over a finite field $\mathbb{F}_q$. Let $G$ be a split semisimple algebraic group over $\mathbb{F}_q$. Let $S$ be a non-empty finite set of points of $X$. 
We are interested in the number of $G$ cuspidal automorphic representations whose local behaviors in $S$ are prescribed. In this article, we consider those cuspidal automorphic representations whose local component at each $v\in S$ contains a fixed irreducible Deligne-Lusztig induced representation of a hyperspecial group. 
We express that the count in terms of groupoid cardinality of $\mathbb{F}_q$-points of Hitchin moduli stacks of groups associated with $G$. In the course of the proof,  we study the geometry of Hitchin moduli stacks and prove some vanishing results on the geometric side of a variant of the Arthur-Selberg trace formula for test functions with small support. 
\end{abstract}

\maketitle

\tableofcontents

\section{Introduction}Let $F$ be a function field of a smooth projective and geometrically connected curve $X$ defined over $\mathbb{F}_q$. 
In 1981,  Drinfeld \cite{Drinfeld} computes the number of everywhere unramified cuspidal automorphic representations over $F$ for $GL_2$.  Using global Langlands correspondence, he observes that the number of those rank $2$ $\bar{\mathbb{Q}}_\ell$-irreducible continuous representations of the geometric fundamental group of  $X$ that can be extended to representations of the arithmetic fundamental group of $X$
is similar to the number of $\mathbb{F}_q$-points of a variety over $\mathbb{F}_q$ of dimension $4g-3$, where $g$ is the genus of the curve $X$.  
Deligne proposes in \cite{Deligne} to extend Drinfeld's result to higher ranks with prescribed ramifications. If the ramifications are tame, Deligne conjectures that this phenomenon has a relation with the coarse moduli space of vector bundles with connections in characteristic $0$ (the coincidence of Euler characteristic, etc.). So conjecturally, it should also have a relation with the moduli space of Higgs bundles due to Simpson's non-abelian Hodge theory. For the general linear groups $GL_n$, there has been some progress in Deligne's conjectures. Based on both Simpson's theory and the Langlands programs, it is natural to extend Deligne's conjectures to representations of the fundamental group of $X$ in reductive groups other than the general linear groups. We hope to gain more insights from this direction, yet very little is known at this time.

In fact, unlike the case of $GL_n$, where we have a relatively comprehensive understanding of the Langlands correspondence, many statements regarding the Langlands correspondence for reductive groups remain conjectural and are not expected to be as straightforward as those for $GL_n$.
Ignoring these issues, it is already very interesting to consider this question on the automorphic side. In fact, calculating the dimension of the space of modular forms is a classical problem in arithmetic. 

Let us introduce Deligne's questions on the automorphic side. Let $G$ be a split semisimple algebraic group defined over $\mathbb{F}_q$ (we can do a similar thing for a reductive group, but let us restrict to semisimple algebraic groups for simplicity).  
We fix a finite set of places $S$ of $F$. For each place $v\in S$, we wish to prescribe local behaviors of cuspidal automorphic representations in $v$ that can reflect the restriction to the inertial Galois subgroup of the corresponding local Langlands parameter. Then we are interested in the number of $G$-cuspidal automorphic representations $\pi$ (with multiplicity) that exhibit the prescribed local behaviors in $S$ and are unramified outside $S$. 
If we only consider tame ramifications, then it is expected that $\pi_v$ ($v\in S$) is of depth zero and the restriction of $\pi_v$ to a parahoric subgroup is inflated from specified representations of the finite reductive quotient of the parahoric subgroup.  

We expect the count to be connected to geometric objects. Suggested by known cases in $GL_n$, we hope to see the appearance of Hitchin moduli spaces. Surprisingly, we will discover that not only the Hitchin moduli space of $G$ contributes to the count but also the Hitchin moduli space of certain relevant groups. This looks strange compared with the Simpson correspondence over $\mathbb{C}$, which establishes a bijection between some $G$-local systems and $G$-Higgs bundles. However, under the Langlands conjectures, we potentially establish a connection between $\hat{G}(\bar{\mathbb{Q}}_\ell)$-local systems and $G$-Higgs bundles. Here $\hat{G}$ denotes the Langlands dual group of $G$ defined over $\bar{\mathbb{Q}}_\ell$.  

In this article, we consider the case that the parahoric subgroup is contained in a hyperspecial subgroup, and the representation of the finite reductive group is a Deligne-Lusztig induced representation (therefore, is induced from a character in general position of a torus). 
We generalize a criterion of Deligne implying that if these representations are in general position, then any automorphic representation whose local behaviors are thus prescribed lies automatically in the cuspidal spectrum.

We use the Arthur-Selberg trace formula for this question.
Ngô B.C. has observed in \cite{NgoH} that the groupoid cardinality of the category of the $\mathbb{F}_q$-points of the Hitchin moduli stack is related to an analog of the automorphic trace formula for the Lie algebra $\ggg$ of $G$. Although neither the cardinality nor the trace is finite, their so-called elliptic (or anisotropic) parts are well-defined and still match for a well-chosen test function. 
The convergence problem for the trace formula over a number field is solved by Arthur by introducing a truncation process.
Chaudouard and Laumon (\cite{C-L1}\cite{Chau}) find that Arthur's truncation process for trace formulas of a Lie algebra is closely related to the semistability of Hitchin bundles for the same test function.  
However, all these works do not link directly to Arthur's trace formula. 
Hence has no direct link with automorphic forms. Arguably, this article's most interesting technical point is that we pass from a group to its Lie algebra for a trace formula with any test function of small support.

\subsection{Main theorems} 
We introduce the geometric objects that will be related to the number of cuspidal automorphic representations. We suppose that the characteristic of the field $\mathbb{F}_q$ is very good for $G$ (see the beginning of \ref{chara}).

\subsubsection{Hitchin moduli stacks}
 We fix a split maximal torus $T$ of $G$ defined over $\mathbb{F}_q$. Let $W$ be the Weyl group of $G$ with respect to $T$ and $\ttt$ be the Lie algebra of $T$. 

Fix a point $\infty\in  X(\mathbb{F}_q)$ identified with a closed point of $X$. Let $S$ be a finite set of closed points of $X$ containing $\infty$. Let $D=K_X+\sum_{v\in S}v$, where $K_X$ is a canonical divisor of $X$, i.e., the line bundle associated to $K_X$ is the canonical line bundle of $X$.  Let $\mathcal{M}_G$ be the Hitchin moduli stack that classifies $G$-torsors $\mathcal{E}$ over $X$ together with a global section of the associated $D$-twisted adjoint vector bundle $\Ad(\mathcal{E})\otimes\mathcal{O}_X(D)$. Let $\mathcal{A}_G$ be the associated Hitchin base which is the space of global sections of the affine bundle $\car_{G, D}=(\ttt \sslash W)\times^{\mathbb{G}_m}\mathcal{O}_X(D)$ over $X$. We denote the Hitchin fibration by $f$:
\[f:\mathcal{M}_G\longrightarrow \mathcal{A}_G. \] 

Let $\mathcal{R}_G$ be the $\mathbb{F}_q$-scheme 
\[\prod_{v\in {S}}\mathrm{R}_{\kappa_v|\mathbb{F}_q}\car_{G, D,v}, \] 
where $\car_{G,D,v}$ is the $\kappa_v$-scheme defined to be the fiber of the bundle $\car_{G,D}$ in $v$ and $\mathrm{R}_{\kappa_v|\mathbb{F}_q}$ is the functor of restriction of scalars. 
We have an evaluation morphism: 
\(\mathrm{ev}_G: \mathcal{A}_G\longrightarrow \mathcal{R}_G \) that evaluates a global section in $S$. We call its composition with the Hitchin fibration the residue morphism: 
\[\mathrm{res}_G: \mathcal{M}_G\longrightarrow \mathcal{R}_G.   \]
Let $\widetilde{\mathcal{R}}_G:=  \ttt_{D,\infty}^{\mathrm{reg}}\prod_{v\in S-\{\infty\}}\mathrm{R}_{\kappa_v|\mathbb{F}_q}\car_{G,D,v}$, where $\ttt^{\mathrm{reg}}\subseteq \ttt$ is the open subset of regular elements. Then we have an \'etale morphism
 \( \widetilde{\mathcal{R}}_G  \longrightarrow \mathcal{R}_G ,\). It induces a Cartesian diagram
\[\begin{CD}
\widetilde{\mathcal{M}}_G @>>> \widetilde{\mathcal{A}}_G @>\widetilde{\ev}_G >> \widetilde{\mathcal{R}}_G  \\
@VVV   @VVV@VVV  \\
\mathcal{M}_G@>>>\mathcal{A}_G@>\ev_G >>\mathcal{R}_G
\end{CD}.  \] 
The stack $\widetilde{\mathcal{M}}_G$ plays an important role in Ngô's proof of the fundamental lemma. 
We use the notation $\widetilde{\res}_G$ for the morphism 
\begin{equation}\widetilde{\res}_G: \widetilde{\mathcal{M}}_G^{\xi}\longrightarrow \widetilde{\mathcal{R}}_G\end{equation} defined to be the composition of the upper horizontal arrows. 

Let $\xi\in X_{\ast}(T)\otimes_{\mathbb{Z}} \mathbb{Q}$. Chaudouard and Laumon (\cite{C-L1}) have defined a $\xi$-stability on $\widetilde{\mathcal{M}}_G$, so that its $\xi$-stable part is a separated Deligne-Mumford stack, denoted by $\widetilde{\mathcal{M}}_G^{\xi}.$ 
 For a point $o\in \widetilde{\mathcal{R}}_G(\mathbb{F}_q)$, we denote the fiber of $\widetilde{\res}_G$ in $o$ by $\widetilde{\mathcal{M}}^{\xi}_G(o)$. 
 
 In this article, we study the dimension of automorphic forms on $G$ with prescribed local behaviors. It will be related to fibers of $\widetilde{\res}_H$ for some  semisimple algebraic groups $H$ associated with $G$ defined below. 

\begin{definition}\label{coendo}
A split elliptic coendoscopic group, $H$ of $G$, is a subgroup of $G$ defined over $\mathbb{F}_q$ so that there is an elliptic element $s\in T(\mathbb{F}_q)$ such that $H=G_s^0$. Two split elliptic coendoscopic groups are called conjugate if conjugate by an element in $G(\mathbb{F}_q)$.  
\end{definition}

Given a conjugacy class of split elliptic coendoscopic group $H$ of $G$, 
we have an obviously defined finite morphism \begin{equation}\widetilde{\pi}_{H, G}: \widetilde{\mathcal{R}}^G_H\longrightarrow  \widetilde{\mathcal{R}}_G.  \end{equation}
It is independent of the group $H$ in its $G(\mathbb{F}_q)$-conjugacy class, where $\widetilde{\mathcal{R}}^G_H$ is the open subset of $\widetilde{\mathcal{R}}_H$ whose components at $\infty$ consists of $G$-regular elements.

\subsubsection{Sum of multiplicities of cuspidal automorphic representations}
Let $\mathbb{A}$ be the ring of ad\`eles of the function field $F$ and $\mathcal{O}$ its subring of integral ad\`eles. The space of cuspidal automorphic forms $C_{cusp}(G(F)\backslash G(\AAA))$ is a semisimple $G(\AAA)$-representation: 
\[C_{cusp}(G(F)\backslash G(\AAA))=\bigoplus {m_\pi} \pi.  \]
Given an irreducible smooth representation $\rho$ of $G(\ooo)$, we are mainly interested in the sum of multiplicities of cuspidal automorphic representations $\pi$ of $G(\AAA)$ containing $\rho$: 
$$\sum_{\pi: \pi_{\rho}\neq 0}m_{\pi}.   $$
Here $\pi_\rho$ denotes the $(G(\mathcal{O}), \rho)$-isotypic subspace of $\pi$. 

In this article, we only consider some special depth zero type of $\rho$ that ensure an automorphic representation $\pi$ containing $\rho$ is automatically cuspidal and has each local component of depth zero. 
For each $v\in S$, we fix a maximal torus $T_v$ of $G$ defined over $\kappa_v$. Let $\theta_v$ be a complex-valued character of $T_v(\kappa_v)$ that is absolutely regular (i.e., is in general position for every field extension, \ref{abreg}).
We fix an isomorphism between $\mathbb{C}$ and $\overbar{\mathbb{Q}}_{\ell}$ ($\ell\nmid q$). 
Let $$\rho_v=\epsilon_{\kappa_v}(T_v)\epsilon_{\kappa_v}(G) R_{T_v}^{G}\theta_v$$ be the Deligne-Lusztig induced (complex) representation of $G(\kappa_v)$. It is an irreducible representation (\cite[Proposition 7.4]{DL}), where $\epsilon_{\kappa_v}(T_v)=(-1)^{\rk_{\kappa_v}T_v}$. 
We view $\rho_v$ as a representation of $G(\ooo_v)$ by inflation and consider the representation $\rho=\otimes_{v}\rho_v$ of $G(\ooo)$ (outside $S$, $\rho_v$ is the trivial representation). 

We say that $\rho$ is a cuspidal filter, if for any $L^2$-automorphic representation, we have
$$(\ast)\quad \pi_\rho\neq 0 \implies \pi \text{ is cuspidal, i.e., contained in the cuspidal spectrum. }$$ 
Inspired by a criterion given by Deligne on the Galois side, we will provide a criterion for $\rho$ as a cuspidal filter in Proposition \ref{generic}. It shows that when $q$ tends to $+\infty$, most of $\rho$ constructed by the Deligne-Lusztig theory are cuspidal filters.  

In a manner similar to geometric settings, where the étale open $\widetilde{R}_G$ is taken into consideration, we assume that $T_\infty = T$ is a split. Let $\ttt_v$ be the Lie algebra of $T_v$. It is a vector scheme defined over $\kappa_v$. We denote the open subscheme of $G$-regular elements by $\ttt_v^{\reg}$. 

\begin{atheorem}(Proposition \ref{generic}, Theorem \ref{main} and Proposition \ref{nho})\label{best}
If the character \[\prod_{v\in S}\theta_v|_{Z_G(\mathbb{F}_q)}\] is non-trivial, then \[\sum_{\pi: \pi_{\rho}\neq 0}m_\pi=0 .  \]

Suppose that $\prod_{v\in S}\theta_v|_{Z_G(\mathbb{F}_q)}$ is trivial. Suppose, moreover, $T_\infty$ is split and that $\rho$ constructed above is a cuspidal filter.
Let $o_G$ be any point in the image of the natural map \(\prod_{v\in S}\ttt_v^{\reg}(\kappa_v) \rightarrow  \widetilde{\mathcal{R}}_G(\mathbb{F}_q), \)
 then for each conjugacy class of split elliptic coendoscopic group $H$ of $G$, each $o\in \widetilde{\pi}_{H,G}^{-1}(o_G)(\mathbb{F}_q)$, there is an explicit integer $n_{H,o}\in \mathbb{Z}$ such that 
\begin{equation}\label{Eq0}\sum_{\pi: \pi_{\rho}\neq 0}m_\pi = \sum_{H}q^{-\frac{1}{2}(\dim \mathcal{M}_H - \dim \mathcal{R}_H) }\sum_{o\in \widetilde{\pi}_{H,G}^{-1}(o_G)(\mathbb{F}_q)}n_{H,o}  |\widetilde{\mathcal{M}}_{H}^{\xi}(o)(\mathbb{F}_q)|.  \end{equation}
The first sum is taken over conjugacy classes of split elliptic coendoscopic group $H$ of $G$.  
The integers $n_{H,o}$ are given by the explicit expression (\ref{537}). Moreover, we know that 
\begin{enumerate}
\item $n_{G, o_G}=|Z_G(\mathbb{F}_q)|$. 
\item There is a constant $C$ depending only on $\deg S$ and on the root system $\Phi(G,T)$ so that 
\[\sum_{H,o}| n_{H,o}|\leq C . \]
\item Suppose that $S\subseteq X(\mathbb{F}_q)$, every $T_v$ is split and $q-1$ is divisible enough (an explicit hypothesis is given in \ref{car2}).  Then every split elliptic coendoscopic group $H$ defined over $\overbar{\mathbb{F}}_q$ is defined over $\mathbb{F}_q$ and the fiber $\widetilde{\pi}_{H,G}^{-1}(o_G)$ is a constant scheme over $\mathbb{F}_q$.
Assume that $\theta_v$ is obtained as the composition of an algebraic character in $X^{*}(T_v)$ with a character $\psi_v: \overbar{\mathbb{F}}_q^{\times}\rightarrow \mathbb{C}^{\times}$.  Then
 $n_{H,o}$ is unchanged if the field $\mathbb{F}_q$ is replaced by $\mathbb{F}_{q^n}$. 
 \end{enumerate}
\end{atheorem}
\begin{aremark}\label{Rmk1}
The equality \eqref{Eq0} is understood in a way that if all the $\widetilde{\mathcal{M}}_{H}^{\xi}(o)$ above have no $\mathbb{F}_q$-points (for example, if they are empty) then the left-hand side equals $0$. This may happen in some extreme cases. If the characteristic does not divide the order of the Weyl group of $G$, we can use the number of $\mathbb{F}_q$-points of corresponding coarse moduli spaces. 
\end{aremark}

We have a result of Theorem \ref{best} and Theorem \ref{connectedc} (see also Corollary \ref{estimate}).
\begin{acoro}\label{esti}
Suppose $\prod_{v\in S}\theta_v|_{Z_G(\mathbb{F}_q)}$ is trivial. 
Assume also that $T_\infty$ is split, $\rho$ is a cuspidal filter and $\deg S>\max\{2-g, 0\}$. Here $g$ is the genus of the curve $X$. 
There is a constant $C$ depending only on the root system of $G$, on $S\otimes \overbar{\mathbb{F}}_q$ and the curve $\overbar{X}$, such that for any $\rho$ that is of Deligne-Lusztig type and is a cuspidal filter, we have 
 \[| \sum_{\pi: \pi_{\rho}\neq 0}m_\pi -   |Z_G(\mathbb{F}_q)||\pi_1(G)|q^{N}| \leq C q^{N-\frac{1}{2}}, \] 
 where $N= \frac{1}{2}((2g-2+\deg S)dim \ggg - \deg S\dim \ttt ).$ In particular, if $q$ is large enough depending only on $S\otimes \overbar{\mathbb{F}}_q$ and on the curve $\overbar{X}$ (but not on $(\theta_v)_{v\in S}$), we have
  \[\sum_{\pi: \pi_{\rho}\neq 0}m_\pi\neq 0. \]
\end{acoro}

\subsection{Some auxiliary results}
Here we gather some results that could be of separate interest. They are either ingredients of the proof or byproducts of the proof.   

\subsubsection{The Hitchin moduli spaces}
The methods used in this paper concern an exploration of the connection between Arthur's non-invariant trace formula and Hitchin moduli spaces. As a result, we have the following results  on the geometry of the Hitchin moduli spaces.

\begin{atheorem}(Theorem \ref{Faltings})\label{connectedc}
Suppose that $\xi$ is in general position. 
Suppose that $|\overbar{S}|> 2-g$. 
For any $o\in \widetilde{\mathcal{R}}_G(\mathbb{F}_q)$, the Deligne-Mumford stack $\widetilde{\mathcal{M}}_G^{\xi}(o)$ is equidimensional and has $|\pi_1(G)|$ connected components. Moreover if every factor of $o$ is regular, then each connected component of $\widetilde{\mathcal{M}}_G^{\xi}(o)$ is geometrically irreducible of dimension $\dim \ggg(2g-2+\deg S)- \dim \ttt\deg S$. 
 \end{atheorem}


\subsubsection{Depth zero Hecke algebras}\label{122}
As remarked in Remark \ref{Rmk1}.2., the condition that each $\theta_v$ is absolutely regular is used in the following theorem. It will be useful for local calculations in applications of trace formulas. 
The case that $T'$ is elliptic is well known. 
\begin{atheorem}(Theorem \ref{commutative}) 
Let $v$ be a place of $F$. 
Let $T'$ be a torus of $G$ defined over $\kappa_v$ and $\theta$ be an absolutely regular character of $T'(\kappa_v)$. 
Let $\rho$ be the inflation to $G(\ooo_v)$ of the Deligne-Lusztig induced representation $\epsilon_{\kappa_v}(T')\epsilon_{\kappa_v}(G) R_{T'}^G(\theta)$ of $G(\kappa_v)$. Then the local Hecke algebra $\mathcal{H}(G(F_v), \rho):= \mathrm{End}(c\text{-}Ind_{G(\mathcal{\ooo}_v)}^{G(F_v)}\rho)$ is commutative. 
\end{atheorem}
This theorem implies that any irreducible smooth representation of $G(F_v)$ that contains $\rho_v$ must contain it with multiplicity one. 

\subsubsection{Trace formulas}\label{123}
In \cite{Yu2}, we have introduced a variant of Arthur's truncation by adding an additional parameter $\xi$. It is a distribution $J^{G,\xi}$. We have given a coarse geometric expansion following semisimple conjugacy classes. This coarse expansion decomposes $J^{G,\xi}$ into a sum of distributions $J_{o}^{G,\xi}$ parametrized by semisimple conjugacy classes. 
Let $\mathcal{I}_{\infty}$ (resp. $\mathfrak{I}_{\infty}$) be the Iwahori subgroup (subalgebra) of $G(\ooo_\infty)$ (resp. $\ggg(\ooo_\infty)$). We justify our construction by the following results. 
\begin{atheorem}(Theorem \ref{expansion})
Suppose that $\xi$ is in general position.
Let $f\in \mathcal{C}_c^{\infty}(G(\AAA))$ be a locally constant function supported in $\mathcal{I}_\infty\times \prod_{v\neq \infty}G(\ooo_v)$, then 
$$J^{G,\xi}_o(f)=0$$
if $o$ can not be represented by a semisimple element $\sigma$ in $T(\mathbb{F}_q)$ or if $o$ is not elliptic.
If $o$ is elliptic and is represented by $\sigma\in T(\mathbb{F}_q)$, suppose that $f$ is $\mathcal{I}_{\infty}\times \prod_{v\neq \infty}G(\ooo_v)$-conjugate invariant, we have 
\[J^{G,\xi}_o(f)=  \frac{\vol({\mathcal{I}_{\infty}} )}{\vol(\mathcal{I}_{\infty,\sigma}) |\pi_0(G_{\sigma})|  } \sum_{w}J^{G^{0}_{\sigma}, w\xi}_{[\sigma]}(f^{w^{-1}}|_{G^{0}_{\sigma}(\mathbb{A})}),      \]
where the sum over $w$ is taken over the set of representatives of $W^{(G_\sigma^{0}, T)}\backslash W$ that sends positive roots in $(G_{\sigma}^0,  T)$ to positive roots, and $f^{w^{-1}}$ is the function such that $f^{w^{-1}}(x)=f({w^{-1}}xw)$. 
\end{atheorem}
\begin{atheorem}(Theorem \ref{expansion'})
Suppose that $\xi$ is in general position. 
Let $f\in \mathcal{C}_c^{\infty}(\ggg(\AAA))$ be a locally constant function supported in $ \mathfrak{I}_\infty\times \prod_{v\neq \infty}G(\ooo_v)$, then 
$$J^{\ggg,\xi}_o(f)=0$$
if $o$ is not the nilpotent class. 
Hence $$ J^{\ggg, \xi}(f)= J_{nil}^{\ggg, \xi}(f). $$
\end{atheorem}

\subsection{Outline of the article}
In Section \ref{alggr}, we do some preparation work on linear algebraic groups over positive characteristics. As the results of this section could be interesting in itself, we do not assume that $G$ is semisimple or split, but we assume that $G$ is defined over a finite field. The main results of this section are Theorem \ref{staconj} and \ref{staconjLie}.

In Section \ref{essuni}, we prove Theorem \ref{expansion} and Theorem \ref{expansion'}. These two theorems are the cores of the applications worked out in this article. The idea is inspired by Theorem \cite[6.2.1]{Chau} of Chaudouard, where similar results are proved for $GL_n$ in the ``coprime" case. For us, the introduction of $\xi$ replaces the ``coprime" condition.

In Section \ref{Hitchin's}, we collect some facts on Hitchin's moduli stacks, mainly from the work of Faltings, Ngô, and Chaudouard-Laumon.  Besides, the main aim of this section is to drop some hypotheses on the divisor $D$ in Ngô's work.

In Section \ref{summul}, we prove the main theorems of this article as presented in this introduction. We express the sum of multiplicities in terms of truncated trace for a well-chosen function and then reduce this trace to truncated ``trace" Lie algebra. The Springer hypothesis is the key to this reduction. 

In Appendix \ref{A}, we give a full description of split elliptic coendoscopic groups. The main result that is needed in this article is that when $q-1$ is divisible enough, then any split elliptic coendoscopic group defined over $\overbar{\mathbb{F}}_q$ is defined over $\mathbb{F}_q$.

\section*{Acknowledgements}
I would like to thank Pierre-Henri Chaudouard, Tamas Hausel and Erez Lapid for many valuable discussions and communications. 
I would like to thank Yang Cao for his help on Lemma \ref{Cao}. I thank Peiyi Cui for her help on the representation theory of $p$-adic groups. 
This project has received funding from the European Union's Horizon 2020 research and innovation programme under the Marie Sk\l odowska-Curie Grant Agreement No. 754411.

\section{Linear algebraic groups over positive charactistic}\label{alggr}
In this section, we assume that $G$ is a reductive group defined over a finite field $\mathbb{F}_q$ if not specified otherwise.  We do not assume $G$ to be split or semisimple for the moment. 

\subsection{Notation}
We use the language of group schemes, so a linear algebraic group $G$ over a ring $k$ is an affine group scheme $G\rightarrow spec(k)$ of finite type. We do not require it to be smooth or connected. Reductive groups are connected smooth group schemes with trivial geometric unipotent radical. Intersection and center are taken in the scheme theoretic sense.

$F, F^a, F^s, \mathbb{F}_q, \overbar{\mathbb{F}}_q$:  Throughout the article, $F$ is a global function field with finite constant field $\mathbb{F}_q$, i.e., a field of rational functions over a smooth, projective and geometrically connected curve $X$ defined over $\mathbb{F}_q$. We fix an algebraic closure ${F}^{a}$ of $F$ and denote by $F^{s}$ be the separable closure of $F$ in $F^{a}$. But we prefer to use  $\overbar{\mathbb{F}}_q$ to denote a fixed algebraic closure of $\mathbb{F}_q$. 

$G^0$, $Z(G)$, $G^{der}$, $G^{sc}$:   For a linear algebraic group $G$, we denote by $G^0$, $Z(G)$, $G^{der}$ and $G^{sc}$ the connected component of the identity of $G$, the center of $G$, and the derived group of $G$, and the simply connected covering of $G^{der}$ respectively. 

$G_x$, $G_X$, $C_G(H)$, $N_G(H)$:   When $G$ is defined over a field $F$ and $x\in G(F)$, and $X\in \ggg(F)$, we denote by $G_{x}$ (resp. $G_X$) the centralizer of $x$ (resp. the adjoint centralizer of $X$) in G. If $H$ is a subgroup of $G$ defined over $F$, we denote by $C_G(H)$ and $N_G(H)$ the centralizer and normalizer of $H$ in $G$ respectively. 

 $\ggg$: We denote by $\ggg$ the Lie algebra of $G$. In general, if a group is denoted by a capital letter, we use its lowercase gothic letter for its Lie algebra. 

From now on, $G$ will always be a reductive group defined over a finite field $\mathbb{F}_q$ if not specified otherwise. For a field $k$ containing $\mathbb{F}_q$, we denote by $G_k$ the base change of $G$ to $k$.

\subsection{Characteristic of the base field and Springer's isomorphisms}\label{chara}
Suppose that $p$ is very good for $G$. We recall that if $G$ is simple, then the characteristic $p$ is said to be very good if $p\neq 2$ when $G$ is one of type $B,C,D$;  $p\neq 2,3$ if $G$ is of one of type $E,F,G$; and $p\neq 2,3,5$ if $G$ is of type $E_8$. For type $A_{n-1}$, we require that $p\nmid n$. In general, a prime is very good for $G$ if it is very good for every simple factor of $G^{sc}_{{F}^a}$. 
For more discussion of ``very good primes", see \cite[I.4]{SS} or \cite[2.1]{McN}.
 
\begin{prop}\label{ch}
Let $\mathcal{U}_G$ be the unipotent variety of $G$ and $\mathcal{N}_\ggg$ the nilpotent variety of $\ggg$. 
There is an isomorphism of varieties defined over $\mathbb{F}_q$:
$$l: \mathcal{U}_G\rightarrow \mathcal{N}_\ggg$$ 
whose derivative at $1$:
$$ \d l_1: \mathcal{N}_\ggg \rightarrow \mathcal{N}_\ggg,   $$
 is the identity morphism. 
\end{prop}
\begin{proof}

If $G$ is semisimple and simply connected, the assertion is a corollary of  \cite[1.8.12]{KV} and \cite[9.1, 9.2, 9.3.4 and 6.3]{BR}. In fact, in this case ($G=G^{sc}$) Kazhdan and Varshavsky have constructed a $G$-equivariant morphism $\Phi: G\rightarrow \ggg$ such that $\Phi(1)=0$ and $\d \Phi_{1}: \ggg\rightarrow \ggg$ is the identity. Their work
extends the same construction for the split case given in \cite[9.3.3]{BR}. Then, the proof of  \cite[9.3.4]{BR} shows that the restriction of $\Phi$ to the unipotent variety is an isomorphism between $\mathcal{U}_G$ and $\mathcal{N}_\ggg$.

Consider the case that $G$ is not necessarily simply connected. The relation 
\begin{equation}\label{lied}\ggg\cong\ggg^{sc}\oplus \zzz. \end{equation}
 shows that $\mathcal{N}_{\ggg^{sc}}$ is $G$-equivariant isomorphic to $\mathcal{N}_{\ggg}$. While the $G$-equivariant covering $G^{sc}\rightarrow G^{der}$ is separable, hence it is \'etale at $1$. Therefore we can apply \cite[6.3]{BR}, which says that the fiber over $1$ of the quotient morphism $G^{sc}\rightarrow G^{sc}/\Ad(G)$ is sent isomorphically to the fiber over $1$ of $G^{der}\rightarrow G^{der}/\Ad(G)$, i.e. $\mathcal{U}_{G^{sc}}\cong \mathcal{U}_{G^{der}}$. Finally, $G^{der}\rightarrow G$ is a closed immersion, and it sends $\mathcal{U}_{G^{der}}(\overbar{\mathbb{F}}_q)$ surjectively  to $\mathcal{U}_{G}(\overbar{\mathbb{F}}_q)$ because the quotient $G/G^{der}$ is a torus hence contains no unipotent element except $1$. Hence the closed immersion $\mathcal{U}_{G^{der}}\rightarrow \mathcal{U}_{G}$ is in fact an isomorphism. The semisimple and simply connected case then extends to the general case. 
\end{proof}

McNinch has observed (\cite[Remark 10]{McN2}) that every $G$-equivariant isomorphism $l: \mathcal{U}\rightarrow \mathcal{N}$ has the property that for each parabolic subgroup of $G$, it induces an isomorphism $$l|_{N_P}: N_P\xrightarrow{\sim}\nnn_P,   $$
and an isomorphism 
$$l|_{M_P}: \mathcal{U}_{M_P}\xrightarrow{\sim}\mathcal{N}_{\mmm_P}. $$
Furthermore, we have the following proposition. 
\begin{prop}\label{222}
Let $P$ be a parabolic subgroup of $G$ defined over $\mathbb{F}_q$. Let $k$ be a field containing $\mathbb{F}_q$, then
for any $u\in \mathcal{U}_{M_P}(k)$, the morphism defined by
$n\mapsto l(un)-l(u)$
from $\mathcal{U}_G$ to $\mathcal{N}_\ggg$ induces an isomorphism of varieties over $k$ from $N_P$ to $\nnn_\ppp$. 
\end{prop}
\begin{proof}
It is essentially the same proof as \cite[Remark 10]{McN2}. Let $\lambda$ be a one parameter subgroup of $G$ such that $P=P(\lambda)$. Then as $\lambda$ has image in $Z_{M}$, for any $k$-algebra $R$
and $u$
 $$\Ad(\lambda(t))(l(un)-l(u))= l(u \lambda(t)n\lambda(t^{-1})) -l(u).$$
 It is clear then $u\mapsto l(un)-l(u)$ has image in $\nnn_{\ppp}(R)$. It is an isomorphism as the same argument shows that the inverse is given by $U\mapsto u^{-1}l^{-1}(U+l(u)).$
\end{proof}


\subsection{Jordan-Chevalley decomposition and conjugacy classes}
In this subsection, I will collect some facts on algebraic groups in positive characteristics. All the results in this subsection are well-known in characteristic zero and are also well-known to experts in algebraic groups over positive characteristics. We provide a proof when it is difficult to locate exact references in literature. 

\subsubsection{Jordan-Chevalley decomposition}
Let $G$ be a connected linear algebraic group defined over a field $k$ of characteristic $p>0$. The Jordan-Chevalley decomposition holds for $G(k^{a})$, while not necessarily for $G(k)$. 
However, if the characteristic of the base field is large enough, it is known that Jordan-Chevalley decomposition holds in $G(k)$ (see \cite[Proposition 48]{McN}). 
Without any restriction on the characteristic, we still have the following result.
\begin{prop}\label{JordanC}
Let $G$ be a linear algebraic group defined over a field $k$ of positive characteristic. Let $x\in G(k)$ be a torsion element, then $x$ admits Jordan-Chevalley decomposition in $G(k)$. 
\end{prop}
\begin{proof}
Let $x$ be of order $p^dm$ with $p\nmid m$. Let $a,b\in \mathbb{Z}$ be integers such that $ap^d+bm=1$.
Then the decomposition $x=x^{ap^d}x^{bm}$ is the Jordan-Chevalley decomposition of $x$ where $x^{ap^d}$ is its semisimple part and $x^{bm}$ is its unipotent part. In fact, this can be checked after the base change to $k^{a}$ where Jordan-Chevalley decomposition exists and is unique. It suffices then to observe that semisimple torsion elements have order prime to $p$, while unipotent elements have order some power of $p$. 
\end{proof}

\begin{coro}[of the proof.]\label{intss}
Let $K$ be a subgroup of $G(k)$ and $x\in K$ be a torsion element. 
Then the semisimple and unipotent parts of $x$ both lie in $K$. 
\end{coro}

We also need a Jordan decomposition for Lie algebra. We restrict ourselves to the case of a global function field $F$ with finite constant field $\mathbb{F}_q$. 
\begin{prop}\label{JordanL}
Let $H$ be a linear algebraic group defined over the finite field $\mathbb{F}_q$ and $\hhh$ be its Lie algebra. 
Let $X\in \hhh(F)$. Then $X$ admits Jordan-Chevalley decomposition in $\hhh(F)$ if $X$ is $H(F_v)$-conjugate to an element in $\hhh(\ooo_v)$ for any completion $F_v$ of $F$ with ring of integers $\ooo_v$. 
\end{prop}
\begin{proof}
Let us choose an
 $\mathbb{F}_q$-embedding of algebraic groups $G\hookrightarrow GL_{n,\mathbb{F}_q}$.
 It induces an embedding of their Lie algebras $\ggg\hookrightarrow \ggg\lll_{n,\mathbb{F}_q}$, which is $G$-equivariant. Since it is an  $\mathbb{F}_q$-embedding, it maps $\ggg(\ooo)$ into $\ggg\lll_n(\ooo)$. Now the characteristic polynomial of the image of $X$ in $\ggg\lll_n(F)$ is separable as its coefficients are in $\mathbb{F}_q=\ooo\cap F$, we deduce that the image of semisimple part and that of unipotent part of $X$ both lie in $\ggg\lll_n(F)$ hence in $\ggg(F)$. 
\end{proof}

\subsubsection{Exact sequences}
Let $$1\longrightarrow N\xrightarrow{\ i \ }H\xrightarrow{\ p \ }Q\longrightarrow 1$$ 
be an exact sequence of linear algebraic groups defined over $k$ (i.e. the morphism $p$ is faithfully flat and $i$ induces an isomorphism from $N$ to the scheme theoretic kernel of $p$). Then
we have an exact sequence of groups: $$1\longrightarrow N(k^{a})\longrightarrow H(k^{a})\longrightarrow Q(k^{a})\longrightarrow 1. $$

\begin{prop}\label{Fpoints}
If, in the above diagram, $N$ and $ H$ are smooth, then we have an exact sequence of groups
$$1\longrightarrow N(k^{s})\longrightarrow H(k^{s})\longrightarrow Q(k^{s})\longrightarrow 1.$$
Hence an exact sequence of pointed sets: 
$$1\rightarrow N(k)\rightarrow H(k)\rightarrow Q(k)\rightarrow H^{1}(k, N)\rightarrow H^{1}(k, H)\rightarrow H^{1}(k, Q). $$ 
If, moreover, $N$ is in the center of $H$, the connecting homomorphism $Q(k)\rightarrow H^{1}(k, N)$ is a group homomorphism. 
\end{prop}
\begin{proof}
The long exact sequence follows from the short one, and the ``moreover" part is a direct cocycle calculation. For the first statement, one only needs to prove the surjectivity of $H(k^{s})\rightarrow Q(k^{s})$. As $H$ and $N$ are smooth, the quotient morphism $H\rightarrow Q$ is also smooth (by a dimension argument of the Lie algebras). 
Hence Zariski locally on the target, it can be factored as the composition $U\xrightarrow{\pi}\mathbb{A}_{V}^{n}\xrightarrow{p}V$ of an \'etale morphism $\pi$ with a projection $p$. Clearly, both morphisms admit lifting of $k^{s}$-points. 
\end{proof}

By Lang's theorem, if $H$ is a linear algebraic group defined over $\mathbb{F}_q$, then $H^{1}(\mathbb{F}_q, H^{0})$ vanishes. Hence the map $H(\mathbb{F}_q)\rightarrow (H/H^{0})(\mathbb{F}_q)$ is surjective. 
We also have the following lemma: 
\begin{lemm}\label{cooll}
Let $H$ be a smooth linear algebraic group defined over $\mathbb{F}_q$.
Let $k$ be a field containing $\mathbb{F}_q$ (not necessarily finite).  
 Then the homomorphism 
$$H(k)\rightarrow (H/H^{0})(k)$$ is surjective. 
\end{lemm}
\begin{proof}
We know that $H/H^{0}$ is a finite \'etale group scheme defined over $\mathbb{F}_q$. Let $l$ be the algebraic closure of $\mathbb{F}_q$ in $k$. Then $(H/H^{0})(l)=(H/H^{0})(k)$. As $H(l)\rightarrow (H/H^{0})(l)$ is surjective, the homomorphism $H(k)\rightarrow (H/H^{0})(k)$ is surjective too.
\end{proof}

\begin{coro}\label{OK}
Under the assumption of Lemma \ref{cooll}, 
 we have \[\ker( H^{1}(k, H^{0})\rightarrow H^{1}(k, H))=\{1\}.\]
\end{coro}
\begin{proof}
This is a corollary of Lemma \ref{cooll} and Proposition \ref{Fpoints}. 
\end{proof}

\subsubsection{Various notions of conjugacy}

\begin{definition}
Let $x$ and $y$ be two semisimple elements in $G(k)$ (resp. in $\ggg(k)$). Let $E|k$ be a field extension. Then we say that $x$ and $y$ are $G(E)$-conjugate if there exists $h\in G(E)$ such that $hxh^{-1}=y$ (resp. $\Ad(h)x=y$). We say that $x$ and $y$ are stably conjugate if there exists $h\in G(k^{s})$ such that $hxh^{-1}=y$ (resp. $\Ad(h)x=y$), and for any $\sigma\in \Gal(k^s|k)$, we have $h^{-1} ( {}^{\sigma} h) \in G_x^{0}(k^{s})$.
\end{definition}
\begin{remark}\normalfont
If $G$ is reductive, $p$ is not a torsion prime (in the sense of \cite[Definition 1.3]{Steinberg}), and $x$ is a semisimple element in $\ggg(k)$, the centralizer $G_x$ is automatically connected (\cite[Theorem 3.14]{Steinberg}). 
Hence in the Lie algebra case, there is no difference between stable conjugation and $F^{s}$-conjugation for semisimple elements if the characteristic is not a torsion prime. 
\end{remark}

Clearly:  stable conjugacy $\implies$ $G(k^{s})$-conjugacy$\implies$ $G(k^{a})$-conjugacy. For reductive groups, we have the following:
\begin{prop}\label{tran}
Let $x, x' \in G(k)$ be two semisimple elements. 
They are $G(k^{a})$-conjugate if and only if  they are $G(k^{s})$-conjugate. The same statement applies to semisimple elements in $\ggg(k)$. 
\end{prop}
\begin{proof}
Let $x$ and $x'$ be two elements in $G(k)$ that are $k^{a}$-conjugate. The transporter $\mathrm{Tran}(x, x')$, defined as an $k$-scheme  by $$\mathrm{Tran}(x, x') = \{h\in G \mid hxh^{-1}=x' \}, $$
is a $G_x$ torsor since $\mathrm{Tran}(x, x')(k^a)\neq \emptyset$. As $x$ is semisimple, $G_x$ is smooth, so is $\mathrm{Tran}(x,x')$. By smoothness,  $\mathrm{Tran}(x, x')(k^a)\neq \emptyset$ $\implies$ $\mathrm{Tran}(x, x')(k^s)\neq \emptyset. $ The Lie algebra case can be proved with the same argument. 
\end{proof}

The following result will be frequently used in this article, sometimes implicitly. 
\begin{prop}\label{smc}
Let $k$ be a field containing $\mathbb{F}_q$ such that $\mathbb{F}_q$ is algebraically closed in $k$.  
Let $x$ and $x'$ be two semisimple elements in $G(\mathbb{F}_{q})$ (or $\ggg(\mathbb{F}_q)$). They are $G(k)$-conjugate if and only if they are $G(\mathbb{F}_{q})$-conjugate. 
\end{prop}
\begin{proof}
Let $\mathrm{Tran}(x, x')$ be the transporter from $x$ to $x'$. It is a $G_x$ torsor since $\mathrm{Tran}(x, x')(k)\neq \emptyset$. 
If $G_x$ is connected, then by Lang's theorem $\mathrm{Tran}(x, x')(\mathbb{F}_q)\neq \emptyset$. In general, let $\mathrm{Tran}(x,x')^{1}$ be the connected component of $\mathrm{Tran}(x,x')$ defined over $\mathbb{F}_q$ that admits a $k$-point. It ensures that $\mathrm{Tran}(x,x')^{1}$ is geometrically connected since $\mathbb{F}_q$ is algebraically closed in $F$. It implies that $\mathrm{Tran}(x,x')^{1}$ is a $G_x^{0}$-torsor. Therefore it has an $\mathbb{F}_q$-point. 
\end{proof}

\begin{prop}

Two semisimple elements $X, Y\in \ggg(k)$ are $G(k^s)$-conjugate if and only if their images under the quotient morphism $\ggg\rightarrow \ggg\sslash G $ coincide. Here $\ggg\sslash G=\Spec(k[\ggg]^{G})$ is the quotient of $\ggg$ by the adjoint action of $G$.  

\end{prop}
\begin{proof}
It is clear that if $X, Y\in \ggg(k)$ are $G(k^s)$-conjugate, then they have the same image in $(\ggg\sslash G)(k)$. 
For the converse, we know that the orbit of a semisimple element is closed, and the quotient morphism separates closed orbits, hence $X$ and $Y$ are $G(k^a)$-conjugate, thus $G(k^s)$-conjugate by Proposition \ref{tran}. 
\end{proof}

\subsection{A theorem on conjugacy classes}\label{lgp}
In this subsection, $G$ is a reductive group defined over a finite field $\mathbb{F}_q$, $F$ is the function field of a projective smooth and geometrically connected curve defined over $\mathbb{F}_q$. We do not require $G$ to be split as the results in this subsection could have other interests than its application in this article.

This subsection aims to prove Theorem \ref{staconj} and Theorem \ref{staconjLie} below. It will be used in a later section to analyse the geometric side of the trace formula.  Note that the first part of Theorem \ref{staconj} in the case that $G$ is semisimple and simply connected has already been proved in \cite[Proposition 8.3]{Gross}.

\begin{theorem}\label{staconj}
Let $x$ be a semisimple torsion element in $G(F)$. Then the stable conjugacy class of $x$ contains a point in $G(\mathbb{F}_q)$. 

If moreover, that for every place $v$ of $F$, $x$ is $G(F_v)$-conjugate to an element in $G(\mathcal{O}_v)$, where $\mathcal{O}_v$ is the ring of integers of $F_v$. Then $x$ is $G(F)$-conjugate to an element in $G(\mathbb{F}_q)$. 
\end{theorem}

  \begin{theorem}\label{staconjLie}
Suppose that the characteristic $p$ is very good for $G$.  
Let $X\in \ggg(F)$ be a semisimple element. If for every place $v$ of $F$, $X$ is $G(F_v)$-conjugate to an element in $\ggg(\mathcal{O}_v)$. Then $X$ is $G(F)$-conjugate to an element in $\ggg(\mathbb{F}_q)$. 
\end{theorem}

The proofs will be given in \ref{stablyconjugate} and \ref{Liestably}. 

\subsubsection{A result on the vanishing of $\ker^{1}$}
The proof of the following result is communicated to me by Yang Cao. 
\begin{lemm}\label{Cao}
Let us denote $\Gamma$ for $\Gal(F^{s}|F)$, $\Gamma_v$ for $\Gal(F^{s}_v|F_v)$ for every place $v$ of $F$. Let $M$ be a discrete group with a continuous $\Gamma$-action, that splits over $F\otimes_{\mathbb{F}_q}\mathbb{F}_{q^n}$ for some $n$ (i.e. $\Gal(F^{s}|F\otimes_{\mathbb{F}_q}\mathbb{F}_{q^n})$ acts trivially). 
Then $$\ker^{1}(F, M):= \ker(H^{1}(\Gamma, M)\rightarrow \prod_v  H^{1}(\Gamma_v, M)) $$ is trivial.
\end{lemm}

\begin{proof}
Let  $N$ be $\Gal(F^s| F\otimes_{\mathbb{F}_q}\mathbb{F}_{q^n})$ which acts trivially on $M$. 
By restriction-inflation exact sequence, one has the following commutative diagram (of pointed sets) with exact rows:
$$\begin{CD}
0@>>> H^{1}(\Gamma/N, M)@>>> H^{1}(\Gamma, M )@>>> H^{1}(N, M)^{\Gamma/N}  \\
@.@V\iota_1VV   @V\iota_2VV@V\iota_3VV  \\
0@>>>\prod_{v}H^{1}(\Gamma_v/\Gamma_v\cap N, M)@>>>\prod_{v}H^{1}(\Gamma_v, M )@>>>\prod_{v}H^{1}(N\cap \Gamma_v, M )^{\Gamma_v/\Gamma_v\cap N}
\end{CD}.$$

Let $v$ be a place of $F$ of degree prime to $n$, then $\Gamma/N\cong \Gamma_v/\Gamma_v\cap N$. It implies $\iota_1$ is injective. 

Let $\varphi\in\ker(\iota_3) $. Then $\varphi$ is represented by an element of $\Hom(N, M)$. Since $\varphi$ is  continuous, its kernel is an open normal subgroup $N'$ of $N$, that corresponds to a finite Galois extension $F_1$ of $F\otimes_{\mathbb{F}_q}\mathbb{F}_{q^n}$. The field $F_{1}$ is in fact Galois over  $F$ because $\varphi$ is fixed by $\Gamma/N$. It implies that $N'$ is normal in $\Gamma$. 
For each place $v$ of $F$, the implicitly chosen embedding $\Gamma_v\hookrightarrow \Gamma$ specifies a place $w$ of $F\otimes_{\mathbb{F}_q}\mathbb{F}_{q^n}$ dividing $v$. 
Note that $\varphi$ is trivial on $\Gamma_v\cap N$ for the place $v$ if and only if the place $w$ is totally split for the extension $F_1|F\otimes_{\mathbb{F}_q}\mathbb{F}_{q^n}$. As $F_1|F$ is Galois, it implies that every place of $F\otimes_{\mathbb{F}_q}\mathbb{F}_{q^n}$ dividing $v$ is totally split for the extension $F_1|F\otimes_{\mathbb{F}_q}\mathbb{F}_{q^n}$. 
It forces $F_1=F\otimes_{\mathbb{F}_q}\mathbb{F}_{q^n}$ by Chebotarev density theorem. Thus $\ker(\iota_3)$ is trivial.

By the (non-abelian) five lemma, $\ker(\iota_2)$ is trivial. 
\end{proof}

When $H$ is connected, the following result is well-known to experts. We extend it to non-connected cases. 
\begin{prop}\label{Hasse}
Let $H$ be a linear algebraic group (not necessarily connected) defined over a finite field $\mathbb{F}_q$ such that $H^{0}$ is reductive. Then 
$$\ker^{1}(F, H):= \ker(H^{1}(F, H)\rightarrow \prod_v  H^{1}(F_v, H)) $$ is trivial. 
\end{prop}
\begin{proof}

By the duality theorem of Kottwitz (\cite[(4.2.2) and Remark4.4]{Ko2}), $\ker^{1}(F, H^0)$ is dual to $\ker^{1}(F, Z_{\hat{H^0}}(\mathbb{C}))$ where $\hat{H^0}$ is the Langlands dual group of $H^0$. Recall that $$Z_{\hat{H^0}}(\mathbb{C})\cong X^{*}(H^0)_{F^{s}}\otimes_{\mathbb{Z}}\mathbb{C}^{\times},$$
the action of $\Gal(F^{s}|F)$ on  $Z_{\hat{H^0}}(\mathbb{C})$ is induced from  its action on $X^{*}(H^0)_{F^{s}}$. 
Note that this result of Kottwitz is proved upon the validity of the Hasse principle (vanishing of $\ker^{1}$) for semisimple and simply connected reductive groups. Over a function field, the Hasse principle is proved by Harder (1974). The proof of Kottwitz described in \cite[Remark 4.4]{Ko2}, stated for a number field, works for a function field too
(see also \cite[Theorem 2.6]{Thang}). 

By Proposition \ref{Fpoints} and Corollary \ref{OK}, we have the following commutative diagram with exact rows:
$$\begin{CD}
1@>>> H^{1}(F, H^{0})@>>> H^{1}(F, H )@>>> H^{1}(F, H/H^{0})\\
@.@Vi_1VV   @Vi_2VV@Vi_3VV  \\
 1 @>>>\prod_{v}H^{1}(F_v, H^{0})@>>>\prod_{v}H^{1}(F, H )@>>>\prod_{v}H^{1}(F_v, H/H^{0})
\end{CD}.$$
As $H$ is defined over the finite field $\mathbb{F}_q$, it splits over $F\otimes_{\mathbb{F}_q}\mathbb{F}_{q^n}$ for some $n$. Hence both $H/H^{0}(F^{s})$ and $Z_{\hat{H^0}}(\mathbb{C})$ splits over $F\otimes_{\mathbb{F}_q}\mathbb{F}_{q^n}$. 
By Lemma \ref{Cao}, $\ker(i_3)$  and $\ker^{1}(F, Z_{\hat{H^0}}(\mathbb{C}))$ are trivial. 
After Kottwitz duality, $\ker(i_1)$ is then trivial. Hence $\ker(i_2)$ is trivial by the (non-abelian) five lemma. \end{proof}

\subsubsection{Proof of Theorem \ref{staconj}.}\label{stablyconjugate}
1.
We first show that there is an element $z\in G(\mathbb{F}_q)$ that is stably conjugate to $x$. 

Let $T$ be a maximally split maximal torus of $G$ defined over $\mathbb{F}_q$ (i.e., one contained in a Borel subgroup). 
The semisimple element $x$ is contained in a maximal torus of $G$ defined over $F^s$. Then $x$ is $G(F^s)$-conjugate to an element in $T(F^s)$. 
Note that torsion elements in $F^s$ coincide with torsion elements in $\bar{\mathbb{F}}_{q}$. The element $x$ is therefore $G(F^s)$-conjugate to an element in $y\in T(\bar{\mathbb{F}}_{q})$. 

Let $\tau\in \Gal(F^{s}| F)$ be an element that is sent to the Frobenius element via the projection $$\Gal(F^{s}| F)\rightarrow \Gal( F\otimes \overbar{\mathbb{F}}_q | F)\cong  \Gal(\overbar{\mathbb{F}}_q|\mathbb{F}_q). $$
Since $x$ is fixed by $\tau$, we see that ${}^{\tau}y$ is $G(F^s)$-conjugate to $y$. Hence they are $G(\overbar{\mathbb{F}}_q)$-conjugate. It shows that the $G({\overbar{\mathbb{F}}_q})$-conjugacy class of $y$ in $G({\overbar{\mathbb{F}}_q})$ is stable under the Frobenius action. By Lang's theorem, the $G({\overbar{\mathbb{F}}_q})$-conjugacy class of $y$ contains an element in $G(\mathbb{F}_q)$ so we can assume $y\in G(\mathbb{F}_q)$. 

By our assumption, there is a $g\in G(F^s)$ such that $y=gxg^{-1}$. Therefore ${}^{\tau} gg^{-1}\in G_{y}(F^s)$. Note that $G_y$ is defined over $\mathbb{F}_q$, hence 
${}^{\tau} gg^{-1}\in G_{y, \bar{\mathbb{F}}_q}^{1}(F^s)$ for some connected component $G_{y, \bar{\mathbb{F}}_q}^{1}$ of $G_{y, \bar{\mathbb{F}}_q}$. 
By Lang's theorem, there is an $h\in G(\overbar{\mathbb{F}}_q)$ such that $^{\tau} hh^{-1}\in G_{y,\overbar{\mathbb{F}}_q}^{1}(\overbar{\mathbb{F}}_q)$. Let $z:= h^{-1}yh$, one can verify that  $z\in G(\mathbb{F}_q)$ and $z$ is stably conjugate to $x$. 

2. Now, we prove that $z$ is in fact $G(F)$-conjugate to $x$ under our further hypothesis. 

Take an element $z\in G(\mathbb{F}_q)$ that is stably conjugate to $x$. In fact, by \cite[\S3]{Ko1}, the $G(F)$-conjugacy classes within the $G(F^s)$-conjugacy class of $z$ are parametrized by $$\ker(H^{1}(F, G_z)\rightarrow H^{1}(F, G)),$$ so that the trivial cohomological class corresponds to the $G(F)$-conjugacy class of $z$. 
Let \[\delta_{x}\in \ker(H^{1}(F, G_z)\rightarrow H^{1}(F, G))\]  
be the cohomological class corresponding to the conjugacy class of $x$ inside the $G(F^s)$-conjugacy class of $z\in G(F)$. 

For any completion $F_v$ of $F$, $x$ is $F_v$-conjugate to an element in $G(\mathcal{O}_v)$. Then after another theorem of Kottwitz \cite[Proposition 7.1]{Ko3}, we know that $z$ and $x$ are in fact $G(F_v)$-conjugate. Therefore the localization of $\delta_x$ at a place $v$ is the trivial element of $H^{1}(F_v, G_y)$ for any $v$. We conclude that $\delta_x\in \ker^{1}(F, G_z)$. 
Since $z\in G(\mathbb{F}_q)$, $G_{z}$ is defined over $\mathbb{F}_q$. By Theorem \ref{Hasse}, $\delta_x$ is trivial. That is, $x$ and $z$ are indeed $G(F)$-conjugate. 

\subsubsection{Proof of Theorem \ref{staconjLie}}\label{Liestably}
Two semisimple elements in $\ggg(F)$ are stably conjugate if they have the same image in $(\ggg\sslash G)(F)$ via the map induced by the quotient morphism $\ggg\rightarrow \ggg\sslash G$. By our hypothesis, the image of $x$ lies in $(\ggg\sslash G)(\ooo_v)\cap (\ggg\sslash G)(F)$ for every $v$, hence it is in $(\ggg\sslash G)(\mathbb{F}_q)$. As every $\Gal(\overbar{\mathbb{F}}_q| \mathbb{F}_q)$-stable conjugacy class in     $\ggg(\overbar{\mathbb{F}}_q)$ contains an $\mathbb{F}_q$-rational point, the map $\ggg(\mathbb{F}_q)\rightarrow (\ggg\sslash G)(\mathbb{F}_q)$ is surjective, hence $x$ is stably to an element $y\in \ggg(\mathbb{F}_q)$.

The rest of the proof is the same as step 2. of the proof of Theorem \ref{staconj}. 

\section{Unipotent/nilpotent reduction of the trace formulas}\label{essuni}
In \cite{Yu2}, we have introduced a variant of Arthur's trace formula coordinating with the stability parameter $\xi$ used in this article and have established some of its basic properties. The purpose of this section is to justify the need for this variant version by proving Theorem \ref{expansion} and  \ref{expansion'}. Although they apply only to a small class of test functions, these results are key to the main results of this article. 
\subsection{Notation}
Note that the statements in Arthur's trace formula involve a large number of notations, and these notations are also part of the theory. We refer the reader to the first chapter of \cite{LabWal}(more precise references will be given in the text) for understanding the notation. For the reader's convenience, we gather some notations frequently used in this article.  
\subsubsection{}
As in the introduction, $F$ is the global function field of a smooth projective and geometrically connected curve defined over $\mathbb{F}_q$. Let $|X|$ be the set of places of $F$. For each $v\in |X|$, let $F_v$ be the completion of $F$ at $v$, $\ooo_v$ be the ring of integers of $F_v$, $\wp_v$ its maximal idea and $\kappa_v$ its residue field. Let $\AAA$ be the ring of ad\`eles of $F$; $\ooo$ the ring of integral ad\`eles. 

Let $G$ be a split semisimple algebraic group defined over $\mathbb{F}_q$. Let us fix a Borel subgroup $B$ defined over $\mathbb{F}_q$, a maximal torus $T$ defined over $\mathbb{F}_q$ contained in $B$. We denote by $W=W^{(G,T)}=N_G(T)/C_G(T)$ be the Weyl group of $G$ with respect to $T$.
A parabolic subgroup of $G$ containing $B$ (resp. $T$) is called standard (resp. semistandard).
A Levi subgroup of a parabolic subgroup is often called a Levi subgroup of $G$.

\subsubsection{}
For any algebraic group $P$ defined over $F$, let $X^{*}(P)=X^*(P)_F$ be the group of rational characters and \[X_*(P)=\Hom(X^*(P), \mathbb{Z}).\] 
If $P$ is a torus, $X_*(P)$ is identified with the group of rational cocharacters of $P$. Let  
\[\ago_P:=\Hom(X^{*}(P), \mathbb{Q}).\] 
When $P$ is a semistandard parabolic subgroup of $G$ defined over $F$, the constructions above depend only on the semistandard Levi factor $M_P$ of $P$. Therefore we can denote $\ago_P$ by $\ago_{M_P}$. 

Let $M$ be a semistandard Levi-subgroup of $G$ and 
 $\mathcal{F}(M)$ be the set of semistandard parabolic subgroups of $G$ containing $M$.
For each $P\in \mathcal{F}(M)$, let $A_P$ be the central torus of $P$.  The restriction map identifies $X^{*}(P)$ with a sub-lattice in $X^{*}(A_P)$, both of which are viewed as lattices in the dual space of $\ago_P$.
The restriction $X^{*}(M)\rightarrow X^{*}(A_P)$ allows us to view $\ago_P$ as a linear subspace of $\ago_M$, and we have a projection induced from the morphism $X^{*}(P)\rightarrow X^{*}(M)$, whose kernel is denoted by $\ago_M^P$. Therefore we have a decomposition of linear space: 
\[\ago_M=\ago_M^P \oplus \ago_P. \] 
Moreover, we have the following notation:
\begin{enumerate}
\item[$\bullet$] $\Delta_B$ the set of simple roots of the root system $\subseteq \Phi(G, A)$; $\Delta_B^{\vee}$ the set of simple coroots. 
\item[$\bullet$] $\hat{\Delta}_B$ the set of fundamental weights, i.e. the dual base of $\Delta_B^{\vee}$; 
\item[$\bullet$] $\Delta_B^P=\Delta_B\cap \Phi(M_P,A);$ $(\Delta_B^{P})^{\vee}$ the set of corresponding coroots; \item[$\bullet$] $\Delta_P^Q=\{\alpha|_{A_P}\mid \alpha \in \Delta_B^Q-\Delta_B^{P}\}$, viewed as a set of linear functions on $\ago_B$ via the projection $\ago_B\rightarrow \ago_P^Q$;
\item[$\bullet$] $\htau_P$ the characteristic function of $H\in \ago_B$ such that $\langle\varpi,H\rangle>0$ for all $\varpi\in \hat{\Delta}_P$; 
\end{enumerate}

\subsubsection{Harish-Chandra's map}
For each parabolic subgroup $P$ of $G$ defined over $F$, we define the Harish-Chandra's map $H_P: P(\AAA)\rightarrow \ago_{P}$ by requiring   
\begin{equation}\langle\alpha,  H_P(p)\rangle= -\sum_{v} \deg v \deg (\alpha(p_v)),  \quad \forall \alpha\in X^{*}(P)_F, p=(p_v)_{v\in |X|}\in P(\AAA). \end{equation}
Here 
Using Iwasawa decomposition we may extend $H_P$ to be a function over $G(\AAA)$ by asking it to be $G(\ooo)$-right invariant, i.e. $H_P(g)=H_P(p)$ for $g=pk$ with $p\in P(\mathbb{A})$ and $k\in G(\ooo)$. 

\subsubsection{Haar measures}
The Haar measures are normalized in the following way.
For any place $v$ of $F$ and 
for a Lie algebra $\ggg$ defined over $\mathbb{F}_q$, the volumes of the sets $\ggg(\ooo_v)$ and $\ggg(\ooo)$ are normalized to be $1$; the volume of $G(\ooo_v)$ and $G(\ooo)$ are normalized to be $1$; for every semistandard parabolic subgroup $P$ of $G$ defined over $\mathbb{F}_q$, the measure on $N_P(\AAA)$ is normalized so that $\vol(N_P(F)\backslash N_P(\AAA))=1$.

\subsection{Geometric side of the trace formulas}
\subsubsection{}\label{recalltrace}
First, we briefly recall the constructions of the trace formulas in the following. More details can be found in \cite{Yu2}.

We define two elements $\gamma$ and $\gamma'$ in $G(F)$ (resp. in $\ggg(F)$) to be equivalent if either both of them admit Jordan-Chevalley decompositions and their semisimple parts $\gamma_s$ and $\gamma'_{s}$ are $G(F)$-conjugate, or neither of them admits Jordan-Chevalley decomposition. Let \[\mathcal{E}=\mathcal{E}(G(F)) \text{ or }\mathcal{E}(\ggg(F)) \] be the set of equivalent classes for either the group case or the Lie algebra case if it does not cause confusion.

Let us recall the construction of  \cite[Proposition 9.1]{Yu2} for a semisimple split algebraic group, where we have given a relatively simpler expression for our variant of Arthur's truncated trace.

Let $o\in \mathcal{E}$ be an equivalence class. 
Let $f\in \mathcal{C}_c^{\infty}(G(\mathbb{A}))$ be a smooth complex function with compact support over $G(\AAA)$. 
We take a standard parabolic subgroup $Q$ of $G$, and define for every $x\in G(\AAA)$: 
\begin{equation}
j_{Q,o}(x)=\sum_{\gamma\in M_Q(F)\cap o}\sum_{n\in N_P(F)} f(x^{-1}\gamma n x).  
\end{equation}
In the Lie algebra case, a similar definition is given by analogy. Let $f\in \mathcal{C}_c^{\infty}(\ggg(\mathbb{A}))$, we define a function $\mathfrak{j}_Q(x)$ for any $x\in G(\AAA)$ by:
 \begin{equation}\mathfrak{j}_{Q,o}(x)= \sum_{\Gamma\in {\mmm_Q}(F)\cap o}\sum_{U\in \nnn_Q(F)} f( \mathrm{Ad}(x^{-1})(\Gamma+ U)). \end{equation}

Suppose there is a place $\infty$ of $F$ of degree $1$. The root datum of $G$ over $F$ coincides with that of $G$ over $F_{\infty}$ since $G$ is split, so the Weyl group of $(G, T)$ over $F$ is identified with that of $(G_{F_{\infty}}, T_{F_\infty})$. For any $x\in  G(\mathbb{A})$, we have an Iwasawa decomposition: $x = pk$, with $p\in B(\AAA)$  and $k\in G(\mathcal{O})$. The element  $k$ is unique up to left translation by $B(\AAA)\cap G(\mathcal{O}) = B(\mathcal{O})$, hence determines an element in $$ \mathcal{I}_\infty \backslash G(\mathcal{O}_{\infty})/ \mathcal{I}_\infty  \cong B({\kappa}_\infty)\backslash G({\kappa}_\infty)/B({\kappa}_\infty)\cong W.$$ where $\mathcal{I}_\infty$ is the Iwahori subgroup of $G$ with respect to $B$. 
We denote this Weyl element by $s_{x}$. 
For any vector $\xi \in \ago_{B}$, 
we use a variant of Arthur's truncated kernel given by:
\begin{equation} {j}_o^{\xi}(x) =\sum_{Q\in \mathcal{P}(B)}(-1)^{\dim\ago_Q^G}\sum_{\delta\in Q(F)\backslash G(F)} \htau_{Q}\left(H_0(\delta x) +  s_{\delta x}\xi   \right) {j}_{Q,o}(\delta x), \end{equation}
for any $x\in G(\AAA)$.  
In the Lie algebra case, it is given by an analogue formula:
\begin{equation}\mathfrak{j}_o^{\xi}(x)=\sum_{Q\in \mathcal{P}(B)}(-1)^{\dim\ago_Q^G}\sum_{\delta\in Q(F)\backslash G(F)} \htau_{Q}\left(H_0(\delta x) +  s_{\delta x}\xi   \right) \mathfrak{j}_{Q,o}(\delta x), \end{equation}
for any $x\in G(\AAA)$.

We have shown in \cite[Proposition 9.1]{Yu2} that for each $o\in \mathcal{E}$, the function $j_o^{\xi}(x)$ in $x\in G(F)\backslash G(\AAA)$ has compact support. We define \[J_o^{G, \xi}(f)= \int_{G(F)\backslash G(\AAA)}{j}_o^{\xi}(x) \d x.   \] 
There are only finitely many $o\in \mathcal{E}$ such that $J_o^{G, \xi}(f)$ is non-zero and we set
\[ J^{G, \xi}(f)=\sum_{o\in \mathcal{E}} J_o^{G, \xi}(f) .  \]

\subsubsection{}
We recall the definition of elliptic elements and elliptic classes.
\begin{definition}
We say that an equivalent class $o\in \mathcal{E}$ ($\mathcal{E}_G$ or $\mathcal{E}_{\ggg}$) is elliptic, if $o$ contains a semisimple element $\sigma$ (hence admits Jordan-Chevalley decomposition) so that the torus $Z_{G_\sigma^{0}}^{0}$ is anisotropic. 
 \end{definition}
 
The definition below is taken from \cite[6.1.3]{C-L1}. 
\begin{definition}\label{gepo}
We say that $\xi\in \ago_T$ is in general position if for every $P\in \mathcal{F}(T)$, a semistandard parabolic subgroup, such that $P\subsetneq G$, the projection of $\xi$ to $\ago_P$ under the decomposition $\ago_T = \ago_T^P\oplus \ago_P$ does not belong to the lattice $X_{*}(P)$. 
\end{definition}

 We have the following theorem. It is inspired by a theorem for $GL_n$ of Chaudouard \cite[6.2.1]{Chau}. We observe that Arthur has found results of the same nature which are more general but involve more terms (\cite[Lemma 6.2]{Ageom}). 

\begin{theorem}\label{expansion}
Suppose that $\xi$ is in general position defined in Definition \ref{gepo}.
Let $f\in \mathcal{C}_c^{\infty}(G(\AAA))$ be a smooth function supported in $\mathcal{I}_\infty\times \prod_{v\neq \infty}G(\ooo_v)$, then if $o$ can not be represented by a semisimple element $\sigma$ in $T(\mathbb{F}_q)$ or if $o$ is not elliptic, we have
$$J^{G,\xi}_o(f)=0.$$
If $o$ is elliptic and is represented by $\sigma\in T(\mathbb{F}_q)$, suppose that $f$ is $\mathcal{I}_{\infty}\times \prod_{v\neq \infty}G(\ooo_v)$-conjugate invariant, we have 
\[J^{G,\xi}_o(f)=  \frac{\vol({\mathcal{I}_{\infty}} )}{\vol(\mathcal{I}_{\infty,\sigma}) |\pi_0(G_{\sigma})|  } \sum_{w}J^{G^{0}_{\sigma}, w\xi}_{[\sigma]}(f^{w^{-1}}|_{G^{0}_{\sigma}(\mathbb{A})}),      \]
where the sum over $w$ is taken over the set of representatives of $W^{(G_\sigma^{0}, T)}\backslash W$ that send positive roots in $(G_{\sigma}^0,  T)$ to positive roots, and $f^{w^{-1}}$ is the function such that $f^{w^{-1}}(x)=f({w^{-1}}xw)$. 
\end{theorem}
Note that any function $f\in \mathcal{C}_c^{\infty}(G(\AAA))$ can be made to be $\mathcal{I}_{\infty}\times \prod_{v\neq\infty}G(\ooo_v)$-conjugate invariant by taking its average which does not affect the value $J^{G,\xi}_o(f)$. 

In the Lie algebra case, when the characteristic is good, the adjoint centralizer $G_s$ of an element in $s\in \ggg(\mathbb{F}_q)$ is a twisted Levi subgroup. It means that it becomes a Levi subgroup after a field extension. In fact, it is clear that the root system of $(G_s)_{\overbar{\mathbb{F}}_q}$ as a subroot system of that of $G$ is the intersection of a vector space with the root system of $G$ when the characteristic is good (\cite[VI, 1, 7, Proposition 24]{Bourbaki}), so $(G_s)_{\overbar{\mathbb{F}}_q}$ is a Levi subgroup.
Therefore, there are no elliptic elements in $\ttt(\mathbb{F}_q)$ as $T$ is split. Similar arguments to the group case then show the following theorem. Let  $\mathfrak{I}_{\infty}$ be the  the Iwahori subalgebra of $\ggg(\mathcal{O}_\infty)$ consisting of elements whose reduction modulo $\wp_\infty$ belongs to  $\bbb(\kappa_\infty)$.  

\begin{theorem}\label{expansion'}
Suppose that $\xi$ is in general position defined in Definition \ref{gepo}. 
Let $f\in \mathcal{C}_c^{\infty}(\ggg(\AAA))$ be a smooth function supported in $ \mathfrak{I}_\infty\times \prod_{v\neq \infty}G(\ooo_v)$, then if $o$ is not the nilpotent class, we have
$$J^{\ggg,\xi}_o(f)=0. $$
Hence $$ J^{\ggg, \xi}(f)= J_{nil}^{\ggg, \xi}(f). $$
\end{theorem}

We concentrate the rest of this subsection on the proof of these theorems. 
We only deal with the group case. The Lie algebra case will follow from the parallel arguments, and an indication will be provided if necessary modifications are needed.  

\subsection{A first reduction after Arthur}\label{reAr}
\subsubsection{}
Now we follow Arthur \cite[Section 3]{Ageom} to reduce 
\begin{equation}
j^{\xi}_o(x)=\sum_{Q\in\mathcal{P}(B)}\sum_{\delta\in Q(F)\backslash G(F)}(-1)^{\dim\ago_Q^G}\htau_Q(H_0(\delta x)+s_{\delta x}\xi )j_{Q,o}(\delta x)
\end{equation}
 into a nicer expression. The reader will notice that our arguments are almost identical to Arthur's, although we deal with different functions. For the reader's convenience, we elaborate in several places. Experts familiar with Arthur's work can skip this subsection. 

\subsubsection{}
Let $o\in \mathcal{E}$ be an equivalent class that admits the Jordan-Chevalley decomposition. We
 fix a semisimple element $\sigma\in o$. Up to a conjugation, we assume that $\sigma$ belongs to $M_{P_1}(F)$ for a fixed standard parabolic subgroup $P_1$ of $G$ so that it belongs to no proper parabolic subgroup of $M_{P_1}$. We shall write $M_1=M_{P_1}$, $A_1=A_{M_1}$ and $\ago_1=\ago_{M_1}$.

Note that $A_1$ is a maximal split torus defined over $F$ of $G_{\sigma}^{0}$. In fact, if $A'$ is an $F$-split torus of $G_{\sigma}^{0}$ such that $A'\supseteq A_1$, then $C_G(A')$ is a Levi subgroup of $G$ defined over $F$ included in $M_1=C_G(A_1)$ that contains $\sigma$. 
By the minimality of $M_1$, we have $A_1=A'$.  Therefore $G_{\sigma}^{0}$ is a connected reductive group with a minimal parabolic subgroup $P_{1\sigma}^0$ and minimal Levi subgroup $M_{1\sigma}^0$. We choose $P_{1\sigma}^0$ for the minimal standard parabolic subgroup of $G_{1\sigma}^0$.

Given a standard parabolic subgroup $Q$ of $G$, let us look at $j_{Q,o}$.

For any $\gamma\in M_Q(F)\cap o$, and let $\gamma=\gamma_s\gamma_u$ be its Jordan-Chevalley decomposition with $\gamma_s$ semisimple and $\gamma_u$ unipotent. 
Since $A_Q\subseteq M_{Q,\gamma_s}^0$ and $M_{Q,\gamma_s}^0$ is the centralizer of $A_Q$ in $G_{Q,\gamma_s}^0$, $M_{Q,\gamma_s}^0$ is a Levi subgroup of $G_{Q,\gamma_s}^0$. As $\gamma_s$ is $G(F)$-conjugate to $\sigma$, $\gamma_s$ commutes with a split torus that is a $G(F)$-conjugate of $A_1$. We can choose the torus to lie in $M_{Q,\gamma_s}^0$. The torus is, in turn, $M_Q(F)$-conjugation to the torus $A_{Q_1}$, for some standard parabolic subgroup $Q_1$ of $G$ associated to $P_1$ and contained in $Q$.
Equivalently there is a ${\mu}\in M_Q(F)$ such that ${\mu}\gamma_s{\mu}^{-1}\in M_{Q_1}(F)=C_G(A_{Q_1})(F)$. 

Let $h\in G(F)$ be an element that conjugates $\sigma$ to ${\mu}\gamma_s {\mu}^{-1}$. Since $A_1$ and $h^{-1} a_{Q_1} h$ are maximal $F$-split tori   
in $G_{\sigma}^0$, modifying $h$ from right by an element of $G_{\sigma}^0(F)$, 
we may suppose that $A_{1}$ is also conjugate to $A_{Q_1}$ by $h$. 
Let ${s}\in W(\ago_{1}, \ago_{Q_1})$ be the element that is represented by $h$. Then $h{w}_{s}^{-1}\in C_G(A_{Q_1})(F)\subseteq M_Q(F)$. 

Thus we can write 
\begin{equation}\label{var-ch}\gamma= \mu^{-1}w_s\sigma uw_s^{-1}\mu \end{equation}
for some  $s\in W(\ago_{1}, \ago_{Q_1})$, $\mu\in M_Q(F)$ and $u\in \mathcal{U}_{G_\sigma}(F)\cap w_s^{-1}M_Q(F)w_s $, where $Q_1$ is a standard parabolic subgroup of $G$ that is associated to $P_1$. The Weyl element $s$ is uniquely determined up to multiplication from left by a Weyl element of $M_Q$ (this may change $Q_1$), and multiplication from the right by elements of the Weyl group of $(G_{\sigma}, A_1)$ (instead of that of $(G_{\sigma}^0, A_1)$). Let $W(\ago_1, Q)$ be the set of double cosets in $W^Q\backslash W/W^{P_1}$ (equivalently, the set of cosets in $W^Q\backslash W$) that can be represented by elements $s\in W$ such that $$s(\ago_1)\supseteq \ago_Q. $$ 
Thus $s$ is  well-defined and uniquely determined by $\sigma$ and $\gamma$ in the quotient $W(\ago_1, Q)/W^{(G_\sigma, A_1)}$. 
Once a representative of $s$ is fixed, $\mu$ is uniquely determined modulo $M_Q(F)\cap w_sG_{\sigma}(F)w_s^{-1}$. The element $u$ is uniquely determined by $\mu$ and $w_s$.  
To fix representatives of $s$, the following result is used in Arthur's arguments, although it has not been written down explicitly. 
\begin{lemm}\label{WaQG}
The inclusion $N_{G_\sigma}(A_1)\rightarrow N_G(A_1)$ induces an injective homomorphism
 \[W^{(G_\sigma, A_1)}\longrightarrow W(\ago_1, \ago_1). \]
Let $W(\ago_1, Q; G_\sigma^0)$ be the union over all standard parabolic subgroups $Q_1\subseteq Q$ of those elements
$$s\in W(\ago_{1}, \ago_{Q_1})$$
such that $s^{-1}\alpha$ is positive for every  root $\alpha\in \Delta_{Q_1}^{Q}$ and $s\beta$ is  positive for every positive root $\beta$  of $(G_\sigma^0, A_1)$ (viewed as a root of $(G,A_1)$). Then the natural map $W(\ago_1; Q, G_\sigma^{0})\rightarrow W(\ago_1, Q)/W^{(G_\sigma^{0}, A_1)}$ is bijective. 
\end{lemm}
\begin{proof}
For the first statement, note that we have $N_G(A_1)=N_G(M_1)$ and each coset of the quotient $N_G(M_1)/M_1$ can be represented by a Weyl element of $(G,T)$, hence  $N_G(M_1)/M_1\cong W(\ago_1, \ago_1)$. The injectivity is clear.

Let us prove the last statement. 

Injectivity is clear. In fact, the map $W(\ago_1, Q; G_{\sigma}^0)\rightarrow W(\ago_1, Q)$ is injective by \cite[1.3.7]{LabWal}, the result follows from the fact that the only element in a Weyl group preserving the set of positive roots is the identity.

Now we prove the surjectivity  of the map $ W(\ago_1, Q; G_\sigma^0)\rightarrow W(\ago_1, Q)/W^{(G^{0}_\sigma, A_1)}$. Let $s\in W(\ago_1, Q)$. Then there is a standard parabolic subgroup $R$ such that $s$ is represented by an element in $W(\ago_1, \ago_R)$ (see \cite[1.3.7]{LabWal} for the fact that $R$ can be assumed to be standard). Let $C$ be the positive chamber in $\ago_{R}$ (i.e., $C$ is determined by the roots in $\Delta_R$). Then $s^{-1}(C)$ is a chamber in $\ago_1$ cut up by the roots in $\Phi(G, A_1)$, hence it is contained in $u^{-1}(C_{\sigma})$ for some $u\in W^{(G_\sigma^0, A_1)}$ where $C_{\sigma}$ is the positive chamber for the root system $\Phi(G_\sigma^{0}, A_1)$. 
It is clear that if $\beta\in \Phi(G_\sigma^0, A_1)$, then $s\beta>0$ iff $u\beta>0$. In particular, we have $su^{-1}\beta>0$ for any positive $\beta$ in $\Phi(G_\sigma^0, A_1)$. It is proved in \cite[1.3.7]{LabWal} that there is an element $w$ in the coset $W^{Q}su^{-1}$  such that $w^{-1}\alpha>0$ for any $\alpha\in \Delta_{B}^Q$. Using \cite[1.3.7]{LabWal} again, we know that $w$ sends $\ago_1$ to $\ago_{Q_1}$ for some standard parabolic subgroup $Q_1$. Hence the restriction of $w^{-1}$ to $\ago_{Q_1}$ sends roots in $\Delta_{Q_1}^Q$ to positive roots. Note that as an element in $W^{Q}su^{-1}$, $w$ must sends a positive root $\beta$ in $\Phi(G_\sigma^0, A_1)$ to a root in $\Phi(Q, A_{Q_1})$. Hence we may write $$w\beta=\sum_{\gamma\in \Delta_{Q_1}^Q}m_{\gamma}\gamma+\sum_{\gamma\in\Delta_{Q_1}- \Delta_{Q_1}^Q}n_{\gamma}\gamma,  $$
with $n_{\gamma}\geq 0$ for any $\gamma\in \Delta_{Q_1}- \Delta_{Q_1}^Q$. 
While $w^{-1}\Delta_{Q_1}$ is a base for $\Phi(G, A_1)$, any element in $\Phi(G, A_1)$ can be written as a linear combination of roots in $w^{-1}\Delta_{Q_1}$ with coefficients either being all non-negative or being all non-positive. If $m_{\gamma}$ and $n_{\gamma}$ above are all non-negative, then we are done. Otherwise, they are all non-positive. Hence $n_{\gamma}=0 $ for any $\gamma\in \Delta_{Q_1}- \Delta_{Q_1}^Q$, and in this case we have
$$\beta=\sum_{\gamma\in \Delta_{Q_1}^Q}m_{\gamma}w^{-1}\gamma,$$
with $m_{\gamma}\leq 0$ and $w^{-1}\gamma>0$. This is a contradiction because $\beta$ is positive. 
\end{proof}

We are going to use the  formula \eqref{var-ch} to replace the sum over $\gamma\in M_Q(F)\cap o$ in $j_{Q,o}$ by allowing $s$ to be taken in the set $W(\ago_1, Q; G_{\sigma}^0)$, $\mu$ to be taken in the set $M_Q(F)\cap w_sG_{\sigma}^0(F)w_s^{-1}\backslash M_Q(F)$ and $u$ to be taken in the set $\mathcal{U}_{G_\sigma}(F)\cap w_s^{-1}M_Q(F)w_s $. In this way, for each $s\in W(\ago_1, \ago_{Q_1})$, the element $\gamma$ is repeated 
\begin{equation}\label{times}
|{sW^{(G_{\sigma}, A_1)} \cap W(\ago_1; Q, G_\sigma^0)   }| |\frac{M_Q(F) \cap w_sG_{\sigma}(F) w_s^{-1} }{ M_Q(F)  \cap  w_s G_\sigma^0(F)w_s^{-1}} |  \end{equation}
times.  
We have \begin{equation}\label{W1}
\frac{M_Q(F)  \cap  w_sG_{\sigma}(F)w_s^{-1} }{M_Q(F)  \cap w_s G_\sigma^{0}(F)  w_s^{-1} }\cong  \frac{  {s}^{-1}W^{Q}(\ago_{Q_1},\ago_{Q_1}){s}\cap W^{(G_\sigma , A_1)}}{ s^{-1}W^{Q}(\ago_{Q_1},\ago_{Q_1})s \cap W^{(G_\sigma^{0} , A_1 )}}.\end{equation}
While by our Lemma \ref{WaQG} and \cite[1.3.7]{LabWal}, we know that
\begin{align*}| W(\ago_1, Q; G_{\sigma}^{0})\cap sW^{(G_\sigma, A_1)}| &=   | W^Q(\ago_{Q_1}, \ago_{Q_1})  \backslash  W^Q(\ago_{Q_1}, \ago_{Q_1})sW^{(G_\sigma, A_1)}/W^{(G_\sigma^0,A_1)}|       \\
&= |  {s^{-1}W^Q(\ago_{Q_1}, \ago_{Q_1})s\cap W^{(G_\sigma, A_1)}  }\backslash {W^{(G_\sigma, A_1)}}/    W^{(G_\sigma^{0},A_1 ) }  | . \end{align*}
As $W^{(G_\sigma^{0}, A_1)}$ is normal in $W^{(G_\sigma, A_1)}$, the above cardinality equals:
 \begin{equation}\label{W2}
 \frac{ |    {s^{-1}W^Q(\ago_{Q_1}, \ago_{Q_1})s\cap W^{(G_\sigma^{0},A_1)} }|| W^{(G_\sigma, A_1)}|  }{   |  {s^{-1}W^Q(\ago_{Q_1}, \ago_{Q_1})s\cap W^{(G_\sigma, A_1)}||  W^{(G_\sigma^0,A_1)}     }| }.   \end{equation}
Combining (\ref{W1}) and (\ref{W2}), we conclude that \eqref{times} equals: 
$\frac{|W^{(G_\sigma,A_1)}|}{| W^{(G_\sigma^0,A_1)}|}. $
This is just $|\pi_{0}(G_{\sigma})(F)|$.

\subsubsection{}
It follows from above discussions that $j_{Q,o}$ equals the sum over $s$ in  $W(a_{1}; Q, G_\sigma^{0})$ of 
\begin{equation}\label{334E1} \frac{1}{| \pi_{0}(G_{\sigma})(F)| }\sum_{\mu}\sum_{\nu}\sum_{u}\sum_{n} f(y^{-1}\nu^{-1}\mu^{-1}w_s\sigma \mu w_s^{-1}\mu n\nu y   ), \end{equation}
in which $\mu$, $\nu$, $u$ and $n$ are summed over 
$$ M_Q(F)\cap w_sG_{\sigma}^{0}(F)w_s^{-1}\backslash M_Q(F), $$
$$  N_{Q, \mu^{-1}w_s \sigma w_s^{-1}\mu}(F) \backslash N_Q(F), $$
 $$ w_s^{-1}M_Q(F)w_s \cap \mathcal{U}_{G_\sigma}(F), $$
 and 
 $$N_{Q,  \mu^{-1}w_s\sigma w^{-1}_s\mu  }(F)$$
respectively. 
We replace $w_s^{-1}\mu n$ by $nw_s^{-1}\mu$ in the expression \eqref{334E1}, then the sum in $n$ needs to be changed to the sum taken over $n\in w_sN_{Q, \sigma}(F)w_s.$
Besides, since $$N_{Q, \mu^{-1}w_s\sigma w_s^{-1}\mu}(F)=\mu^{-1}(N_Q(F)\cap w_sG^0_{\sigma}(F)w_s^{-1})\mu,   $$
we can combine the sum in $\mu$ and $\nu$ to a sum over $$\pi\in Q(F)\cap w_sG_{\sigma}^{0}(F)w_s^{-1}\backslash Q(F),$$
and the set that $u$ is taken in is unchanged. We obtain
$$j_{Q, o}(y)=\frac{1}{| \pi_{0}(G_{\sigma})(F)|  }   \sum_{s}\sum_{\pi}\sum_{u}\sum_{n\in w_{s}^{-1}N_{Q, \sigma}(F)w_s } f(y^{-1} \pi^{-1} w_s\sigma u n w_s^{-1}\pi y ). $$
We substitute this into the expression
\begin{equation}\label{3.1}
\sum_{\delta\in Q(F)\backslash G(F)}  \htau_Q(H_Q(\delta x)+  s_{\delta x}\xi  )j_{Q, o}(\delta x).
\end{equation}
Then we take the sum over $\delta$ inside the sum over $s$ and combine it with the sum over $\pi$. For a given $s$, this produces a  sum over 
$$\eta\in  Q(F)\cap w_sG_{\sigma}^{0}(F)w_s^{-1}\backslash G(F), $$
and the expression $(\ref{3.1})$ becomes 
\begin{equation}\frac{1}{|\pi_{0}(G_{\sigma})(F)|}  \sum_{s}\sum_{\eta}\sum_{u}\sum_{n} f(x^{-1}\eta^{-1}w_s \sigma u n w_s^{-1}\eta x   )  \htau_{Q}(H_Q(\eta x) +   s_{\eta x}\xi  ).   \end{equation}

Let $R= w_s^{-1}Qw_s\cap G_{\sigma}^{0}$. Then $R$ is a standard parabolic subgroup of $G_\sigma$ with Levi decomposition $$R=M_R N_R=( w_s^{-1}M_Q w_s\cap G_{\sigma}^{0})( w_s^{-1}N_Qw_s\cap G_{\sigma}^{0}). $$
Replace $\eta$ by $w_s\eta$, changing the corresponding sum to one over $R(F)\backslash G(F)$. It follows that the expression (\ref{3.1}) equals the sum over $$s\in W(\ago_1, Q; G_{\sigma}^{0})$$
of $$\frac{1}{|\pi_{0}(G_{\sigma})(F) |}\sum_{\eta   \in R(F)\backslash G(F)   }\htau_Q(H_Q(w_s \eta x)+ s_{w_s\eta x}\xi ) \sum_{u\in \mathcal{U}_{M_R}(F) } \sum_{n\in N_R(F)}  f(x^{-1}\eta^{-1} \sigma un \eta x  ) .   $$


We see that $j_{o}^{\xi}(x)$ equals the sum over standard parabolic subgroups $R$ of $G_{\sigma}^{0}$ and elements $\eta$ in $R(F)\backslash G(F)$ of the product of 
$$\frac{1}{| \pi_{0}(G_{\sigma})(F) | } \sum_{u\in \mathcal{U}_{M_R}(F)}\sum_{n\in N_R(F)} f(x^{-1}\eta^{-1}\sigma un \eta x    ) \d n         $$
and 
\begin{equation} \label{3.2}\sum_{(Q,s)}(-1)^{\dim \ago_Q^G}\htau_Q(H_Q(w_s\eta x)  + s_{w_s \eta x}\xi    )  ,  \end{equation}
where $(Q, s)$ is to be summed over the set \begin{equation}\label{set.}
\{ (Q,s)\in \mathcal{P}(B)\times W(\ago_1, Q; G_{\sigma}^{0} )|     w_s^{-1}Qw_s\cap G_\sigma^{0}=R \}. \end{equation}
We shall write $\mathcal{F}_R(M_1)$ for the set of parabolic subgroups $P\in \mathcal{F}(M_1)$ such that $P_{\sigma}^{0}=R$. It turns out that the map $(Q,s)\mapsto w^{-1}_sQw_s\cap G_{\sigma}^0$ establishes a bijection between the set \eqref{set.} and the set $\mathcal{F}_R(M_1)$. 
 
 \subsubsection{}
 Suppose that $(Q,s)$ is a pair in \eqref{set.}. Then $P=w_s^{-1}Qw_s$ is  a parabolic subgroup in $\mathcal{F}_R(M_1)$. Conversely, for any $P\in \mathcal{F}_R(M_1)$, there is a unique standard parabolic subgroup $Q$ and an element $s\in W$ such that $P=w_s^{-1}Qw_s$. We demand that $s^{-1}\alpha$ be positive for each root $\alpha\in \Delta_{B}^{Q}$, so that $s$ will be also uniquely determined. Since the space $s(\ago_{P_1})$ includes $\ago_Q$, it must be of the form $\ago_{Q_1}$ for a standard parabolic subgroup $Q_1$. Otherwise, $s^{-1}$ would map some positive, non-simple linear combination of roots in $\Delta_{B}^{Q}$ to a simple root in $\Delta_{B}$, a contradiction. Moreover, combining this property with the fact that $Q$ includes $w_sRw_s^{-1}$, we see that $s\beta$ is positive for every positive root $\beta$ of $(G_{\sigma}^{0}, A_1)$. It follows that the restriction of $s$ to $\ago_{P_1}$ defines a unique element in $W(\ago_1, Q; G_{\sigma}^{0})$. Therefore, the sum in (\ref{3.2}) can be replaced by the sum over $P\in \mathcal{F}_R(M_1)$.

As the Weyl element $w_s$ is chosen in $G(\mathbb{F}_q)$, we have
 \begin{equation}H_Q(w_s\eta x) =sH_P(\eta x).  \end{equation}
So $$\htau_Q(H_Q(w_s \eta x) +  s_{w_s\eta x}\xi )=\htau_P(H_P(\eta x) +  \xi_{\eta x, P}), $$
where $\xi_{\eta x, P}=s^{-1}s_{w_s\eta x}\xi$.

 \begin{prop}(Compare  \cite[Lemma 3.1.]{Ageom})\label{interm}
Given a class $o\in \mathcal{E}$ admitting the Jordan-Chevalley decomposition, we choose a semisimple representative $\sigma$ contained in a standard minimal Levi subgroup. Then
$j_{o}^{\xi}(x)   $ equals the sum over standard parabolic subgroups $R$ of $G_{\sigma}^0$ and elements $\eta\in R(F)\backslash G(F)$ of the product of 
$$  \frac{1}{| \pi_{0}(G_{\sigma})(F)|}    \sum_{u\in \mathcal{U}_{M_R}(F)   }\sum_{n\in N_R(F)} f(x^{-1}\eta^{-1}\sigma u n\eta x  )  $$
with $$   \sum_{P\in {\mathcal{F}}_R(M_1)}(-1)^{\dim\ago_{P}^{G}}  \htau_P(H_P(\eta x)+ \xi_{\eta x, P}) . $$
 \end{prop}

 \subsection{Proof of Theorem \ref{expansion}}\label{vanish}
 \subsubsection{}
 Now suppose that the support of the test function $f$ is contained in $ \mathcal{I}_\infty\times \prod_{v\neq \infty}G(\ooo_v)$. We will further simplify the intermediate expression in Proposition \ref{interm}.
  
First of all, suppose $o\in \mathcal{E}$ is a class such that $J_{o}^{G, \xi}(f)\neq 0$.
Since the support of $f$ is contained in $G(\ooo)$, it implies that there is an element $\gamma\in o$ and $x\in G(\AAA)$ such that $x^{-1}\gamma x\in G(\ooo)$.  However, this implies that $\gamma^n \in xG(\ooo)x^{-1}\cap G(F)$ for any $n$. Note that $xG(\ooo)x^{-1}\cap G(F)$ is finite, so the element $\gamma$ is a torsion element. By Proposition \ref{JordanC}, $\gamma$ admits the Jordan-Chevalley decomposition. Moreover, Theorem \ref{staconj} implies that we can take as representative a semisimple element $\sigma\in G(\mathbb{F}_q)$. 

In the Lie algebra case, we can use the same argument, then Proposition \ref{JordanL} implies that elements in $o\in \mathcal{E}$ admit Jordan-Chevalley decomposition and Theorem \ref{staconjLie} implies that we can take as representative a semisimple element $y\in \ggg(\mathbb{F}_q)$.

Up to $G(\mathbb{F}_q)$-conjugation, we assume that $\sigma$ lies in a minimal standard Levi subgroup $M_1$ defined over $\mathbb{F}_q$. Suppose that $M_1$ is the standard Levi subgroup of the standard parabolic subgroup $P_1$. The group $M_1$ is also minimal among Levi subgroups defined over $F$ since the split rank of $M_{1\sigma}^{0}$ defined over $\mathbb{F}_q$ is unchanged after the base change to $F$. 

We need to prepare some results thanks to the smallness of the support of $f$. 
 
 \begin{lemm}\label{integral}
Let $P$ be a standard parabolic subgroup of $G$, and $\sigma\in M_P(\mathbb{F}_q)$ be a semisimple element. For any $n\in N(\AAA)$ and $u\in N_{P,\sigma}(\AAA)$ such that $$\sigma^{-1}n^{-1}\sigma un\in  N_P(\mathcal{O}), $$
we have $n\in N_{P, \sigma}(\AAA)N_P(\mathcal{O})$. 
\end{lemm}
\begin{proof}
Applying proposition \cite[Proposition 7.2]{Yu2} to the element $x:=\sigma^{-1}n^{-1}\sigma un$ when $A=\ooo$ and use the uniqueness part of this proposition when $A=\AAA$, we see that the set $N_{\sigma}(\AAA)n$ has a nonzero intersection with $N(\mathcal{O})$. 
\end{proof}

 \begin{lemm}[Chaudouard]\label{support}
 Let $\sigma\in G(\mathbb{F}_q)$, and $R$ be a standard parabolic subgroup of $G_\sigma^0$. 
 If $$ (\sigma \mathcal{U}_{M_R(F)}{N_{R}(F)})^{x}\cap G(\ooo)\neq \emptyset ,   $$
 then $x\in M_R(F)M_{1,\sigma}^{0}(\AAA) N_{1, \sigma}(\AAA) G(\ooo)$. 
 \end{lemm}
\begin{proof}
This is a generalisation of the case when $G=GL_n$ of Lemma 6.2.5. of \cite{Chau}. His proof is valid in our cases. One argument used in his proved is that any unipotent element in $G_{\sigma}^{0}(F)$ is contained in the unipotent radical of a parabolic subgroup of $G_{\sigma}^0$. This is \cite[Proposition 3.2]{Yu2}(and \cite[Proposiiton 3.3]{Yu2} for the Lie algebra version). Another argument used in his proof is proved in Lemma \ref{integral}. 
\end{proof}

After these lemmas, since the support of the test function $f$ is contained in $\mathcal{I}_{\infty}\times \prod_{v\neq \infty}G(\ooo_v)$, the element $\sigma$ is $G(\kappa_\infty)$-conjugate to an element in $B(\kappa_\infty)$. While $\kappa_\infty=\mathbb{F}_q$, we must have $M_1=T$.  

\subsubsection{}
Now we prove that $J^{G, \xi}_o(f)$ vanishes if $o$ is not elliptic. We need the following result first. 
\begin{lemm}\label{okay}
Let $s\in W(\ago_1, Q; G_{\sigma}^{0})$ as defined in Lemma \ref{WaQG} and $R= w_s^{-1}Qw_s\cap G_{\sigma}^0$. Suppose that $x$ is an element in $G(\AAA)$ such that 
$$ (\sigma \mathcal{U}_{M_R(F)}{N_{R}(F)})^{x}\cap (\mathcal{I}_\infty\times \prod_{v\neq \infty}G(\ooo_v)) \neq \emptyset .   $$
We have $$W^{Q}s_{w_s x}=W^{Q}s s_x. $$
\end{lemm}
\begin{proof}
The result is local in nature. Assume simply that $x\in G(F_{\infty})$. 
By Lemma \ref{support}, we have $x\in R(F_\infty)G(\ooo_\infty)$. Recall that $s_x$ is defined to be the Weyl element represented by $k$ in $\mathcal{I}_\infty \backslash G(\ooo_\infty)/\mathcal{I}_\infty$ for any Iwasawa decomposition $x=bk$ with $b\in B(F_\infty)$ and  $k\in G(\ooo_\infty)$. 
Since $w_sR(F_\infty) \subseteq Q(F_\infty)w_s$, we are reduced to the case that $x=k\in G(\ooo_\infty)$.

As everything is unchanged after reduction mod-$\wp_\infty$, we can assume that $k\in G(\kappa_\infty)$ and $k^{-1}\sigma k\in B(\kappa_\infty)$.

The hypothesis implies that $k^{-1}\sigma k$ is a $\kappa_\infty$-point of a split maximal torus in $B$ defined over $\kappa_\infty$.
We can conjugate this torus to $T$ by an element in $B(\kappa_\infty)$, hence for some element $b\in B(\kappa_\infty)$ and $w\in N_G(T)(\kappa_\infty)$, we have $$w^{-1} b^{-1}k^{-1}\sigma kbw =\sigma. $$
By modifying $w$ if necessary, we can assume that $kbw$ lies in the connected centralizer of $\sigma$, i.e. $kbw\in G_{\sigma}^{0}(\kappa_\infty)$. 
We could further assume that $w$ sends positive roots of $\Phi(G_{\sigma}^{0}, T)$ to positive roots (of $\Phi(G, T)$) (see Lemma \ref{WaQG} or rather its proof), hence $ B_{\sigma}w^{-1}\subseteq  w^{-1} b $. Then using Bruhat decomposition of $G_{\sigma}^{0}(\kappa_\infty)$ relative to $B_{\sigma}(\kappa_\infty)$ to $kbw$, we conclude that $k\in B_{\sigma}(\kappa_\infty) s_k B(\kappa_\infty)$ ($s_k=w_0w^{-1}$ if $kbw$ belongs to the Bruhat cell associated to the Weyl element $w_0$).

We conclude that 
$$w_sk\in  w_sB_{\sigma}(\kappa_\infty)s_kB(\kappa_\infty)\subseteq B(\kappa_\infty)w_s s_k B(\kappa_\infty).$$
This implies the result needed.  
\end{proof}

We come back to the expression for $j^{\xi}_o(x)$ in Proposition \ref{interm}. Suppose that there is an $x\in G(\AAA)$, such that $$j^{\xi}_o(x)\neq 0.  $$
Then there is a standard parabolic subgroup $R$ of $G_{\sigma}^0$ and an element $\eta\in R(F)\backslash G(F)$ such that \[(\sigma \mathcal{U}_{M_R(F)}{N_{R}(F)})^{\eta x}\cap (\mathcal{I}_\infty\times \prod_{v\neq \infty}G(\ooo_v))\neq \emptyset , \]
 and \[
 \sum_{P\in \mathcal{F}_{R}(M_1)}(-1)^{\dim \ago_P^G}\htau_P(H_{M_1}(m) + s_{\eta x}\xi ) \neq 0,  \]
 where $\mathcal{F}_{R}(M_1)$ is the set of parabolic subgroups $P$ containing $M_1$ such that $P_{\sigma}=R$.

By Lemma \ref{support}, there is an $m\in M_{1,\sigma}^{0}(\AAA)$ such that $\eta x \in M_R(F)mN_{1,\sigma}(\AAA)G(\ooo)$.
 For any $P\in \mathcal{F}_{R}(M_1)$, 
by Lemma \ref{okay}, we have $$\htau_{P}(H_0(\eta g)+ s_{\eta x, P}\xi     )=\htau_P(H_{M_1}(m)+  s_{\eta x}\xi  ). $$
We can apply Arthur's combinatoric lemma \cite[Lemma 5.2]{Ageom}. It says that since the sum
 \begin{equation}\label{taur}
 \sum_{P\in \mathcal{F}_{R}(M_1)}(-1)^{\dim \ago_P^G}\htau_P(H_{M_1}(m) + s_{\eta x}\xi )   \end{equation}
 is non-zero, under the decomposition \[\ago_T=\ago_T^R\oplus\ago_R,\] we have 
 $$[H_{M_1}(m)+  s_{\eta x}\xi ]_{R}\in \ago_{R}^{G_{\sigma}^{0}},$$ 
 where $[H_{M_1}(m)+  s_{\eta x}\xi ]_{R}$ is the $\ago_R$ part of $H_{M_1}(m)+  s_{\eta x}\xi$. 
 It means that with the decomposition \[[H_{M_1}(m)+  s_{\eta x}\xi ]_{R}= [H_{M_1}(m)+  s_{\eta x}\xi ]_{R}^{G_{\sigma}^0}+  [H_{M_1}(m) + s_{\eta x}\xi ]_{G_\sigma^0}, \]
 we have \[[H_{M_1}(m) + s_{\eta x}\xi ]_{G_\sigma^0}=0.  \]
In particular, \[[s_{\eta x}\xi ]_{G_\sigma^0}= [- H_{M_1}(m)]_{G_\sigma^0}\in X_*(L), \]
where $L=C_{G}(A_{G^0_\sigma})$ is the Levi subgroup equal to the centralizer of the maximal split central torus $A_{G^0_\sigma}$ of $G_{G^{0}_{\sigma}}$ defined over $F$. 
Since $\xi$ is in general position and $G$ is semisimple, this is possible only if $A_{G^0_\sigma}$ is trivial i.e. $\sigma$ must be elliptic and in this case \cite[Lemma 5.2]{Ageom} says that \eqref{taur} equals
\[(-1)^{\dim \ago_R}\htau_{R}(H_{M_1}(m)+ s_{\eta x}\xi). \]

\subsubsection{}
In this last part, we must deal with some complexities caused by $\xi$ (although we have also benefited from it to simplify things). 

We suppose that $o$ is represented by a semisimple elliptic element in $T(\mathbb{F}_q)$. First of all, notice that $G_\sigma^0$ is split so $\pi_0(G_\sigma)$ is a constant group scheme and we have  $|\pi_0(G_\sigma)(F)|=| \pi_0(G_\sigma)|$, the degree of $\pi_0(G_\sigma)$. 
In this case $J^{G, \xi}_{o}(f)$ equals 
$$ \frac{1}{| \pi_0(G_\sigma)|}\sum_{R}\int_{R(F)\backslash G(\AAA)} 
(-1)^{\dim\ago_R}\htau_{R}(H_{R}(x)+s_{x}\xi  ) \sum_{u\in \mathcal{U}_{M_R}(F)}\sum_{n\in N_R(F)} f(x^{-1}\sigma u nx) \d x, $$
where the sum over $R$ is taken over the set of standard parabolic subgroup of $G_{\sigma}^0$.
By Lemma \ref{support}, the integral over $x\in R(F)\backslash G(\AAA)$ can be taken over the smaller domain $$R(F)\backslash G^0_{\sigma}(\AAA)G(\ooo).$$
Since we have normalized the measure of $G(\AAA)$ and $G^0_{\sigma}(\AAA)$ such that both $G(\ooo)$ and $G_{\sigma}^{0}(\ooo)=G(\ooo)\cap G_{\sigma}^{0}(\AAA)$ have volume $1$, the integration over $x$ can be further decomposed to the double integration over $(x,k)\in G_{\sigma}^0(\AAA)\times G(\ooo)$ by product integration formula. Therefore $J^{G, \xi}_o$ equals ${|\pi_0( G_\sigma )|}^{-1}$ times 
the sum over standard parabolic subgroup $R$ of $G^0_{\sigma}$ of
\begin{equation}\label{integ}(-1)^{\dim\ago_R} \int_{R(F)\backslash G_{\sigma}^{0}(\AAA)} \int_{G(\ooo)}
\htau_{R}(H_{R}(x)+ s_{xk}\xi  ) \sum_{u\in \mathcal{U}_{M_R}(F)}\sum_{n\in N_R(F)} f(k^{-1}x^{-1}\sigma u nxk) \d k \d x. \end{equation}



It follows from the proof of Lemma \ref{okay} that the integration over $G(\ooo)$ can be taken over the smaller subset $G^0_{\sigma}(\ooo_\infty)\dot{W}\mathcal{I}_{\infty}\times \prod_{v\neq \infty}G(\ooo_v)$. 
where $\dot{W}$ is the subset of elements $w\in W$ such that $w^{-1}$ sends positive roots of $(G_{\sigma}^0, T)$ to positive roots  (hence a set of representative of $W^{(G_{\sigma}^0, T)}\backslash W$). For each $k\in G^0_{\sigma}(\ooo_\infty)\dot{W}\mathcal{I}_{\infty}$, suppose $w(k)\in \dot{W}$ is the element such that $k\in G^0_{\sigma}(\ooo_\infty)w(k)\mathcal{I}_{\infty}$. 
Choose an element $p_k\in G^0_{\sigma}(\ooo_\infty)$ such that $k=p_kw(k)\mathcal{I}_{\infty}$. 
Then we have $$s_{xk}=s_{xp_k}w(k).$$

Therefore we can change the order of the integration in \eqref{integ}. After a change of variable $x\mapsto xp_k^{-1}$,  $J^{G, \xi}_o$ equals the integration over $k\in (G_{\sigma}^{0}(\mathcal{O}_{\infty})\dot{W}\mathcal{I}_{\infty})\times \prod_{v\neq \infty}G(\ooo_v)$ of the sum over standard parabolic subgroup $R$ of $G_{\sigma}^0$ of ${| \pi_0(G_{\sigma})|}^{-1}$ times
$$ (-1)^{\dim\ago_R }\int_{R(F)\backslash G_{\sigma}^{0}(\AAA)}
\htau_{R}(H_{R}(x)+ s_{x}w(k)\xi  ) \sum_{u\in \mathcal{U}_{M_R}(F)}\sum_{n\in N_R(F)} f(w(k)^{-1}x^{-1}\sigma u nxw(k))  \d x\d k. $$
where we have used the fact that $f$ is $\mathcal{I}_{\infty}\times \prod_{v\neq\infty}G(\ooo_v)$-conjugate invariant. Thus we have reduced $J^{G, \xi}_o$ to the following expression:
\[\frac{1}{|\pi_0( G_\sigma ) | }\int_{G_{\sigma}^{0}(\mathcal{O}_{\infty})\dot{W}\mathcal{I}_{\infty}} J^{G_{\sigma}^0, w(k)\xi}_{[\sigma]}(f^{w(k)^{-1}})\d k= \frac{\vol(G_{\sigma}^{0}(\mathcal{O}_{\infty})w\mathcal{I}_{\infty})   }{|\pi_0( G_\sigma ) | }\sum_{w\in \dot{W}} J^{G_{\sigma}^0, w\xi}_{[\sigma]}(f^{w^{-1}})\d k. \]
Note that for any $w\in W$, $wBw^{-1}\cap G^{0}_{\sigma}$ is a Borel subgroup in $G_\sigma^{0}$, therefore 
we have $$\frac{|(G^0_{\sigma}(\mathbb{F}_q)wB(\mathbb{F}_q))|}{|G(\mathbb{F}_q)|}= \frac{| G^0_{\sigma}(\mathbb{F}_q)|| B(\mathbb{F}_q) |}{| B_{\sigma}(\mathbb{F}_q)||G(\mathbb{F}_q)|}. $$
This is independent of $w$. It shows that under our normalization of the Haar measure ($\vol(G(\ooo_\infty))=\vol(G_{\sigma}^{0}(\ooo_\infty))=1$), 
\[{\vol(G_{\sigma}^{0}(\mathcal{O}_{\infty})w\mathcal{I}_{\infty})   }=  \frac{\vol(\mathcal{I}_{\infty})}{\vol(\mathcal{I}_{\infty,\sigma})}.  \]
This completes the proof.

\section{The Hitchin moduli stack}\label{Hitchin's}
We use Ngô, Chaudouard, and Laumon's work on the Hitchin fibration. However, we need to make some remarks first.

A problem for us is that some of Ngô's results require that the divisor $D$ has degree $>2g$. For the divisor we will use, it means that $\deg S\geq 3$, which is unpleasant for us. His condition is not essential in our case. In fact, our divisor is the sum of a canonical divisor and an effective divisor, and we can replace some of Ngô's dimension calculations to make his arguments still work. In \ref{Hitchinbase}, we do these dimension calculations and recall some of Ngô's results that we need. The results that we need are Ngô's calculations of some connected components. 

\label{rmk321} Another problem for us is that Ngô's results are based on the existence of Hitchin-Kostant section $\epsilon: \mathcal{A}_G\rightarrow \mathcal{M}_G$, which is constructed under the hypothesis that there exists a line bundle $\mathcal{L}$ on $X$ such that $\mathcal{O}_X(D)\cong \mathcal{L}^{\otimes 2}$, as shown in \cite[2.5]{NgoH}. This ``square root" $\mathcal{L}$ is needed since the half sum of the positive coroots $\rho$ sometimes does not give a morphism $\mathbb{G}_m\rightarrow G$. Although such a phenomenon never happens for some types of root systems and never happens if $G$ is of adjoint type, we still want to avoid it. Otherwise, we have to assume at least that $S$ has an even degree. For our purpose, it is only relevant to the fact that the Hitchin fibration is surjective and flat (for generic regular semisimple parts). We will deal with this in Theorem \ref{Faltings} (see also Remark \ref{surjective}).

\subsection{The Hitchin base}\label{Hitchinbase}\label{4.1}
\subsubsection{}
As a general rule, if a notation uses the group $G$ as the subscript and it is clear from the context, we will omit $G$ from the notation. 

Let $M$ be a semistandard Levi subgroup of $G$ defined over $\mathbb{F}_q$. 
We define \begin{equation}\mathfrak{c}_M=\mmm\sslash M, \end{equation}
the categorical quotient of $\mmm$ by the adjoint action of $M$. So $\mathfrak{c}_M = \Spec(\mathbb{F}_q[\mmm]^{M})$.  Since $\mathbb{F}_q[\mmm]^{M}$ is (non-canonically) isomorphic to a polynomial ring, $\mathfrak{c}_M$ is an affine space. It has a structure of vector space by choosing a Kostant section (\cite[1.2]{Ngo}). It is  known that when the characteristic is very good, we have an isomorphism (see, for example, \cite[2.3.2]{Riche}) \begin{equation}\ttt\sslash W^M\cong \mathfrak{c}_M,\end{equation}
where $W^M$ is the Weyl group of $M$ relative to $T$. 
Let \[\chi_M: \mmm\rightarrow \car_M, \] be the natural projection. 
Let $ \ggg^{\mathrm{reg}}\subseteq \ggg$ be the open set of regular semisimple elements, $\ttt=\ttt\cap \ggg^{\rs}$ and $\mmm^{\rs}=\mmm\cap \ggg^{\rs}$. 
Let $\mathfrak{c}_M^{\rs}$ be the open subset of $\car_M$ that consists of the image of $G$-regular semisimple elements.

We fix a finite set of closed points $S$ of $X$.  We sometimes identify a point $x\in S$ with an associated point in $X(\kappa_x)$. 
We suppose that $S$ contains a fixed point $\infty\in X(\mathbb{F}_q)$. 
Let \[D=K_X+\sum_{x\in S}  x, \] 
where $K_X$ is any canonical divisor of $X$. 
Let $\mathcal{O}_X(D)$ be the associated line bundle over $X$. Note that the divisor $D$ furnishes $\mathcal{O}_X(D)$ a canonical trivialization outside the support of $D$. 
Define \begin{equation}\car_{M,D}:=\car_{M}\times^{\mathbb{G}_m}\mathcal{O}_{X}(D), \end{equation} as an affine  bundle over $X$. 
Let \[\mathcal{A}_{M}=H^{0}(X, \car_{M,D}), \] 
as an $\mathbb{F}_q$-scheme. When $M=G$, we will omit $M$ from the notation. 

The canonical trivialization of $\mathcal{O}_X(D)$ over $X-\mathrm{supp}(D)$ induces a generic trivialization of $\mathcal{O}_X(D)$. With this generic trivialization, we have an injective map \begin{equation}\label{Acar}
\mathcal{A}_G(\overbar{\mathbb{F}}_q) \hookrightarrow \car_G(F\otimes\overbar{\mathbb{F}}_q),  \end{equation}
that sends a section $a\in \mathcal{A}_G(\overbar{\mathbb{F}}_q)$ to its value $a_\eta$ at the generic point $\eta$ of $X\otimes \overbar{\mathbb{F}}_q$. Sometimes, especially when dealing with the ring of ad\`eles, it is convenient and useful to view $\mathcal{A}_G(\overbar{\mathbb{F}}_q)$ as a subset of $\car_G(F\otimes\overbar{\mathbb{F}}_q)$.

\subsubsection{}
Let $\overbar{X}:=X\otimes_{\mathbb{F}_q}\overbar{\mathbb{F}}_q$ and $\overbar{S}:=S\otimes_{\mathbb{F}_q}\overbar{\mathbb{F}}_q$. By abusing the notation, we still denote by $\car_{G, D}$ the base change of $\car_{G, D}$ over $X$ to $\overbar{X}$. 
\begin{lemm}\label{511}
Suppose the cardinality $|\overbar{S}|>2-g$, where $g$ is the genus of the curve $X$. Let $M$ be a semistandard Levi subgroup of $G$ defined over $\mathbb{F}_q$. 
\begin{enumerate}
\item
We have a linear map 
\begin{equation}\label{555}
\ev_{M}: H^{0}(\overbar{X}, {\car}_{M, D})  \rightarrow  \prod_{\overbar{v}\in \overbar{S}}( \car_{M, D}\otimes_{\mathcal{O}_{\overbar{X}}}(\mathcal{O}_{\overbar{X}, \overbar{v}}/\wp_{\overbar{v}}) ),\end{equation}
where $\wp_{\overbar{v}}$ is the maximal idea of $\mathcal{O}_{\overbar{X}, {\overbar{v}}}$. 
The kernel of the map has dimension 
\[\frac{1}{2}(2g-2+|\overbar{S}|) \dim \mmm- \frac{1}{2}|\overbar{S}| \dim  \ttt +\dim \zzz_M,  \] 
and the image of the map has codimension $\dim \zzz_M$ in the target. In particular, the map is surjective when $M=G$ ($G$ is assumed to be semisimple). 
\item
For any closed point ${\overbar{v}}$ of $\overbar{X}$, the map
\begin{equation}\label{666}
H^{0}(\overbar{X}, {\car}_{G, D})  \rightarrow {\car}_{G, D}\otimes_{\mathcal{O}_{\overbar{X}}}(\mathcal{O}_{\overbar{X}, {\overbar{v}}}/\wp^2_{\overbar{v}}), \end{equation}
is surjective. 
\end{enumerate}
\end{lemm}
\begin{proof}
$(1).$ Since the characteristic $p$ is very good, we have \[\ttt\sslash W^{M}\cong \zzz_M \oplus \car_{M^{der}}. \]
We know that $\car_{M^{der}}$ is an affine space and the weights of the $\mathbb{G}_m$-action (induced from the $\mathbb{G}_m$ action on $\mmm$) are $m_1+1, \ldots, m_r+1$ where $m_i$ are exponents of the Weyl group of $M$ and $r$ is the rank of $M^{der}$. The integers $m_i$ are strictly positive (\cite[p118.(1)]{Bourbaki}). Moreover, $\mathbb{G}_m$ acts on $\zzz_M$ by homothety, hence has weight $1$. 
Therefore, the vector bundle $\car_{M^{der}, D}$ is isomorphic to the direct sum of 
\[\mathcal{O}_X((m_i+1)D)\] for \( i=1, \ldots, r.  \)
We also have \[\car_{M, D}\cong \car_{M^{der}, D}\oplus \mathcal{O}_X(D)^{\dim\zzz_M}.  \]  

The first assertion is hence a direct conclusion of the following dimension calculations. 
The target of the map \eqref{555} has dimension 
\[|\overbar{S}|  \dim\ttt.  \]
By Riemann-Roch theorem, we have, \[\dim H^{0}(\overbar{X}, {\car}_{G, D})  =\frac{1}{2}|\overbar{S}| \dim  \ttt  +  \frac{1}{2}(2g-2+|\overbar{S}|) \dim \ggg ,   \]
where we have used the fact (\cite[p118.(3), p.119. Theorem 1.(ii)]{Bourbaki}) that \[\sum_{i=1}^{r}m_i =\frac{1}{2}|\Phi(M,T)|= \frac{1}{2}(\dim \mmm-\dim \ttt).\] 
Suppose $|\overbar{S}|>2-g\geq 2-2g$, we
apply Riemann-Roch theorem to $\car_{M^{der}, D}$ and $\mathcal{O}_X(D)^{\dim\zzz_M}$ respectively. Since the kernel is the space of global sections of the direct sum of $\mathcal{O}_{\overbar{X}}(m_i D+K_{\overbar{X}})$ or that of $\mathcal{O}_{\overbar{X}}(K_{\overbar{X}})^{\dim \zzz_M}$ respectively, we conclude that the kernel of the map has dimension 
\[\frac{1}{2}(2g-2+|\overbar{S}|) \dim \mmm- \frac{1}{2}|\overbar{S}| \dim  \ttt +\dim \zzz_M.  \]

$(2).$ It results from the above arguments and the fact that when $|\overbar{S}|>2-g$, the map
 \[H^{0}(\overbar{X},  \mathcal{O}_{\overbar{X}}(mD)) \longrightarrow  \mathcal{O}_{\overbar{X}}(mD)\otimes \mathcal{O}_{\overbar{X}, {\overbar{v}}}/\wp^2_{{\overbar{v}}}\]
 is surjective for any $m\geq 2$. 

\end{proof}

We can upgrade the linear map \eqref{555} to a morphism of $\mathbb{F}_q$-schemes, so that we have a surjective morphism
\begin{equation}\label{resi}
\ev_{G}: {\mathcal{A}}_G\longrightarrow \prod_{v\in S}\mathrm{R}_{\kappa_v | \mathbb{F}_q}\car_{G, D, v}.  \end{equation}
We denote the target by $\mathcal{R}_G$:
\[\mathcal{R}_G:= \prod_{v\in S}\mathrm{R}_{\kappa_v | \mathbb{F}_q}\car_{G, D, v} . \]
The first corollary we obtain from the precedent lemma is the following.

\begin{prop}\label{ARs}
The morphism $\ev_{G}: {\mathcal{A}}_G\rightarrow \mathcal{R}_G$ is surjective when $G$ is semisimple and $\deg S> 2-g$. 
\end{prop}

Ngô has introduced several open $\mathbb{F}_q$-subschemes of $\mathcal{A}_G$: 
\begin{equation}\mathcal{A}_G^{\diamondsuit}\subseteq\mathcal{A}_G^{\mathrm{ani}}\subseteq  \mathcal{A}_G^{\heartsuit}\subseteq \mathcal{A}_G, \end{equation}
where $\mathcal{A}_G^{\heartsuit}$ consists of global sections of $\car_{G,D}$ of generic semisimple regular elements (\cite[4.5]{Ngo}), 
$\mathcal{A}_G^{\mathrm{ani}}$ consists of generic semisimple regular elements that are elliptic over $\overbar{\mathbb{F}}_q$ (\cite[4.10.5, 5.4.7]{Ngo}), and $\mathcal{A}_G^{\diamondsuit}$ is defined in \cite[4.7]{Ngo}(which is called the subset of generic characteristic by Faltings, see \cite[III.1]{Faltings0} and \cite[p.328]{Faltings}). If a point $a\in \mathcal{A}_G(\overbar{\mathbb{F}}_q)$ lies in the open subset $\mathcal{A}_G^{\diamondsuit}$, then the associated cameral curve is smooth. The converse statement is also true if the characteristic $p$ does not divide the order of the Weyl group (\cite[4.7.3]{Ngo}). 

The following two propositions are needed in Ngô's work when he discusses connected components of the Hitchin fibers. They are results of Faltings and Ngô. Here we need to relax the hypothesis on the divisor $D$.  
\begin{prop}\label{nonvide}
Suppose that $G$ is semisimple and $|\overbar{S}|>2-g$. Then if $a\in \mathcal{A}_G^{\heartsuit}(\overbar{\mathbb{F}}_q)$, the cameral curve $\widetilde{X}_a$ (defined in \cite[4.5]{Ngo}) is connected. 
 \end{prop}
 \begin{proof}
  This proposition is \cite[4.6.1]{Ngo} where it is stated under the hypothesis that $\deg D>2g$. However, all one needs in the proof is that the line bundle $\mathcal{O}_{X}((m_i+1)D)$ is very ample (where $m_i\geq 1$ are exponents of the Weyl group as in Lemma \ref{511}). When $G$ is semisimple, we have $m_i\geq 1$, hence 
\[\deg D= (m_i+1)(2g-2+|\overbar{S}|)\geq 2(2g-2+3-g)=2g+2,  \]
when $|\overbar{S}|>2-g$. Therefore $\mathcal{O}_{X}((m_i+1)D)$ is very ample in our case. 
 \end{proof}

\begin{prop}
Suppose that $|\overbar{S}|> 2-g$, then the open subset $\mathcal{A}_G^{\diamondsuit}$ of $\mathcal{A}_G$ is non-empty. 
\end{prop}
\begin{proof}
This is \cite[4.7.1]{Ngo}. It is proved under the hypothesis that $\deg D>2g$. This condition is only needed in \cite[4.7.2]{Ngo}, so his proof works without modification using the second part of Lemma \ref{511} instead. 
\end{proof}

\subsubsection{}\label{eva}
Following \cite[5.3.1]{Ngo}, we introduce an \'etale open subset $\widetilde{\mathcal{A}}_G\rightarrow \mathcal{A}_G$ defined by the following Cartesian diagram:
\[\begin{CD}
\widetilde{\mathcal{A}}_G@>>>\ttt_{D,\infty}^{\mathrm{reg}}\\
@VVV@VVV\\
\mathcal{A}_G@>\ev_{G, \infty}>> \car_{G, D,\infty}\\
\end{CD}.\]
We have an induced morphism, denoted  by $\widetilde{\ev}_G$: 
\begin{equation}
\widetilde{\ev}_G: \widetilde{\mathcal{A}}_G\longrightarrow \widetilde{\mathcal{R}}_G:=  \ttt_{D,\infty}^{\reg}\times \prod_{v\in S-\{\infty\}}\mathrm{R}_{\kappa_v | \mathbb{F}_q}\car_{G, D, v}, 
\end{equation}
given by the base change of the morphism \eqref{resi}.

\begin{prop}\label{irred}
Suppose that $|\overbar{S}|> 2-g$. 
The scheme $\widetilde{\mathcal{A}}_G$ is smooth and geometrically irreducible. 
\end{prop}
\begin{proof}
This is \cite[5.3.2]{Ngo} stated under the hypothesis that $\deg D>2g$, but all one needs are the surjectivity of $\mathcal{A}_G\rightarrow \car_{G, D, \infty}$. This is satisfied because of Lemma \ref{511}. 
\end{proof}

\begin{theorem}Suppose that $|\overbar{S}|> 2-g$. 
There is a unique non-empty maximal open subset of $\mathcal{R}_G$, denoted by $\mathcal{R}^{\ani}_G$, with the property that $\mathrm{ev}_G^{-1}(\mathcal{R}^{\ani}_G)\subseteq \mathcal{A}^{\ani}_G$ (note that the morphism $\mathrm{ev}_G$ is surjective).  
\end{theorem}
\begin{proof}
By definition, the open subset $\mathcal{R}^{\ani}_G$ is unique. We need to show that it is non-empty. For this we can work over the \'etale open subset $\widetilde{\mathcal{A}}_G$ of $\mathcal{A}_G$.    
We have the following commutative diagram:
\[\begin{CD}
\widetilde{\mathcal{A}}_M@>\widetilde{\ev}_{M}>> \ttt_{D,\infty}^{\mathrm{reg}}\prod_{v\in S-\{\infty\}}\mathrm{R}_{\kappa_v | \mathbb{F}_q}\car_{M, D, v}\\
@VVV@VVpV\\
\widetilde{\mathcal{A}}_G@>\widetilde{\ev}_{G}>>  \ttt_{D,\infty}^{\mathrm{reg}}\prod_{v\in S-\{\infty\}}\mathrm{R}_{\kappa_v | \mathbb{F}_q}\car_{G, D, v}\\
\end{CD}.\]

The left vertical morphism is a closed immersion (see \cite[3.5.1]{C-L1}), and \(\widetilde{\mathcal{A}}_G-\widetilde{\mathcal{A}}^{\mathrm{ani}}_G \) is a union of $\widetilde{\mathcal{A}}_M$ over the set of proper semistandard Levi subgroups $L$ of $G$ 
 (see \cite[6.3.5]{Ngo} and also \cite[3.6.2]{C-L1}): 
\[\widetilde{\mathcal{A}}_G - \widetilde{\mathcal{A}}_G^{\mathrm{ani}}= \bigcup_{M\in\mathcal{L}(T), M\neq G}\widetilde{\mathcal{A}}_M. \]
Therefore for any point $a$ in $\widetilde{\mathcal{A}}_G$, if $a$ does not belong to $ \widetilde{\mathcal{A}}_M$ for any proper semistandard Levi subgroup, then $a$ belongs to $\widetilde{\mathcal{A}}^{\mathrm{ani}}$. 

The morphism \[\mathcal{A}_M\longrightarrow \prod_{v\in S}\mathrm{R}_{\kappa_v | \mathbb{F}_q}\car_{M, D, v}\] is affine hence quasi-compact, so its scheme theoretic image commutes with flat base change induced by the morphism \(\ttt^{\reg}_{M,D,\infty}\rightarrow \car^{\rs}_{M,D,\infty}. \)
Moreover, the morphism \(\ttt^{\reg}_{M,D,\infty}\rightarrow \car^{\rs}_{M,D,\infty}\) is finite, while finite morphism is closed and preserves dimension. We conclude that the scheme theoretic image of the morphism \(\res_{M}\) has codimension $\dim \zzz_M$ (Lemma \ref{511}). 
The morphism $p$ in the commutative diagram above is also finite, so the scheme theoretic image of $p\circ \res_M$ also has codimension $\dim \zzz_M$. 
It implies that the union of the images of \(p\circ \res_{M}\) over the set of proper semistandard Levi subgroups (there are finitely many of them) is contained in a closed sub-scheme of strictly smaller dimension. So the complement is a non-empty open subset whose inverse image under the morphism $\mathrm{ev}$ is contained in $\widetilde{\mathcal{A}}_G^{\ani}$.   \end{proof}

\subsection{Chaudouard-Laumon's $\xi$-stability}\label{4.2}
In this subsection, we review Chaudouard-Laumon's work (\cite{C-L1}) and move some of their results from $\overbar{\mathbb{F}}_q$ to $\mathbb{F}_q$. For what we need, it will also be possible to work with the semistable part of the parabolic Hitchin moduli stack. However, it is more convenient for us to work with Chaudouard-Laumon's version since we know more about them. 

\subsubsection{More on the \'etale open $\widetilde{\mathcal{A}}_{M}$}

Since the maximal torus $T$ is split, every semistandard Levi subgroup of $G_{\overbar{\mathbb{F}}_q}$ is defined over $\mathbb{F}_q$.
For each Levi subgroup $M$ of $G$ defined over $\mathbb{F}_q$, we have the Hitchin base $\mathcal{A}_M$ as well as its \'etale open set $\widetilde{\mathcal{A}}_M$ defined by the Cartesian diagram
\[\begin{CD}
\widetilde{\mathcal{A}}_M @>>> \ttt^{\reg}_{D, \infty}\\
@VVV@VVV\\
{\mathcal{A}}_M @>\ev_{\infty}>> \car_{M,D, \infty}
\end{CD}. \]
Recall that in the definition of $\ttt^{\reg}_{D, \infty}$,  we only consider $G$-regular elements instead of $M$-regular ones. The set $\widetilde{\mathcal{A}}_M(\mathbb{F}_q)$ consists of pairs $(a, t)$ with $a\in \mathcal{\mathcal{A}}_M(\mathbb{F}_q)$ and $\ttt^{\reg}_{D, \infty}(\mathbb{F}_q)$ such that 
\[\ev_\infty(a) = \chi_M(t). \]
The obviously defined morphism \begin{equation}\chi^M_G: \widetilde{\mathcal{A}}_M \longrightarrow \widetilde{\mathcal{A}}_G,  \end{equation}
is a closed immersion (\cite[3.5.1]{C-L1}).
Let
\begin{equation}\widetilde{\mathcal{A}}_{M}^{\mathrm{ell}}=\widetilde{\mathcal{A}}_{M}-\cup_{L\in \mathcal{L}(T), L\subsetneq M}\widetilde{\mathcal{A}}_L. \end{equation}
It is an open subscheme of $\widetilde{\mathcal{A}}_M$ (\cite[3.6.1]{C-L1}). 
We have \[\widetilde{\mathcal{A}}_{G}^{\mathrm{ell}}= \widetilde{\mathcal{A}}_{G}^{\mathrm{ani}}, \] see \cite[6.3.5]{Ngo} or \cite[3.6.2]{C-L1}. And any point of $\widetilde{\mathcal{A}}_G$ is in the elliptic part $\widetilde{\mathcal{A}}_M^{\mathrm{ell}}$ for exactly one semistandard Levi subgroup $M$ (\cite[3.6.2]{C-L1}):
\[\widetilde{\mathcal{A}}_G= \cup_{M}\widetilde{\mathcal{A}}_M^{\mathrm{ell}}. \]

Suppose $n_v$ are integers so that \[D=\sum_{v\in |X|}n_v v.  \]
As we have mentioned in \eqref{Acar} that $D$ furnishes $\mathcal{O}_X(D)$ a generic trivialization. Consider the formal neighbourhood $\Spec(\ooo_\infty)$ of the closed point $\infty$ of $X$. 
We get an induced generic trivialization of the bundle $\car_{M,D,\ooo_\infty}:=\car_{M,D}|_{\Spec(\ooo_\infty)}$.
Therefore we can identify $\car_{M,D, \mathcal{O}_\infty}(F_\infty) \cong  \car_{M}(F_\infty)$ and we have an injection $\mathcal{A}_{M}(\mathbb{F}_q)\hookrightarrow \car_{M,D}(F_\infty)$ given by restricting a section in $\mathcal{A}_{M}(\mathbb{F}_q)$ to $\Spec(\ooo_{\infty})$. Moreover,  we have the following commutative diagram.
\begin{equation}\label{41222}\begin{tikzcd}
&\ttt_{D,\ooo_\infty}(\ooo_\infty)\arrow[d,hook] \arrow[r, hook] & \ttt_{D,\ooo_\infty}(F_\infty)  \arrow{r}{\sim}\arrow[d,hook] &\ttt(F_\infty) \arrow[d,hook]  \\ 
&\mmm_{D,\ooo_\infty}(\ooo_\infty)\arrow[d] \arrow[r, hook] & \mmm_{D,\ooo_\infty}(F_\infty)  \arrow{r}{\sim}\arrow[d] &\mmm(F_\infty) \arrow[d]  \\ 
\mathcal{A}_{M}(\mathbb{F}_q)\arrow[r, hook] & \car_{M,D,\ooo_\infty}(\ooo_\infty)\arrow[r, hook] &\car_{M,D,\ooo_\infty}(F_\infty) \arrow{r}{\sim}& \car_{M}(F_\infty)  
\end{tikzcd} \end{equation}
Note that the image of $\mmm_{D,\ooo_\infty}(\ooo_\infty)$ in $\mmm(F_\infty)$ via the above diagram is \[\wp_{\infty}^{-n_\infty}\mmm(\ooo_\infty)\] (``functions admitting poles of order $\leq n_\infty$"), since $\mathbb{G}_m$ has weight $1$ on $\mmm$. 

The the morphism $\ev_\infty: \mathcal{A}_M(\mathbb{F}_q) \rightarrow \car_{M,D,\infty}(\kappa_\infty)$ can be factored as 
\[\ev_{\infty}: \mathcal{A}_M(\mathbb{F}_q)\longrightarrow \car_{M,D,\ooo_\infty}(\ooo_\infty)\longrightarrow \car_{M,D,\infty}(\kappa_\infty). \]
Note that under the above point of view, we have the following commutative diagram:
\begin{equation}\label{41333}\begin{tikzcd}
\wp_{\infty}^{-n_\infty}\mmm(\ooo_\infty)\arrow[d] \arrow[r] & \mmm(\kappa_\infty)   \arrow[d]   \\ 
\car_{M,D,\ooo_\infty}(\ooo_\infty)\arrow[r] &\car_{M,D,\infty}(\kappa_\infty)  \end{tikzcd}, \end{equation}
where the top horizontal arrow is given by ``mod-$\wp_{\infty}^{-n_\infty+1}$".

\subsubsection{$\mathbb{F}_q$ vs. $\overbar{\mathbb{F}}_q$}\label{422}
Suppose $(a,t)\in \widetilde{\mathcal{A}}_{M}(\mathbb{F}_q)$, 
since the morphism $\ttt^{\reg}\rightarrow \ttt^{\reg}\sslash W$ is \'etale, for any $(a,t)\in \widetilde{\mathcal{A}}_{M}(\mathbb{F}_q)$, there is a unique element \[t_{a}\in \ttt^{\reg}_{D,\ooo_\infty}(\ooo_\infty) \cong  \wp_\infty^{-n_\infty} \ttt^{\reg}(\mathcal{O}_\infty),\] such that the restriction of $a$ to $\Spec(\mathcal{O}_{\infty})$ equals $\chi_M(t_a)$. 
The  following lemma is an ``${\mathbb{F}}_q$-version" of part of \cite[Proposition 3.9.1]{C-L1} which is proved over $\overbar{\mathbb{F}}_q$. 
\begin{lemm}\label{rational'}
Suppose that $M$ is a semistandard Levi subgroup defined over $\mathbb{F}_q$ and $(a,t)\in \widetilde{\mathcal{A}}_{M}(\mathbb{F}_q)$. 
Then there exists a semisimple $G$-regular semisimple element $X\in \mmm(F)$ such that 
$$\chi_M(X)=a_{M,\eta}.$$
Moreover for any such $X$, there is an $m\in M(F_{\infty})$ such that 
 $\Ad(m)(X)=t_a$. 
\end{lemm}
\begin{proof}
The existence of semisimple element $X\in \mmm(F)$ such that $\chi_M(X)=a_{\eta}$ can be deduced using a Kostant section. If the characteristic is very large (larger than two times the Coxeter number), we refer the reader to \cite[1.2, 1.3.3]{Ngo} for references and more detailed discussions. In \cite{Riche}, over an algebraically closed field, S. Riche proved that under very mild conditions (in particular if $p$ is very good), a Kostant section exists. 
In fact, fix a pining $(T,B,e)$ defined over $\mathbb{F}_q$, where $e=\sum_{\alpha\in \Delta_B}x_{\alpha}$ with $x_{\alpha}$ being a non-zero vector in the root space $\ggg_\alpha$. 
Choose a $\mathbb{G}_m$-stable complement $\mathfrak{s}$ to $[e, \nnn_B]$ in $\bbb$ that is defined over $\mathbb{F}_q$ and let $\mathcal{S}=e+\mathfrak{s}$ as in \cite[p.230]{Riche}, Riche proves that \cite[3.2.1, 3.2.2]{Riche} the $\mathbb{F}_q$-morphism $\mathcal{S}\rightarrow \ggg/G$ is an isomorphism.

By \cite[3.8.1]{C-L1}, $X$ and $t_a$ are $M(F_\infty\otimes\overbar{\mathbb{F}}_q)$-conjugate in $\mmm(F_{\infty})$. 
Hence using an element in $M(F_\infty\otimes\overbar{\mathbb{F}}_q)$ that conjugates $X$ to $t_a$ we can define a cocyle in $\ker(H^{1}(F_\infty, T)\rightarrow H^{1}(F_{\infty}, M))$ which is trivial as $T$ is split. Therefore $X$ and $t_a$ are $M(F_\infty)$-conjugate. 
\end{proof}

\begin{lemm}\label{lev}
Let $Y\in \ggg(F)$ be a semisimple regular element, such that $Y$ is stably conjugate to an element in $\mmm(F)$ for a Levi subgroup $M$ of (an $F$-parabolic subgroup of) $G$ defined over $F$. Then there is an element in $\mmm(F)$ conjugate to $Y$.  
\end{lemm}
\begin{proof}
Suppose that $Y$ is stably conjugate to $X\in \mmm(F)$ by $g\in G(F^s)$, i.e. 
$$\Ad(g)Y=X. $$ 
Let $P$ be an $F$-parabolic subgroup of $G$ such that $M$ is a Levi subgroup of $P$. 
Note that $g^{-1}Mg$ is then a Levi subgroup of $g^{-1}Pg$. A priori, they are defined over $F^{s}$. However, as $X$ is regular, we have $G_X(F^s)\subseteq M(F^s)$. It implies that both $g^{-1}Pg$ and $g^{-1}Mg$ are stable under $\Gal(F^s|F)$ hence are defined over $F$. Following Borel-Tits, \cite[Th\'eor\`eme4.13.c]{BTu}, two $F$-parabolic subgroups are conjugate if they are $F^s$-conjugate. The result of the lemma then follows as two $F$-Levi subgroups in $P$ are conjugate by an element in $N_P(F)$. 
\end{proof}

\begin{prop}\label{423}
Let $X\in \mmm(F)$ with $M$ a standard Levi subgroup of $G$ and $(a_M, t)\in \widetilde{\mathcal{A}}_{M}(\mathbb{F}_q)$. 
Suppose that $$ \chi_M(X)=a_{M, \eta}. $$

Then $X$ is elliptic in $\mmm(F)$ (equivalently $M$ is the minimal Levi subgroup defined over $\mathbb{F}_q$ such that $X\in \mmm(F)$) if and only if   $X$ is elliptic in   $\mmm(F\otimes\overbar{\mathbb{F}}_q)$ if and only if $(a_M, t)\in \widetilde{\mathcal{A}}_{M}^{\mathrm{ell}}(\mathbb{F}_q) $. 
\end{prop}
\begin{proof}
If $(a_M, t)\in \widetilde{\mathcal{A}}_{M}^{\mathrm{ell}}(\mathbb{F}_q)$, then by Lemma \ref{rational'} (cf. \cite[3.9.1]{C-L1}) and \cite[3.9.2]{C-L1}, we know that that $X$ is elliptic over $F\otimes \overbar{\mathbb{F}}_q$. Therefore it is elliptic over $F$.

Consider the other direction. Suppose $X$ is elliptic in $\mmm(F)$.  By \cite[3.6.2]{C-L1}, there is a Levi-subgroup $M'$ of $M$ defined over $\overbar{\mathbb{F}}_q$ such that $(a,t)=\chi_G^M(a_M,t)$ is  contained in $\widetilde{\mathcal{A}}_{M'}^{\mathrm{ell}}(\overbar{\mathbb{F}}_q)$. Since $(a,t)$ is fixed by the Frobenius element $\tau$ and by uniqueness of $M'$, we see that $M'$ is defined over $\mathbb{F}_q$ and $(a,t)\in \widetilde{\mathcal{A}}_{M'}^{\mathrm{ell}}(\mathbb{F}_q)$. By Lemma \ref{rational} and \cite[3.8.1]{C-L1}, there is an $X'\in \mmm'(F)$ such that $X$ and $X'$ are stably conjugate. Hence by Lemma \ref{lev}, there is an $X''\in \mmm'(F)$ that is $G(F)$-conjugate to $X$. Since $X$ is elliptic in $\mmm(F)$, it shows that $M=M'$. This means that $(a_M,t)\in \widetilde{\mathcal{A}}_{M}^{\mathrm{ell}}(\mathbb{F}_q)$. 
\end{proof}

\begin{prop}\label{rationality}

Moreover, for any $h\in M(F\otimes\overbar{\mathbb{F}}_q)\backslash M(\AAA\otimes \overbar{\mathbb{F}}_q)/M(\AAA)$ such that $h\tau(h)^{-1}\in M(F\otimes\overbar{\mathbb{F}}_q)$, 
the double coset represented by  $h$ contains the trivial element. 
\end{prop}
\begin{proof}

Suppose that $h\in M(\AAA\otimes \overbar{\mathbb{F}}_q)$ is an element such that $$h\tau(h)^{-1}\in M(F\otimes \overbar{\mathbb{F}}_q). $$
We get a cocycle in $$\ker(H^{1}(F, M)\rightarrow  \prod_{v}H^{1}(F_v ,  M)). $$
However, this kernel is trivial (Proposition \ref{Hasse}), so after replacing $h$ by an element $mh$ for some $m\in M(F\otimes \overbar{\mathbb{F}}_q)$ we can force $h=\tau(h)$. This implies that $h\in M(\AAA)$.
\end{proof}

\subsubsection{The Hitchin moduli stacks}
For the reader's convenience and to fix notation, we introduce the Hitchin moduli stacks and Chaudouard-Laumon's $\xi$-stable \'etale open substack.
We refer the reader to \cite[Chapitre 4]{Ngo} and \cite[Section 4 - Section 6]{C-L1} for more details.

Let $S$ be a finite set of closed points of $X$ containing a point $\infty$ of degree $1$ as before, and \[D=K_X+\sum_{v\in S}v. \] Let $\mathcal{M}_G$ be the algebraic stack of the Hitchin pairs parametrizing $G$-torsors $\mathcal{E}$ over $X$ together with a section of $\ad(\mathcal{E})\otimes \mathcal{O}_X(D)$ over $X$. 
We have a morphism (\cite[4.2.3]{Ngo}), called the Hitchin fibration: \[f:\mathcal{M}_G\longrightarrow \mathcal{A}_G. \] 
Let $\widetilde{\mathcal{M}}_G$ be the base change of $\mathcal{M}_G$ along the Hitchin fibration from $\mathcal{A}_G$ to $\widetilde{\mathcal{A}}_G$. 


Let $\xi\in \ago_T=\Hom(X^{*}(T), \mathbb{R})$ be any vector. Chaudouard and Laumon have defined (\cite[6.1.1, 6.1.2]{C-L1}) a stability condition on \(\widetilde{\mathcal{M}}_G\).  When $\xi$ is in general position as defined in Definition \ref{gepo}, being semistable is equivalent to being stable, and in this case, the stable part \(\widetilde{\mathcal{M}}_G^{\xi}\) is a separated Deligne-Mumford stack which is also smooth and of finite type. Moreover, the induced Hitchin morphism 
\[\widetilde{f}^{\xi}:  \widetilde{\mathcal{M}}^{\xi}_G\longrightarrow \widetilde{\mathcal{A}}_G, \]
is proper (\cite[6.2.2]{C-L1}).

Let $M$ be a semistandard Levi subgroup of $G$, $\varphi \in \mathcal{C}_c^{\infty}(\ggg(\AAA))$ and $Y\in \mmm(F)$ be an elliptic element, the weighted orbital integral $J_{M}^{\ggg}(Y,  \varphi)$ is defined by
\begin{equation}J_{M}^{\ggg}(Y, \varphi )= \int_{G_Y(\AAA)\backslash G(\AAA)}\mathrm{v}_{M}(x)\varphi(\Ad(x^{-1})Y)\d x , \end{equation}
where $\mathrm{v}_{M}(x)$ is the weight factor of Arthur. Recall that the weight factor is defined to be the volume of the convex envelope in $\ago_{M}$ spanned by $(-H_{P}(x))_{P\in \mathcal{P}(M)}$. 
We have not chosen a Haar measure on $G_Y(\AAA)$ nor on $\ago_M$. But for the result that we will state, it does not depend on these choices if they are chosen compatible. 
We can state \cite[Th\'eor\`eme 11.1.1]{C-L1} in the following form. 
\begin{theorem}[Chaudouard-Laumon]\label{C-L1}
Suppose that $\xi$ is in general position. 
Suppose $n_v$ are integers such that $D=\sum_{v\in |X|}n_v v$, let $\mathbbm{1}_{D}\in \mathcal{C}_c^{\infty}(\ggg(\AAA)) $ be the characteristic function of the set 
\[\prod_{v\in |X|}\wp_v^{-n_v}\ggg(\mathcal{O}_v),\] where $\wp_v$ is the maximal ideal of $\mathcal{O}_v$. 

Let $(a_M,t)\in \widetilde{\mathcal{A}}_{M}^{\mathrm{ell}}(\mathbb{F}_q)$ and $(a,t)=\chi_G^{M}((a_M,t))$. The cardinality of the groupoid \[\widetilde{f}^{\xi, -1}(a,t)(\mathbb{F}_q)\] of the rational points of the fiber $\widetilde{f}^{\xi}$ in the point $(a, t)$ is equal to 
$$\frac {\vol(G_{X}(F)\backslash G_{X}(\AAA)^{1})}{\vol(\ago_{M}/X_{*}(M) )} \sum_{X}J_M^\ggg(X, \mathbbm{1}_{D}),  $$
where the sum over $X$ is taken over the set \[\{X\in \mmm(F) \mid  \chi_M(X)= a_{M, \eta}\}.\]
\end{theorem}

\subsection{A residue morphism on the Hitchin stack}
\subsubsection{}
Recall that we have introduced in \ref{Hitchinbase} the scheme $\mathcal{R}_G$ as well as the \'etale open \[\widetilde{\mathcal{R}}_G=\ttt_{D,\infty}^{\reg}\prod_{v\in S-\{\infty\}}\mathrm{R}_{\kappa_v|\mathbb{F}_q}\car_{G, D,v}. \] We have a morphism that we call the evaluation morphism: 
\[\widetilde{\ev}_G: \widetilde{\mathcal{R}}_G\longrightarrow \widetilde{\mathcal{A}}_G.\]
It is flat and surjective if $|\overbar{S}|>2-g$ (Proposition \ref{ARs}).
We define the composition of $\widetilde{\ev}$ with the Hitchin fibration $f: \widetilde{\mathcal{M}}_G\longrightarrow \widetilde{\mathcal{A}}_G$ as the residue morphism: \[\widetilde{\res}_G: \widetilde{\mathcal{M}}_G\longrightarrow \widetilde{\mathcal{R}}_G.\] Given $o\in \widetilde{\mathcal{R}}_G(\overbar{\mathbb{F}}_q)$, we denote $\widetilde{\res}_G^{-1}(o)$ by $\widetilde{\mathcal{M}}_G^{\xi}(o)$. It has an $\mathbb{F}_q$-structure if $o\in \widetilde{\mathcal{R}}_G({\mathbb{F}}_q)$.

We are interested in the cardinality of $\mathbb{F}_q$-points of $\widetilde{\mathcal{M}}_G^{\xi}(o)$ for some $o\in \widetilde{\mathcal{R}}_G(\mathbb{F}_q)$. 
To estimate its cardinality of $\mathbb{F}_q$-points. We have to understand their irreducible components. The following theorem will be proved using the Lie algebra version of the Arthur-Selberg trace formula, and we need it later in Theorem \ref{Faltings}. 
\begin{theorem}\label{independent}
Let $\xi$ be in general position. 
For a family of $\mathbb{F}_q$-tori $(T_v)_{v\in S}$ with $T_\infty$ being split, the cardinality of $\widetilde{\mathcal{M}}_{G}^{\xi}(o)(\mathbb{F}_q)$ is independent of $o\in \widetilde{\mathcal{R}}_G(\mathbb{F}_q)$ as long as $o$ lies in the image of $ \prod_{v\in S}\ttt_v^{\mathrm{reg}}(\mathbb{F}_q)\rightarrow \widetilde{\mathcal{R}}_G(\mathbb{F}_q)$. 
\end{theorem}
A proof of this theorem will be provided in Section \ref{proofind}. The reader can verify that the proof is not circular. In fact, its proof relies only on results of Section \ref{4.1}-\ref{4.2}.

\subsubsection{}
As explained in the beginning of this section, the purpose of the following theorem is to avoid using the Kostant-Hitchin section constructed by Ngô. 
\begin{theorem}\label{surjective}
The morphism $f:\mathcal{M}_G^{\heartsuit}\longrightarrow \mathcal{A}_G^{\heartsuit}$ is surjective and flat.
\end{theorem}
\begin{proof}
Without using a Kostant-Hitchin section constructed by Ngô, we can still use a Kostant section. It suffices to show that for every $a\in \mathcal{A}_G^{\heartsuit}(\overbar{\mathbb{F}}_q)$, the fiber $\mathcal{M}_{a}$ is non-empty. We can fix a Kostant section $\epsilon: \car_G \rightarrow \ggg$. Identify $\mathcal{A}_G^{\heartsuit}(\overbar{\mathbb{F}}_q)$ with a subset of $\car_G^{\mathrm{reg}}(F\otimes\overbar{\mathbb{F}}_q)$, the set consisting of regular semisimple characteristics. The element $\epsilon(a)\in \ggg(F\otimes\overbar{\mathbb{F}}_q)$ will be conjugate by an element in $x\in G(\AAA\otimes \overbar{\mathbb{F}}_q)$ to an element in \[\wp^{-D}\ggg(\prod_{\overbar{v}\in | X\otimes \overbar{\mathbb{F}}_q| }\ooo_{\overbar{v}}), \] where \[\wp^{-D}=\prod_{\overbar{v}}\wp_{\overbar{v}}^{-n_{\overbar{v}}}, \] and \( D=\sum n_{\overbar{v}}\overbar{v}\) ($| X\otimes \overbar{\mathbb{F}}_q|$ is the set of closed points of $ X\otimes \overbar{\mathbb{F}}_q$). The point is that the Kostant section is defined over $\mathbb{F}_q$, so it sends $\car_G(\ooo_{\overbar{v}})$ to $\ggg(\ooo_{\overbar{v}})$. While for any $t\in F_{ \overbar{v}}$ (in particular for a power of uniformizer), we know that  $t^{-1}\epsilon (t.a)$ has the same characteristic as $\epsilon(a)$. So they are conjugate by an element in $G(F_{ \overbar{v}})$ (see \cite[3.8.1]{C-L1}).  Note that the element $(x, \epsilon(a))$ defines an element in $\mathcal{M}_a(\overbar{\mathbb{F}}_q)$ using Weil's dictionary.

The flatness of $f$ is shown in \cite[4.16.4]{Ngo}. In fact, the proof uses \cite[4.3.3]{Ngo} where a Kostant-Hitchin section is needed to show the non-emptiness of the fibers of the Hitchin fibration. It then implies that the regular part $\mathcal{M}_G^{\reg}$ of $\mathcal{M}_G$ a torsor under a Picard stack $\mathcal{P}$. The flatness comes from the miracle flatness theorem and dimension calculations. Therefore being surjective implies that $f$ is flat. 
\end{proof}
\begin{remark}\normalfont
There is another way to prove the surjectivity hence the flatness. 
In \cite{Faltings}, Faltings has showed that the fibers of the morphism $\mathcal{M}_G^{\mathrm{\diamondsuit}}\rightarrow \mathcal{A}_G^{\mathrm{\diamondsuit}}$ is non-empty (\cite[p.319]{Faltings} and \cite[Proposition 11]{Faltings}, see also \cite[Theorem III.2.(ii)]{Faltings0}).
His statement is made for the case of a simply connected, semisimple group and with the canonical divisor, but it works without these restrictions. 
Besides, the morphism  \[\widetilde{f}^{\xi}:  \widetilde{\mathcal{M}}_G^{\xi}\longrightarrow \widetilde{\mathcal{A}}_G\]
is proper (\cite[6.2.2]{C-L1}). 
Therefore, the image of $\widetilde{f}^{\xi}$ is a closed subscheme containing the open subset $\widetilde{\mathcal{A}}_G^{\diamondsuit}$, which is non-empty (Proposition \ref{nonvide}). By irreducibility of $\widetilde{\mathcal{A}}_G$ (Proposition \ref{irred}), 
we deduce that $\widetilde{f}^{\xi}$ is surjective. By varying the point $\infty$, we also get the desired surjectivity. 
\end{remark}

The essential part of the following result is due to Faltings and Ngô's study on the Hitchin fibers.
\begin{theorem}\label{Faltings}
Suppose that $\xi$ is in general position. 
Suppose that $|\overbar{S}|> 2-g$. 
For any $o\in \widetilde{\mathcal{R}}_G(\mathbb{F}_q)$, the stack $\widetilde{\mathcal{M}}_G^{\xi}(o)$ is equidimensional and has $|\pi_1(G)|$ connected components. Moreover if every factor of $o$ is regular, then each connected component of $\widetilde{\mathcal{M}}_G^{\xi}(o)$ is geometrically irreducible of dimension $\dim \mathcal{M}_G-\dim\mathcal{R}_G= \dim \ggg \deg D- \dim \ttt\deg S$. 
\end{theorem}
\begin{proof}
Now we apply Faltings arguments to calculate the number of connected components.

Let $\mathcal{P}$ be the Picard stack over $\mathcal{A}_G$ introduced by Ngô in \cite[4.3.1]{Ngo}. It acts on $\mathcal{M}_G$. For every $a\in {\mathcal{A}}_G^{\diamondsuit}(\overbar{\mathbb{F}}_q)$, $\mathcal{M}_a$ is a $\mathcal{P}_a$-torsor (\cite[4.3]{NgoH}). 
We know that (see \cite[4.10.3, 4.10.4]{NgoH}, or \cite[III.2.(iv)]{Faltings0}) for every $a\in {\mathcal{A}}_G^{\diamondsuit}(\overbar{\mathbb{F}}_q)$, one has a canonical isomorphism \[\pi_0(\mathcal{P}_a) \cong X_{*}(T)/\mathbb{Z}\Phi(G, T)^{\vee}. \]
There is also a canonical isomorphism (it can be obtained from \cite[Proposition 5]{DS} for simply connected case) \[\pi_0(\mathrm{Bun}_G)\rightarrow X_{*}(T)/\mathbb{Z}\Phi(G, T)^{\vee}. \] 
So we have an isomorphism from $\pi_0(\mathcal{M}_a)$ to $ X_{*}(T)/\mathbb{Z}\Phi(G, T)^{\vee}$ that is compatible with the action of $\mathcal{P}_a$.

For any $\vartheta\in  X_{*}(T)/\Phi(G, T)^{\vee}$, let $\mathrm{Bun}_G^{\vartheta}$ be the corresponding connected component of $\mathrm{Bun}_G$. Let $\mathcal{M}^{\vartheta}$ the open substack of $\mathcal{M}$ consisting of Hitchin pairs $(\mathcal{E}, \varphi)$ with the underling $G$-torsor $\mathcal{E}$ being in $\mathrm{Bun}_G^{\vartheta}$. 
The morphism
 \[\widetilde{f}^{\vartheta}: \widetilde{\mathcal{M}}_G^{ \xi, \vartheta}\longrightarrow \widetilde{\mathcal{A}}_G, \] is the restriction of the Hitchin morphism to $\widetilde{\mathcal{M}}^{\xi, \vartheta}_G$. It has geometrically connected fibers over the open subset $\widetilde{\mathcal{A}}^{\diamondsuit}_G$. Observe that $\widetilde{f}^{\vartheta}$ is flat and proper, locally of finite presentation with geometrically reduced fiber (\cite[4.16.4]{Ngo}). By  \cite[15.5.9]{EGA}, the number of geometric connected components of the fibers of $\widetilde{f}^{\vartheta}$ is locally constant in $\widetilde{\mathcal{A}}_G$. By irreducibility of $\widetilde{\mathcal{A}}_G$, every fiber of $\widetilde{f}^{\vartheta}$ is geometrically connected. As $\widetilde{f}^{\vartheta}$ is open, $\widetilde{\mathcal{M}}_G^{ \xi, \vartheta}$ is connected as well. 
Therefore we have the desired number of connected components because $|\pi_1(G)|= |X_{*}(T)/\mathbb{Z}\Phi(G, T)^{\vee}|$. 

Let  $\widetilde{f}(o)$ be the induced morphism 
\[\widetilde{f}(o): \widetilde{\mathcal{M}}_G^{\xi, \vartheta}(o)\longrightarrow \widetilde{\ev}_G^{-1}(o) ,\] 
obtained by base change.  
Since $\widetilde{f}(o)$ is a flat, surjective and equidimensional morphism (Theorem \ref{surjective} and \cite[4.16.4]{Ngo}), we conclude that $\widetilde{\mathcal{M}}_G^{\xi}(o)$ is equidimensional. 

Since $\widetilde{\mathcal{R}}_G\cap \widetilde{\ev}_G(\widetilde{\mathcal{A}}_G^{\diamondsuit})$ is non-empty, for $n$ divisible enough, there is a point $o'\in (\widetilde{\mathcal{R}}_G\cap \widetilde{\ev}_G(\widetilde{\mathcal{A}}_G^{\diamondsuit}))(\mathbb{F}_{q^n})$. Since $\widetilde{\ev}_G^{-1}(o')\cap \widetilde{\mathcal{A}}_G^{\diamondsuit}\neq \emptyset$, and $\widetilde{\ev}_G^{-1}(o')$ is irreducible, there is a dense open subset of $\widetilde{\ev}_G^{-1}(o')$ so that the fiber of $\widetilde{f}(o')$ is smooth. So these fibers are geometrically irreducible. As $\widetilde{f}(o')$ is open and it has geometric irreducible fibers over a dense subset of the geometric irreducible scheme $\widetilde{\ev}_G^{-1}(o')$, by a simple topological argument, $ \widetilde{\mathcal{M}}_G^{\xi, \vartheta}(o')$ is geometrically irreducible too. 
Therefore, the last statement of the theorem is proved for such $o'$. 

Now, given any $o\in\widetilde{\mathcal{R}}_G(\mathbb{F}_q)$, whose factors are all regular, 
there is a family of maximal torus $(T_v)_{v\in S}$, so that $o$ lies in the image of $\prod_{v\in S}\mathrm{R}_{\kappa_v|\mathbb{F}_q}\ttt_v^{\reg}(\mathbb{F}_q)$ and $T_\infty$ is split over $\kappa_\infty=\mathbb{F}_q$. 
Noting that the morphism from $\ttt_v^{\reg}$ to $\car_G^{\reg}$ is surjective. For $n$ divisible enough, there is an $\mathbb{F}_{q^n}$-point $o'$ in $(\widetilde{\mathcal{R}}_G\cap \widetilde{\ev}_G(\widetilde{\mathcal{A}}_G^{\diamondsuit}))(\mathbb{F}_{q^n})$ that is also in the image of $\mathbb{F}_{q^n}$-point of $\prod_{v\in S}\mathrm{R}_{\kappa_v|\mathbb{F}_q}\ttt_v^{\reg}$.  After Theorem \ref{independent}, $\widetilde{\mathcal{M}}_G^{\xi}(o)_{\mathbb{F}_{q^n}}$ and $\widetilde{\mathcal{M}}_G^{\xi}(o')$ have the same number of geometric irreducible components of dimension \[\dim \widetilde{\mathcal{M}}_G^{\xi}(o)= \dim \mathcal{M}_G-\dim \mathcal{R}_G\]
as the dominant term of the number of $\mathbb{F}_{q^{mn}}$-points of a geometric irreducible variety over $\mathbb{F}_{q^n}$ is $q^{mn}$ when $m\rightarrow \infty$ (see the proof of Corollary \ref{corollary} too). This finishes the proof since $\widetilde{\mathcal{M}}_G^{\xi}(o)$ is equidimensional. 
\end{proof}

\begin{coro}\label{estimate}
Under the hypothesis of Theorem \ref{Faltings}, suppose that every factor of $o$ is regular. 
There is a constant $C$ depending only on the root datum of $G$, on the curve $X\otimes \overbar{\mathbb{F}}_q$ and on the set $S\otimes \bar{\mathbb{F}}_q$  such that 
\[|  | \widetilde{\mathcal{M}}_G^{\xi}(o)(\mathbb{F}_{q}) | - |\pi_1(G)| q^{\dim  \mathcal{M}-\dim \mathcal{R}}|\leq C q^{\dim \mathcal{M}-\dim \mathcal{R}- \frac{1}{2}}. \]
\end{coro}
\begin{proof}

The statement is a corollary of the Grothendieck-Lefschetz fixed point formula and Deligne's purity theorem  \cite{Weil2}. 
These results are for schemes. They are generalized to Artin stacks due to the work of K. Behrend \cite{Behrendladic} and Shenghao Sun's thesis. 
\end{proof}

\section{Sum of multiplicities of cuspidal automorphic forms}\label{summul}
\subsection{The main result}\label{DLi}
Let $\mathcal{C}_{cusp}(G)$ be the space of cusp forms over $G$, i.e., the space of complex-valued functions $\varphi$ over the quotient $G(F)\backslash G(\AAA)$ that is right $G(\ooo)$-finite and their constant terms along all proper parabolic standard subgroups $P$ of $G$ are identically zero.

The space $\mathcal{C}_{cusp}(G)$ is a semisimple $G(\AAA)$-representation:
\[\mathcal{C}_{cusp}(G)=\bigoplus_{\pi} m_{\pi}\pi,   \]
with finite multiplicities $m_\pi$. Given an irreducible smooth representation $\rho$ of $G(\ooo)$, we say that a cuspidal automorphic representation $\pi$ contains $\rho$ if the $(G(\ooo), \rho)$-isotypic  subspace $\pi_{\rho}$ is non-trivial. We are mainly interested in the sum of multiplicities of cuspidal automorphic representations containing $\rho$: 
$$\sum_{\pi: \pi_{\rho}\neq 0}m_{\pi}.   $$

In this article, we only consider a special type of $\rho$. The following definition is introduced in \cite[p.24]{Kaletha}, note that if the group $G$ has connected center, then absolute regularity is equivalent to the general position defined by Deligne-Lusztig (\cite[3.4.15]{Kaletha}). 
\begin{definition}\label{abreg}
For a torus $T_v$ defined over $\kappa_v$, a character $\theta_v$ of $T_v(\kappa_v)$ is called absolutely regular if for some (hence any) finite fields extension $k|\kappa_v$ that splits $T_v$, the associated character $\theta_v\circ N$ of $T_v(k)$ has trivial centralizer in $W^{(G_{\kappa_v},T_v)}(k)$, where $N: k\rightarrow \kappa_v$ is the norm map. 
\end{definition}

Let $S$ be a finite set of places of $F$. For each $v\in S$, fix a maximal torus $T_v$ of $G$ defined over $\kappa_v$, let $\theta_v$ be an absolutely regular character of $T_v(\kappa_v)$. Let $$\rho_v=\epsilon_{\kappa_v}(T_v)\epsilon_{\kappa_v}(G) R_{T_v}^{G}\theta_v$$ be the Deligne-Lusztig induced representation (by choosing an isomorphism between $\mathbb{C}$ and $\overbar{\mathbb{Q}}_\ell$) of $G(\kappa_v)$, where $\epsilon_{\kappa_v}(\cdot)$ equals $(-1)^{r}$ if the rank of a maximal split subtorus over $\kappa_v$ of the group is $r$. It is an irreducible representation (\cite[Proposition 7.4]{DL}). We view $\rho_v$ as a representation of $G(\ooo_v)$ by inflation and consider the representation $\rho=\otimes_{v}\rho_v$ of $G(\ooo)$ (outside $S$, $\rho_v$ is the trivial representation). We say that such a representation $\rho$ is of the Deligne-Lusztig type.

The following proposition is inspired by Deligne's criterion on irreducibility of $\ell$-adic local systems \cite[(2.10.3), Remarque 2.12]{Deligne}. In this article, all the automorphic representations are meant to be irreducible $L^2$-automorphic representations, i.e., (topologically) irreducible subquotients in the $L^2$-spectrum. 
\begin{prop}\label{generic}
Let $(\theta_v)_{v\in S}$ of be a tuple of characters of $(T_v(\kappa_v))_{v\in S}$. 
If the character \[\prod_{v\in S}\theta_v|_{Z_G(\mathbb{F}_q)}\] of $Z_G(\mathbb{F}_q)$  is non-trivial, then for any automorphic representation $\pi$ of $G(\AAA)$, we have \[\pi_\rho=0. \]

Moreover, suppose that for any standard proper Levi subgroup $M$ of $G$ defined over $\mathbb{F}_q$, any tuple of elements $x_v\in G(\kappa_v)$, such that $T_v\subseteq x_vM_{\kappa_v}x_v^{-1}$, the character $\prod_{v\in S}\theta_v^{x_v}|_{Z_{M}(\mathbb{F}_q)}$ is non-trivial. Then the following condition holds: 

$(\ast)$ Let $\pi$ be a constituent of the parabolically induced representation $\Ind_{P(\AAA)}^{G(\AAA)}\sigma$ built from an irreducible automorphic representation $\sigma$ of $M_P(\AAA)$ with $P\subsetneq G$ being a proper standard parabolic subgroup, we have \[\pi_{\rho}=0 .  \]

In particular, for an automorphic representation $\pi$, we have $\pi_{\rho}\neq 0\implies$ $\pi$ is cuspidal. 
\end{prop}

Before we give proof of the proposition, let us state the main theorem of this section and give an outline of its proof first. We call a representation $\rho$ of $G(\ooo)$ a cuspidal filter if the condition $(\ast)$ in the precedent proposition is satisfied. 

\begin{theorem}\label{main}
Suppose that $\prod_{v\in S}\theta_v|_{Z_G(\mathbb{F}_q)}$ is trivial. Moreover, suppose that $T_\infty$ is split and that $\rho$ constructed above is a cuspidal filter. 
Let $o_G$ be any point in the image of the natural map \(\prod_{v\in S}\ttt_v^{\reg}(\kappa_v) \rightarrow  \widetilde{\mathcal{R}}_G(\mathbb{F}_q). \)
 Then for each conjugacy class of split elliptic coendoscopic group $H$ of $G$ and each $o\in \widetilde{\pi}_{H,G}^{-1}(o_G)(\mathbb{F}_q)$, there is an explicit integer $n_{H,o}\in \mathbb{Z}$ which is explicit in terms of $(\theta_v)_{v\in S}$, 
 such that 
\begin{equation}\sum_{\pi: \pi_{\rho}\neq 0}m_\pi = \sum_{H}\sum_{o\in \widetilde{\pi}_{H,G}^{-1}(o_G)(\mathbb{F}_q)}n_{H,o}q^{-\frac{1}{2}(\dim \mathcal{M}_H- \dim \mathcal{R}_H) }  |\widetilde{\mathcal{M}}_{H}^{\xi}(o)(\mathbb{F}_q)|.  \end{equation}
\end{theorem}
The explicit formula for $n_{H,o}$ is given in \eqref{537}. 
This whole section is concentrated on the proof of this theorem, and we will discuss the properties of the numbers $n_{H,o}\in \mathbb{Z}$ in the above theorem later in Proposition \ref{nho}. For the reader's convenience, let us briefly outline the structure of the proof.

\subsubsection{Outline of the proof of the main theorem \ref{main}.}
We use our variant of Arthur's trace formula. It is a distribution $J^{G,\xi}$ on $C_c^{\infty}(G(\AAA))$. For any function $f\in C_c^{\infty}(G(\AAA))$ called test function, we have an identity:
\begin{equation}\label{cdots}J^{G,\xi}(f)=\sum_{\pi}m_{\pi}\Tr(f| \pi ) + \cdots,   \end{equation}
where the sum is taken over the set of cuspidal automorphic representations of $G(\AAA)$ and 
$f$ acts on $\pi$ via the regular representation of $G(\AAA)$ on $C_{cusp}(G(\AAA))$. 
We will choose a test function $f=e_\rho$ constructed form the trace function of $\rho$ as in Proposition \ref{mul2}, so that $\Tr(e_{\rho}|\pi)=1$ if and only if $\pi_{\rho}\neq \{0\}$. The proof of it involves a local calculation where essentially we need to show that a relevant Hecke algebra is commutative (Theorem \ref{commutative}). It is done by Corollary \ref{mul1}.
We prove in Proposition \ref{generic} that the condition $(\ast)$ in Proposition \ref{generic} allows us to kill the extra ``$\cdots$" parts in \eqref{cdots}. Therefore, we obtain:
\[J^{G,\xi}(e_\rho) =  \sum_{\pi: \pi_\rho\neq 0}m_{\pi}.   \] 

Our next task is to relate $J^{G,\xi}(e_\rho)$ with distributions $J^{\ggg,\xi}$, the analogue truncated trace for Lie algebras. The key point is to use the coarse geometric expansion established in \cite{Yu2} and then use Theorem \ref{expansion} to reduce $J^{G,\xi}(e_\rho)$ to essentially unipotent contributions (i.e., in terms of Jordan decomposition, the semisimple part is central) of the relevant groups. A springer isomorphism then easily reduces the unipotent contributions to nilpotent contributions of the Lie algebra. The last step comes from an observation on the Springer's hypothesis: the Fourier transform of a semisimple adjoint orbit coincides with the Green function on unipotent elements. Therefore, via a trace formula of the Lie algebra established in \cite{Yu2}, we will reduce the contribution to a better result. 

However, there are some technical difficulties from the validity of Theorem \ref{expansion} and Theorem \ref{expansion'}, and from the parameter $\xi$. The first difficulty is bypassed if we set the torus $T_\infty$ to be split and introduce another test function $\widetilde{e}_\rho$ which does not change the value of the distribution $J^{G,\xi}$. It is interesting to see that this operation is similar to Ngô's introduction of the \'etale cover $\widetilde{\mathcal{A}}$ of the Hitchin base $\mathcal{A}$. 
For difficulties raised by $\xi$, we need to deal with the Lie algebra first to see the independence of $\xi$  (Corollary \ref{indepxi}), then we can deal with the group case easier.

\subsubsection{Proof of Proposition \ref{generic}}\label{5111}
The key point is very simple and can be summarized as follows. 

For any reductive group $M$ defined over $\mathbb{F}_q$. Let $\sigma$ be an irreducible automorphic representation of $M(\AAA)$ with central character $\chi_{\sigma}$, then $\chi_{\sigma}|_{Z_M(F)}$ is trivial, in particular 
$\chi_{\sigma}|_{Z_M(\mathbb{F}_q)}$ is trivial. Besides, the central character of $\chi_{\sigma}$ is also the product of those of local components of $\sigma$.

Now we provide more details. After Langlands \cite[Proposition 2]{Langlands}, every automorphic representation $\pi$ is an irreducible sub-quotient of $\Ind_{P(\AAA)}^{G(\AAA)}\sigma$, for a standard  parabolic subgroup $P$ and a cuspidal automorphic representation $\sigma$ of $M_P(\AAA)$. 

If $\pi $ is not cuspidal, then we may suppose that $P$ is a proper parabolic standard subgroup $P$ of $G$. 

If $\pi_{\rho}\neq 0$, then for any $v\in S$, we have $\pi_{v, \rho_v}\neq 0$ and $\pi_v$ is unramified for $v$ outside $S$.  Let $v\in S$.
After Moy-Prasad, the supercuspidal support of   $\sigma_v$ can be determined in the following way. Let $M_v$ be the minimal Levi subgroup of $G_{\kappa_v}$ defined over $\kappa_v$ containing $T_v$. Let $\tau_v$ be an extension of $\rho_v$ to $A_{M_v}(F_v) M_v(\ooo_v)$, which is the normalizer of $M_v(\ooo_v)$ in $M_v(F_v)$ by Cartan decomposition. 
The representation $\phi_v=\cInd_{A_{M_v}(F_v) M_v(\ooo)}^{M_v(F_v)}(\tau_v)$ is irreducible and supercuspidal (\cite[Proposition 6.6]{MP2}).
It is showed in \cite[Theorem 6.11(2)]{MP2},  that  for some extesion $\tau_v$ of $\rho_v$ to $A_{M_v}(F_v) M_v(\ooo_v)$, the $G(F_v)$-conjugacy class $[(M_v, \phi_v)]$ is the supercuspidal support of $\pi_v$.

For each place $v\in S$, the local component $\pi_v$ is a subquotient of $Ind_{P(F_v)}^{G(F_v) }\sigma_v$. Therefore there is an element $x_v\in G(F_v)$ such that $x_v^{-1}M_v x_v \subseteq M_P$ and $\sigma_v$ has supercuspidal support $[(x_vM_v x_v^{-1}, \phi_v^{x_v})]$. Since $P$ is defined over $\mathbb{F}_q$ and $M_v$ is defined over $\kappa_v$, we can take $x_v\in G(\kappa_v)$. 
The central character $\chi_{\sigma_v}$ of $\sigma_v$ is the restriction to $Z_{M_P}(F_v)$ of the central character of $\phi_v^{x_v}$. It equals
$\chi_{\tau_v}^{x_v}$. Its restriction to $Z_{M_P}(\kappa_v)$ equals to $\theta_{v}^{x_v}|_{Z_{M_P}(\kappa_v)}$. In particular, $$  \chi_{\sigma_v}|_{Z_{M_P}(\mathbb{F}_q)}=\theta_{v}^{x_v}|_{Z_{M_P}(\mathbb{F}_q)}. $$
 However, we also have $$1=\chi_{\sigma}|_{Z_{M_P}(\mathbb{F}_q)}=\prod_{v}\chi_{\sigma_{v}}|_{Z_{M_P}(\mathbb{F}_q)},$$
where $\chi_{\sigma_v}$ is the central character of local component $\sigma_v$ of $\sigma$. This contradicts our assumption if $P\neq G$. 

\subsection{Local calculations through unrefined depth zero types}

The main result in this section is to prove Proposition \ref{mul2}. 
We need first to have some preparations on the Hecke algebras. 

\subsubsection{Some subgroups}\label{BT}
To facilitate some statements and proofs, we must introduce some subgroups of $G(F_v)$. 
Denote by $\mathcal{B}(G)_v$ the Bruhat-Tits building of $G(F_v)$. For every point $x\in \mathcal{B}(G)_v$, we denote by $G(F_v)_{x}\subseteq  G(F_v)$ (resp. $\ggg(F_v)_{x}\subseteq \ggg(F_v)$)
the corresponding parahoric subgroup (resp. subalgebra). We define  $ G(F_v)_{x+}\subseteq G(F_v)_{x}$ (resp. $\ggg(F_v)_{x+}\subseteq \ggg(F_v)$) to be the pro-unipotent (resp. pro-nilpotent) radical of $G(F_v)_{x}$ (resp. of $\ggg(F_v)_{x}$). For the Iwarhori subgroup $\mathcal{I}_{\infty}$ of $G(F_{\infty})$, we denote $\mathcal{I}_{\infty+}$ its pro-unipotent radical.

\subsubsection{}

In this subsection, we do some local calculations. Let us fix a place $v$ of $F$. 

For a compact open subgroup $K$ of $G(F_v)$ and a finite-dimensional irreducible smooth complex representation $\rho$ of $K$. Let  $\Theta_{\rho}: K\rightarrow \mathbb{C}$ be the trace function of $\rho$. 
Let $$e_\rho(x):=\begin{cases}
\frac{1}{\vol(K)}\Theta_{\rho}(x^{-1}), \quad x\in K;\\
0, \quad x\notin K. 
\end{cases}$$
Then $({\dim \rho}) e_{\rho}$ is an idempotent under the convolution product.
 
We will only consider the case that $K=G(\ooo_v)$ or $K=\mathcal{I}_{v}$, the standard Iwahori subgroup. 

Let $\mathcal{H}(G(F_v), \rho):= \End_{G(F_v)}(\cInd_{G(\ooo_v)}^{G(F_v)}\rho)$. It is well known that $\mathcal{H}(G(F_v), \rho)$ is finitely generated.  In fact, let $\mathcal{H}(G(F_v))$ be the local Hecke algebra of $G(F_v)$ which is the space $\mathcal{C}_c^{\infty}(G(F_v))$ equipped with the convolution product $\ast$.  
Then (see \cite[(2.12)]{BK}) $$\mathcal{H}(G(F_v), \rho)\otimes \End_{\mathbb{C}}(V_{\rho})\cong e_{\rho}\ast\mathcal{H}(G(F_v))\ast e_{\rho}, $$ where $V_{\rho}$ is the underlying vector space of $\rho$. As $e_{\rho}\ast\mathcal{H}(G(F_v)\ast e_{\rho}$ is finitely generated as $\mathbb{C}$-algebra (with $1$ but not necessarily commutative), $\mathcal{H}(G(F_v), \rho)$ is also finitely generated. Usually the algebra $\mathcal{H}(G(F_v), \rho)$ is not commutative.  However, we have the following result, which will be crucial for us. This theorem is well known when $T'$ is maximally anisotropic. 
\begin{theorem}\label{commutative}
Let $T'$ be a torus of $G$ defined over $\kappa_v$ and $\theta$ be an absolutely regular character of $T'(\kappa_v)$. 
Let $\rho$ be the inflation to $G(\ooo_v)$ of the Deligne-Lusztig induced representation $\epsilon_{\kappa_v}(T')\epsilon_{\kappa_v}(G) R_{T'}^G(\theta)$ of $G(\kappa_v)$. Then $\mathcal{H}(G_v, \rho)$ is commutative. 
\end{theorem}
\begin{proof}
Morris has given an explicit description for this kind of Hecke algebras in a more general setting by generators and relations (see \cite[7.12 Theorem]{Morris}). To apply his result, one still needs to calculate some subgroups of an affine weyl group. We show a proof based on some of his results and also on \cite{BK} so that the calculation are more direct.

Let $P$ be a minimal parabolic subgroup defined over $\kappa_v$ containing $T'$ and $M$ its $\kappa_v$-Levi subgroup containing $T'$. Note that we can take $M$ as the centralizer of the maximal $\kappa_v$-split subtorus of $T'$. After conjugation, we may suppose that $P$ is standard. 
Let $\mathcal{P}\subseteq G(\ooo_v)$ be the inverse image of $P(\kappa_v)$ of the map $G(\ooo_v)\rightarrow G(\kappa_v)$. 
Let $\tau$ be the inflation to $\mathcal{P}$ of the Deligne-Lusztig induced representation  $\epsilon_{\kappa_v}(T') \epsilon_{\kappa_v}(G) R_{T'}^{M}(\theta)$ of $M(\kappa_v)$. By definition and composition property of compact induction (and that the Deligne-Lusztig induction $R_M^G$ coincides with Harish-Chandra's induction), we have $$\mathcal{H}(G(F_v), \rho)\cong \mathcal{H}(G(F_v), \tau).$$ 
Let $\tau_M:=\tau|_{M(\ooo_v)}$ be the restriction of $\tau$ to $M(\ooo_v)$, then it is just the inflation of the representation $$\overbar{\tau}=\epsilon_{\kappa_v}(T') \epsilon_{\kappa_v}(G) R_{T(\kappa_v)}^{M(\kappa_v)}(\theta)$$ to $M(\ooo_v)$, hence irreducible.
Since $T'$ is elliptic in $M$ and $\theta$ is in general position, the representation $\overbar{\tau}$ is cuspidal. 
It is a well known that $\mathcal{H}(M(F_v), \tau_M)\cong \mathbb{C}[A_M(F_v)M(\ooo_v)/M(\ooo_v)]$ where $A_M$ is the $F_v$-maximal split subtorus in the center of $M$ (see the proof of \cite[Proposition 6.6]{MP2}). 

We are going to show that $\mathcal{H}(M(F_v), \tau_M)\cong \mathcal{H}(G(F_v), \tau).$  In \cite{BK}, Bushnell and Kutzko have developed a general framework for these kinds of results. However, their criteria are not easy to verify. We shall use their intermediate result \cite[Theorem(7.2)(ii)]{BK} instead. 
We recall that our group $G$ is unramified over $F_v$ since it is defined over the finite base field $\mathbb{F}_q$, the maximal compact subgroup $G(\ooo_v)$ gives a hyperspecial point (in particular a special point) in the extended Bruhat-Tits building of $G$, denoted by $0$. Moreover, the choice of $B$ gives rise to a fixed alcove in the building and a choice of basis of affine roots. We shall use the language of Bruhat-Tits buildings and affine Weyl groups for the rest of the proof. A reader unfamiliar with these is invited to \cite[1.1-1.4]{Morris} or \cite{Tits1} for introductions.

Let $\mathcal{A}=\mathcal{A}(G,T)$ be the apartment associated to $T$ and $V=X_*(T)_{F_v}\otimes \mathbb{R}$ be the vector space that acts on $\mathcal{A}$. 
We denote by $W_a=W^{(G, T)}(F_v)\ltimes X_*(T)_{F_v}$ the extended affine Weyl group of the affine Weyl group $W_a'= W^{(G, T)}(F_v)\ltimes \mathbb{Z}\Phi(G,T)_{F_v}^{\vee}$. We have a surjection  $$N_G(T)(F_v)\rightarrow W_a,$$ 
whose kernel is denoted by $Z_c$ 
and
 $N_G(T)(F_v)$ acts on the apartment $\mathcal{A}$ via its projection to $W_a$ (see \cite[Section 1.2]{Tits1}).
 Let $\overbar{n}\in W_a$ be the image of $n$ under this map. 
 Let $D$ be the gradient map, i.e. for an affine function $f$ on $\mathcal{A}$, we have $Df\in \Hom_{\mathbb{R}}(V, \mathbb{R})$ and $$f(x+v)=f(x)+(Df)(v), \quad \forall x\in \mathcal{A}, v\in V. $$
More generally, for an affine transformation $w$, $Dw$ is the ``derivative" of $w$. Note that $W_a$ acts on affine functions on $\mathcal{A}$ by taking bull-back. It is clear by definition that $$D(wf)=DwDf $$ for any $w\in W_a$ and any affine function $f$  on $\mathcal{A}$. 
Since $W^{(G, T)}(F_v)$ is identified with the subset of $W_a$ which fixes the special point $0$, 
the map $D$ is a surjection from $W_a$ to $W^{(G, T)}(F_v)$.

Let us show that any element $n\in G(F_v)$ that intertwines $\tau$, i.e. $$\Hom_{\mathcal{P}\cap{{}^{n}\mathcal{P}} }(\tau, {}^{n}\tau)\not= 0, $$
is contained in the set $\mathcal{P}M(F_v)\mathcal{P}$. 
By Bruhat decomposition, we can suppose without loss of generality that $n\in N_G(T)(F_v)$, recall that $N_G(T)$ is the normalizer of the maximal torus $T$. Thus $n$ stabilizes the apartment $\mathcal{A}$. 
We are going to apply \cite[4.15.Theorem]{Morris}. We need to make clear how $n$ acts on $M$. Let $J$ be a subset in the basis of the affine root system that corresponds to the parahoric subgroup $\mathcal{P}$. Note that elements in $J$ vanish at the point $0$ and can be identified with a basis of the root system $\Phi(M, T)_{\kappa_v}$ (via the gradient map $D$). 
After Morris's Proposition in  \cite[4.13]{Morris}, we know that $\overbar{n}$ fixes $J$.
Since $0$ is special, there is an element $w\in W_a$ that fixes $0$ such that $Dw=D\overbar{n}$. 
So for any $\alpha\in J$, we have $D(w\alpha)=(Dw)(D\alpha)$. 
Then $w$ also fixes $J$ because for any $\alpha\in J$, we have $D(w\alpha)=(Dw)(D\alpha)$ and $w\alpha$ vanishes at $0$.

Let $w'= w^{-1}\overbar{n}$, then $w'\in \ker D$. Hence it is a translation. But $w'J=J$, so $w'$ is the translation along a vector parallel to the intersection of the hyperplanes where $\alpha\in J$ vanishes. 
Therefore under the projection $N_G(T)(F_v)\rightarrow W_a$, $w'$ can be represented by an element $n'\in A_M(F_v)$. Note that $n'$ acts trivially on the reductive quotient $M(\kappa_v)$, we conclude that  $n.n'^{-1}$ acts on $\overbar{\tau}$ as the Weyl element $D(n.n'^{-1})$ and fixes $\overbar{\tau}$ (\cite[4.15.Theorem]{Morris}). 
 So $(w^{-1}T'w, \theta^w)$ and $(T', \theta)$ must be geometrically conjugate in $M$ (\cite[Corollary 6.3]{DL}), i.e. for some finite fields extension $k|\kappa_v$, the pairs $(w^{-1}T'w, \theta^w\circ N_{k|\kappa_v})$ and  $(T', \theta\circ N_{k|\kappa_v})$ are $M(k)$-conjugate. We may assume that the extension is large enough so that $T'_k$ is split. However, $\theta$ is absolutely regular. It means that $\theta\circ N_{k|\kappa_v}: T'(k)\rightarrow \mathbb{C}^{\times}$ is in general position, therefore $w\in W^M$.

Now we can apply \cite[Theorem(7.2)(ii)]{BK} to complete the proof. In fact, given a strongly $(P, \mathcal{P})$-positive element $\zeta$ (defined in \cite[(6.16)]{BK}, such an element is shown to exist in \cite[2.4, 3.3]{Morris2}) so that $\mathcal{P}\zeta\mathcal{P}$ supports an invertible element of $\mathcal{H}(G(F_v), \tau)$, Bushnell and Kutzko have constructed an algebra homomorphism $\mathcal{H}(M(F_v), \tau_M)\rightarrow \mathcal{H}(G(F_v), \tau)$. In our case, i.e., the case where any element $n$ intertwining $\tau$ is contained in $\mathcal{P}M(F_v)\mathcal{P}$, they showed that the homomorphism is an isomorphism. 
\end{proof}


\begin{coro}\label{mul1}
Let $\rho$ be as in Theorem \ref{commutative}. 
Let $\pi$ be any irreducible representation of $G(F_v)$, then $$\Tr(e_{\rho}|\pi)=\dim_{\mathbb{C}}\Hom_{G(\mathcal{O}_v)}(\rho, \pi|_{G(\mathcal{O}_v)})\leq 1. 
$$
\end{coro}
\begin{proof}
Let $$e_{0+}=\frac{1}{\vol(G(\mathcal{O})_{+})}\mathbbm{1}_{G(\mathcal{O}_v)_{+}}. $$
Then $e_{0+}$ is an idempotent and $e_{0+}\ast e_{\rho}=e_{\rho}$. 
Since $e_{0+}$ acts on $\pi$ as a projection onto its $G(\mathcal{O}_v)_{+}$-fixed part, we have 
$$\Tr(e_{\rho}|\pi)=\Tr(e_{\rho}|\pi^{G(\mathcal{O}_v)_{+}}).  $$
Let $\Theta_{\pi}$ be the trace function of the representation $(\pi^{G(\mathcal{O}_v)_{+}})|_{G(\mathcal{O}_v)}=(\pi|_{G(\mathcal{O}_v)})^{G(\mathcal{O}_v)_{+}}$, then 
$$\Tr(e_{\rho}|\pi^{{G(\mathcal{O}_v)_{+}}})= \frac{1}{|G(\kappa_v)|}\sum_{x\in G(\kappa_v)} \Theta_\rho(x^{-1})\Theta_{\pi}(x)=\dim_{\mathbb{C}}\Hom_{G(\mathcal{O}_v)}(\rho, \pi|_{G(\mathcal{O}_v)}). $$

 The association $\pi\mapsto \Hom_{G(\mathcal{O}_v) }(\rho, \pi|_{G(\mathcal{O}_v)})$ gives a functor from the category of smooth representations of $G(F_v)$ to that of $\mathcal{H}(G(F_v), \rho)$-modules. It sends an irreducible representation to a simple module. As $\mathcal{H}(G(F_v), \rho)$ is a finitely generated commutative $\mathbb{C}$-algebra, its simple modules are $1$-dimensional. 
 \end{proof}
 
\subsubsection{Expressing the sum of multiplicities in terms of the truncated trace}\label{proofmul2}
 \begin{prop}\label{mul2}
Let $\rho=\otimes_{v}{\rho_v}$ be a representation of Deligne-Lusztig type of $G(\ooo)$ as defined in 
the subsection \ref{DLi} with each $\theta_v$ absolutely regular. We define 
$$e_{\rho}= \bigotimes_{v\in S} e_{\rho_v}\otimes \bigotimes_{v\notin S}\mathbbm{1}_{G(\ooo_v)},$$
to be the tensor product of the characteristic function of ${G(\ooo_v)}$ for $v\notin S$ with $e_{\rho_v}$ defined below for $v\in S$:
$$e_{\rho_v}(x):=\begin{cases}
\Theta_{\rho_v}(\overbar{x}^{-1}), \quad x\in G(\ooo_v);\\
0, \quad x\notin G(\ooo_v); 
\end{cases}$$
where $\Theta_{\rho_v}$ is the trace function (irreducible character) of the representation $\rho_v$ of $G(\kappa_v)$ and $\overbar{x}$ denotes the image of $x$ via the reduction map $G(\ooo_v)\rightarrow G(\kappa_v)$,  

Then for any $\xi\in \ago_{B}$, if $\rho$ is a cuspidal filter (i.e. satisfies the condition $(\ast)$ in Proposition \ref{generic}),
we have
$$\frac{1}{\dim \rho}L^{2}(G)_{\rho}=\sum_{\pi: \pi_{\rho}\neq 0}m_{\pi}= J^{G, \xi}(e_\rho).  $$ 
\end{prop}
\begin{remark}\normalfont
From this, we see that $J^{G,\xi}$ is independent of $\xi$. However, each geometric expansion $J^{G,\xi}_o$ for $o\in \mathcal{E}$ probably depends on $\xi$. 
\end{remark}
\begin{proof}
After \cite[Section 5, Proposition 9.1]{Yu2}, Corollary \ref{mul1}, and Proposition \ref{generic}, this result is surely obvious for the experts in automorphic forms. We give a brief proof for completeness and convenience to those unfamiliar with the trace formula. 

Let us recall the theory of the pseudo-Eisenstein series. For details, see \cite[Chapter 2]{MW}, especially  II.1.2, II.1.3, II.1.10, II.1.12 and II.2.4. 
It tells us that the $L^2$-space \[L^{2}(M_Q(F)N_Q(\AAA)\backslash G(\AAA))\] is the Hilbert direct sum of subspaces 
\[L^{2}(M_Q(F)N_Q(\AAA)\backslash G(\AAA))_{ [(M_P, \pi)] }\]
for all equivalent classes of cuspidal pairs $(M_P, \pi)$ with $P\subseteq Q$, and $\pi$ an automorphic cuspidal representation of $M_P(\AAA)$. Here two cuspidal pairs are equivalent if one pair can be transformed into the other by conjugation and tensoring with a norm character. 
For each cuspidal pair $(M_P, \pi)$, this subspace is generated topologically by pseudo-Eisenstein series. Recall that pseudo-Eisenstein series are defined for every $G(\ooo)$-finite function $\varphi\in L^{2}(N_P(\AAA)M_P(F)\backslash G(\AAA))_{\pi}$ ($\pi$-isotypic part,  see page 78 - 79 of the \textit{op. cit.}), and every compactly supported function $f\in C_{c}(\ago_{P}^{G})$, by
 \[\theta=\sum_{\gamma\in P(F)\cap M_Q(F) \backslash M_Q(F)} f( H_P(\gamma g)) \varphi(\gamma g).\]  
Therefore, its constituents are irreducible subquotients of the representation parabolically induced from $(M_P, \pi)$. Since $R_Q((\dim \rho)e_{\rho})$ is a $G(\ooo)$-projection onto the $\rho$-isotypic part, we conclude from the condition that $\rho$ is a cuspidal filter that $R_Q((\dim \rho)e_{\rho})$ is the zero operator if $Q\neq G$ and the image of $R_G((\dim \rho)e_{\rho})$ is contained in the cuspidal spectrum. 



The above arguments imply that the kernel function $k_Q(x,y)$ of $R_Q(e_{\rho})$ is identically zero 
if $Q\neq G$. And when $Q=G$, the trace of the regular representation $R_G(e_\rho)$ exists and equals the sum \[\sum_{\pi: \pi_{\rho}\neq 0}m_{\pi}\] by Corollary \ref{mul1}. 
Recall that in \cite[Proposition 9.1]{Yu2}, we proved that
\[ J^{G,\xi}(e_\rho)=\int_{G(F)\backslash G(\AAA)}\sum_{Q\in \mathcal{P}(B)}(-1)^{\dim\ago_Q^G}\sum_{\delta\in Q(F)\backslash G(F)} \htau_{Q}\left(H_0(\delta x) +  s_{\delta x}\xi   \right) k_{Q}(\delta x, \delta x)\d x. \]
Therefore the truncated trace $J^{G, \xi}(e_\rho)$, which is clearly independent of $\xi$ in our case, equals  the trace of $R_G(e_{\rho})$. 
\end{proof}

\subsection{Trace formula for Lie algebras}
Before we study $J^{G, \xi}(e_{\rho})$, we need to study a trace formula for Lie algebra first. 

We invite the reader to read Section \ref{finalsec} first. The reason that this subsection is placed before Section \ref{finalsec} is that we need a $\xi$-independence result proved in this subsection. The main results of this subsection that will be needed later are the corollaries \ref{indepxi} and \ref{corollary}.

\subsubsection{Some explicit calculations on Fourier transforms}\label{Fourier}

Let $\langle \cdot, \cdot \rangle$ be a $G$-invariant bilinear form on $\ggg$ defined over $\mathbb{F}_q$. Thus, it defines by taking $\AAA$-points, and $F_v$-points (for every place $v$), a bilinear form on $\ggg(\AAA)$ and $\ggg(F_v)$ respectively.

Let $D=\sum_{v}n_v v$ be a divisor on $X$. We denote \[\wp^{D}= \prod_{v\in |X|} \wp_v^{n_v},\] where $\wp_v$ has been defined to be the maximal ideal of $\mathcal{O}_v$. 
Let $$K_X=\sum_{v} d_v v $$ be a divisor whose associated line bundle is the canonical line bundle $\omega_X$ of $X$. 
We have 
$$\mathbb{A}/(F+ \wp^{-K_X})\cong H^{1}(X, \omega_X)\cong H^{0}(X,\mathcal{O}_{X})^{*}\cong\mathbb{F}_{q},$$ by Serre duality and the fact that $X$ is geometrically connected. 
We fix a non-trivial additive character $\psi$ of $\mathbb{F}_q$. 
Via the above isomorphisms, $\psi$ can be viewed as a character of $\AAA/F$. We use this $\psi$ in the definition of Fourier transform. 
We can decompose $$\psi=\otimes_{v}\psi_{v}, $$
with $\psi_v: F_v\rightarrow \mathbb{C}^{\times}$. It is clear that $\psi_v$ is trivial on $\wp_{v}^{-d_v}$ but not on $\wp_{v}^{-d_v-1}$. 
We define a local Fourier transformation for functions in $\mathcal{C}_{c}^{\infty}(\ggg(F_v))$ using the character $\psi_v$.

It is clear by definition that for any $f=\otimes_v f_v$, we have \[\hat{f}=\otimes_v \hat{f_v }. \]
For any place $v$ of $F$, with our normalization of measures, that $\vol(\ggg(\ooo_v))=1$, we have 
\begin{equation}\hat{\hat{f}_v}(X)=q^{d_v \deg v\dim \ggg}   f_v(-X). \end{equation}

We summarize some calculations on Fourier transformations by the following proposition. 
\begin{prop}\label{ft}
We have the following results on Fourier transformation.
\begin{enumerate}
\item
For any place $v$ of $F$, for characteristic function of $\ggg(\ooo_v)$, we have
\[\hat{\mathbbm{1}}_{\ggg(\ooo_v)}=   \mathbbm{1}_{ \wp_{v}^{-d_v} \ggg(\ooo_v)}. \] 
\item
Let \[{f}_\infty=   \mathbbm{1}_{ \mathfrak{I}_{\infty}} \chi,  \] 
where $\chi:  \mathfrak{I}_{\infty}\rightarrow \ttt(\kappa_\infty)\rightarrow \mathbb{C}^{\times}$ is the character that lifts the character on $\ttt(\kappa_\infty)$ determined by $\chi(\cdot)=\psi(\langle \cdot, t\rangle)$, then 
$\hat{f}_\infty$ equals \[q^{-\frac{1}{2}(\dim \ggg- \dim \ttt)}\mathbbm{1}_{ \wp_{\infty}^{-d_\infty-1}({t+ \mathfrak{I}_{\infty+}})}, \] 
where \( t \in \ttt(\kappa_\infty)\), viewed as an element in \(  \mathfrak{I}_{\infty} \), is determined by $\chi$ (it is the support of the Fourier transform of $\chi$).

\item (Springer's hypothesis \cite{Kazhdan}\cite{KV}.) 
Let $l: \mathcal{U}_G\rightarrow \mathcal{N}_{\ggg}$ be an isomorphism as in Proposition \ref{ch}. Let $T_v$ a maximal torus of $G_{\kappa_v}$ and $\theta$ any character of $T_v(\kappa_v)$, $t$ a regular element in $\ttt_v(\kappa_v)$ and $\rho=\epsilon_{\kappa_v}(G)\epsilon_{\kappa_v}(T_v)R_{T_v}^G\theta$ the Deligne-Lusztig virtual
representation of $G(\kappa_v)$ corresponding to $T_v$ and $\theta$. 
Let $$e_{\rho}:=\begin{cases}
\Tr({\rho}(\overbar{x}^{-1})), \quad x\in G(\ooo_v);\\
0, \quad x\notin G(\ooo_v); 
\end{cases}$$
where $\overbar{x}$ denotes the image of $x$ under the map $G(\ooo_v)\rightarrow G(\kappa_v)$. Let $\overbar{\Omega}_t\subseteq \ggg(\kappa_v)$ the $\Ad (G(\kappa_v))$-orbits of $t$ and $\Omega_t\subseteq \ggg(\ooo_v)$ be the preimage of $\overbar{\Omega}_t$ of the map $\ggg(\mathcal{O}_v)\rightarrow \ggg(\kappa_v)$.

Then for any unipotent element $u\in \mathcal{U}(\ooo_v)$, we have
$$e_{\rho}(u)= q^{-\deg  v (d_v\dim \ggg+\frac{1}{2}(\dim \ggg-\dim \ttt))  }\hat{\mathbbm{1}}_{\wp_{v}^{-d_v-1}\Omega_{t} }(l(u)).  $$

\end{enumerate}
\end{prop}
\begin{proof}
By the fact that $\langle \cdot, \cdot \rangle$ is defined over $\mathbb{F}_q$, and $\psi_v$ has conductor $\wp_v^{-d_{v}}$,
the dual lattice of $\ggg(\ooo_v)$ defined by
$$\ggg(\ooo_v)^{\perp}:=\{ g\in \ggg(F_v)\mid \langle g, h \rangle\in \wp_v^{-d_v}\text{ for each $h\in \ggg(\ooo_v)$}  \}, $$
equals $\wp_v^{-d_v}\ggg(\ooo_v)$. 
As $\ggg(\ooo_v)=\ggg(F_v)_0$, and $\wp_v\ggg(\ooo_v)=\ggg(F_v)_{0+}$, 
it follows (\cite[Lemma 1.8.7(a)]{KV}) that for any point $x\in \mathcal{B}(G)_v$, 
$$ \ggg(F_v)_{x}^{\perp}=\wp_{v}^{-d_v-1}\ggg(F_v)_{x+}.$$

The first and second statements follow from direct calculations. 

The last statement can be reduced to the Springer hypothesis, first proved in \cite{Kazhdan} for large characteristics. A proof of Springer's hypothesis for all characteristics (as long as a Springer isomorphism exists) can be found in \cite[Theorem A.1]{KV}. 
\end{proof}

\subsubsection{Weighted orbital integrals}\label{s532}
With notation in Proposition \ref{ft}, 
let $\varphi=\otimes \varphi_v\in \mathcal{C}_c^{\infty}(\ggg(\AAA))$, where for $v$ outside $S$, 
$$\varphi_v=\mathbbm{1}_{\ggg(\ooo_v)};$$
for $v=\infty$, 
$$\varphi_{\infty}=  \mathbbm{1}_{ \mathfrak{I}_{\infty}}\chi, $$ with $\chi$ being a character given by a regular semisimple element ${t}_{\infty}\in \ttt(\kappa_\infty)$ as in Proposition \ref{ft}; and for $v\in S-\{\infty\}$,  \begin{equation*}
\varphi_{v}= q^{- \deg  v (d_v\dim \ggg+\frac{1}{2}(\dim \ggg-\dim \ttt))  }\hat{\mathbbm{1}_{\wp_{v}^{-d_v-1}\Omega_{- t_v} }}, 
\end{equation*}
where $t_v$ is a semisimple regular element in $\ttt_v(\kappa_v)$ for a maximal torus $T_v$ of $G_{\kappa_v}$ defined $\kappa_v$. 

Note that by Proposition \ref{ft}, for $v\notin S$, we have 
\begin{equation}
\hat{\varphi_v}=\mathbbm{1}_{\wp_v^{-d_v}\ggg(\ooo_v)}, 
\end{equation}
for $v\in S-\{\infty\}$, we have
 \begin{equation}
 \hat{\varphi_v} =  q^{-\frac{1}{2}\deg v(\dim \ggg-\dim \ttt) }\mathbbm{1}_{\wp_{v}^{-d_v-1}\Omega_{t_v}},
 \end{equation}
and \begin{equation}\hat{\varphi_\infty}= q^{-\frac{1}{2}(\dim \ggg-\dim \ttt)}\mathbbm{1}_{\wp_\infty^{-d_\infty-1}(t_\infty+\mathfrak{I}_{\infty+})}. \end{equation}

A first observation is that the support of $\varphi$ is contained in $\prod_{v\neq \infty}\ggg({\mathcal{O}_v})\times \mathfrak{I}_{\infty}$, so Theorem \ref{expansion'} is applicable.  We have 
\begin{equation}\label{method1} J^{\ggg, \xi}(\varphi)= J_{nil}^{\ggg, \xi}(\varphi).  \end{equation}
By the trace formula of Lie algebras established in \cite[5.7]{Yu2}, we get a coarse geometric expansion:
\begin{equation}\label{method2}
J^{\ggg, \xi}(\varphi) =q^{(1-g)\dim \ggg}\sum_{o\in \mathcal{E}}J_{o}^{\ggg,\xi}(\hat{\varphi}).     \end{equation}
We are going to study $J^{\ggg, \xi}_{o}(\hat{\varphi})$ for each $o\in \mathcal{E}$ separately.  Recall that the set of semisimple conjugacy class $\mathcal{E}$ is defined in \ref{recalltrace}.

Suppose $o\in \mathcal{E}$ be a class such that $J_{o}^{\ggg, \xi}(\hat{\varphi})$ does not vanish, then there is a $Y\in o$ and an element $x\in G(F_\infty)$ such that \[\Ad(x^{-1})Y \in \wp_\infty^{-d_\infty-1}(t_{\infty}+\mathcal{I}_{\infty,+}).  \]
Necessarily $Y$ is regular semisimple. Therefore $o$ is a class admitting Jordan decomposition and consists of semisimple elements.

We fix a semisimple representative  $Y\in o$ contained in a minimal standard Levi subalgebra $\mmm_1$ of $\ggg$ defined over $F$. 
After Proposition \cite[7.2]{Yu2} and since $Y$ is regular, we have
\[\mathfrak{j}_{Q, o}(x) =\sum_{Y \in \mmm_\ppp(F)\cap o}\sum_{\eta\in N_Q(F)}\hat{\varphi}(\Ad(x^{-1})\Ad(\eta^{-1}) Y  ).  \]
Using the same arguments as in \cite[p.949-p.951]{Arthur} or \cite[Section 5.2]{ChauLie}, we can reduce $J_{o}^{\ggg, \xi}(\hat{\varphi})$ to the following expression
\begin{equation}\label{ok?}\int_{G_{Y}(F)\backslash  G(\AAA)}\chi^{\xi}_{M_1}(x) \hat{\varphi}(\Ad(x^{-1})  Y) \d x,   \end{equation}
where 
$$ \chi^{\xi}_{M_1}(x)= \sum_{P\supseteq P_1}(-1)^{\dim\ago_P^G }\sum_{s\in W(\ago_{P_1},  P )}  \htau_P(H_0(w_s  x)   + s_{w_s  x}\xi  ),   $$
the set $W(\ago_{P_1}, P)$ consists of double classes in $W^{P}\backslash W/W^{P_1}$ formed by elements $s\in W$ such that $s(\ago_{P_1})\supseteq \ago_P$. The expression could be further simplified due to the following lemma.

\begin{lemm}\label{wsgxi}
Let $Y\in \ggg(F)$ be an element such that \[\Ad(x^{-1})(Y)\in {t}_{\infty}+\mathfrak{I}_{\infty+}, \] for some $x\in G(F_{\infty})$. Then there exists a Weyl element $w\in W$ which is independent of $x$, such that for any $x\in G(F_\infty)$ satisfying $\Ad(x^{-1})(Y)\in {t}_{\infty}+\mathfrak{I}_{\infty+}$ we have \[\chi^{\xi}_{M_1}(x)=\Gamma_{M_1}({w\xi}, x).\] Here $\Gamma_{M_1}({w\xi}, \cdot )$ is defined to be the characteristic function of those $y\in G(\AAA)$ so that the convex envelope of $(-H_P(y))_{P\in \mathcal{P}(M_1)}$ contains $[w\xi]_{M_1}$, the projection of $w\xi$ in $\ago_{M_1}$. 
\end{lemm}
\begin{proof}
Let $Y_1= \Ad(x^{-1})(Y)$. 
By \cite[Lemma 2.2.5(a)]{KV} (where they suppose that $T$ is elliptic, but their proof works in our case by just replacing $G_a$ in their proof by $\mathcal{I}_{\infty}$), the stabilizer $G_{Y_1}$ is $\mathcal{I}_{\infty+}$-conjugate to $T$. The image of $Y_1$ under this conjugation is an element in $\ttt(\ooo_\infty)$ whose reduction mod-$\wp_\infty$ equals $t_{\infty}$. We can pick a $t$ in ${t}_{\infty}+\mathfrak{I}_{\infty,+}$ so that $t\in \ttt(\ooo_\infty) $ is $G(F_\infty)$-conjugate to $Y_1$. So $t$ and $Y$ are $G(F_\infty)$-conjugate. 
Since $Y$ is regular, its centralizer $G_{Y}$ is contained in $M_1$ which is $G(F_\infty)$-conjugate to $T$. It shows that $G_{Y}$ is an $F_\infty$-split maximal torus in $M_1$. Therefore $T$ and $G_Y$ are conjugate by an element in $M_1(F_\infty)$. We  conclude that there exist a Weyl element $w$ and an $m\in M_1(F_\infty)$ such that $$ \Ad(w^{-1})\Ad(m)(Y)=t.$$
Thus for any $s\in W(\ago_{P_1}, P)$ we have $$ s_{w_s x}\xi=w_s w\xi. $$
The result then is a corollary of Proposition \cite[1.8.7]{LabWal}.  
\end{proof}
In the setup of Lemma \ref{wsgxi}, if we replace $Y$ by $X=\Ad(w^{-1})(Y)$, then we deduce the following corollary. 
\begin{coro}[of the proof of Lemma \ref{wsgxi}]\label{533}
Let $Y$ be as in Lemma \ref{wsgxi}. The $G(F)$-conjugacy class of $Y$ contains an element $X\in \ggg(F)$ such that the minimal Levi subgroup $M$ defined over $F$ whose Lie algebra containing $X$ is semistandard and there is an element $m\in M(F_\infty)$ satisfying 
\begin{equation}\label{Mconj}
\Ad(m^{-1})(X)\in {t}_{\infty}+\mathfrak{I}_{\infty+}.\end{equation}
Any element $X'\in \mmm(F)$ that is $G(F)$-conjugate to $X$ and satisfy \eqref{Mconj} is also $M(F)$-conjugate to $X$. With the element $X$ replacing $Y$, the Weyl element in Lemma \ref{wsgxi} is the trivial one. 
\end{coro}

Now we use $X$ instead of $Y$ to represent the class $o\in\mathcal{E}$ (hopefully, the reader will not confuse it with the $X$ that we use to denote the curve), which surely does not affect the value $J_{o}^{\ggg,\xi}(\hat{\varphi})$. From Lemma \ref{wsgxi} and  Corollary \ref{533}, we deduce that:
\begin{equation}\label{better} J_{o}^{\ggg,\xi}(\hat{\varphi}) = \int_{G_X(F)\backslash G(\AAA)}\Gamma_{M}(\xi, x) \hat{\varphi}(\Ad(x^{-1})X )\d x.\end{equation}
\begin{prop}\label{reg-ss}
If there exists a place $v\in S$, such that for all $x \in G(F_v) $, \[\Ad(x^{-1})(X) \notin \wp_{v}^{-d_v-1}\Omega_{t_{v}},\] then we have 
\[J_{[X]}^{\ggg, \xi}(\hat{\varphi})=0.\] Otherwise $J_{[X]}^{\ggg, \xi}(\hat{\varphi})$ equals
\begin{equation}\vol(\mathcal{I}_{\infty})   q^{ -\frac{1}{2}(\dim \ggg (2g-2+\deg S)-  \dim \ttt \deg S)}\frac {\vol(G_{X}(F)\backslash G_{X}(\AAA)^{1})}{\vol(\ago_{M}/X_{*}(M) )}J_{M}^{\ggg}(X, \mathbbm{1}_{D}), \end{equation}
where $D$ is the sum of a canonical divisor $K_X=\sum_{v} d_v v$ with $\sum_{v\in S}v$, $\mathbbm{1}_{D}$ is the characteristic function of the open compact subset $\wp^{-D}\ggg(\ooo)$, and $G_X(\AAA)^1$ is the kernel of $H_{M}|_{G_X(\AAA)}$, the restriction of the Harish-Chandra's morphism $H_{M}$ to $G_X(\AAA)$.   
\end{prop}
\begin{remark}\normalfont
We have not made precise the Haar measure on $G_X(\AAA)$ nor the Haar measure on $\ago_{M}$.
The result does not depend on these choices if they are chosen in a compatible way since they also appear in weighted orbital integral. 
\end{remark}
\begin{proof}
It is clear that if the integral \[\int_{G_{X}(F)\backslash  G(\AAA)}\Gamma_{M}(\xi, x) \hat{\varphi}(\Ad(x^{-1}) X) \d x\]
in \eqref{ok?} is non-zero, then for any $v\in S$, there is an element $x_0 \in G(F_v)$ such that  $\Ad(x_0^{-1})X\in \wp_{v}^{-d_v-1}\Omega_{t_{v}}$. 
Suppose in the following that $X$ is an element. The point is to change the test function $\hat{\varphi}$ to $\mathbbm{1}_D$.

For any $v\in S-\{\infty\}$, 
and any element $x\in G(F_v)$, we have $$\hat{\varphi}_v( \Ad(x^{-1})(X) )   \neq 0 \Longleftrightarrow \Ad(x^{-1})(X)\in \wp_{v}^{-d_v-1}\ggg(\ooo_v). $$ 
The left-hand side clearly implies the right-hand side.
For the other direction, suppose $x\in G(F_v)$ satisfies \[\Ad(x^{-1})(X)\in \wp_{v}^{-d_v-1}\ggg(\ooo_v).\]  
Since both $\Ad(x^{-1})(\wp_{v}^{d_v+1}X) $ and $\Ad(x^{-1}_0)(\wp_{v}^{d_v+1}X)$ are elements in $\ggg(\ooo_v)$  and the image mod-$\wp_v$ of ${\Ad(x^{-1}_0)(\wp_{v}^{d_v+1}X)}$ in $\ggg(\kappa_v)$ is regular, by the Lie algebra version of a result of Kottwitz \cite[Proposition 7.1]{Ko3}(which can be proved by the same argument), they are conjugate under $G(\ooo_v)$. Observe that the set $\Omega_{t_{v}}$ is $G(\ooo_v)$-conjugate invariant, so we have \[\Ad(x^{-1})(X)\in \wp_{v}^{-d_v-1}\Omega_{t_{v}}. \] Therefore, the value of $J_o^{\ggg,\xi}(\hat{\varphi})$ does not change if the local factor $\hat{\varphi_v}$ is replaced by $\mathbbm{1}_{\wp_v^{-d_v-1}\ggg(\ooo_v)}$.

Now we need to consider the place $\infty$. 
Note that the function $x\mapsto \Gamma_{M}(\xi, x)$ is $G(\ooo)$-right invariant. So we can write $J^{\ggg, \xi}_o(\hat{\varphi})$ as 
$$\int_{G_X(F)\backslash G(\AAA)}\Gamma_{M}(\xi, x)\int_{G(\ooo)}\hat{\varphi}(\Ad(k^{-1})\Ad(x^{-1})X) \d k \d x.$$
Therefore, we can replace $\hat{\varphi}$ by its $G(\ooo)$-average. 
We need the following lemma to deal with this average.
\begin{lemm}\label{okaiii}
We have 
$$\bigcup_{k\in G(\ooo_\infty)}\Ad(k)(t_{\infty}+  \mathfrak{I}_{\infty+})=\Omega_{t_{\infty}}.   $$
\end{lemm}
\begin{proof}
Since $\wp_\infty\ggg(\ooo_\infty)\subseteq \mathfrak{I}_{\infty+}\subseteq \ggg(\ooo_\infty)$, the statement is equivalent to the following statement about finite groups:
\[\bigcup_{k\in G(\kappa_\infty)}\Ad(k)(t_{\infty}+ \nnn_B (\kappa_\infty))=\overbar{\Omega}_{{t}_\infty}, \]
where $\overbar{\Omega}_{{t}_\infty}$ is the set of conjugacy class of $t_\infty$ in $\ggg(\kappa_\infty)$. 
This last statement is clear since any element in $t_\infty+ \nnn_B(\kappa_\infty)$ is conjugate to $t_\infty$. This can be proven either by a Lie algebra version of the proposition \cite[7.2]{Yu2} or simply by the well-known fact that the morphism $N_B\rightarrow \nnn_B$ given by $\Ad(n)t_\infty-t_\infty$ is an isomorphism of schemes (see for example \cite[18.7.2]{C-L2}).  
\end{proof}
By Lemma \ref{okaiii} and \cite[2.2.5(b)]{KV}, for any $X\in \ggg(F_{\infty})$ we have 
$$\int_{G(\ooo_\infty)}\mathbbm{1}_{\wp_\infty^{-d_\infty-1}(t_{\infty}+  \mathfrak{I}_{\infty+})}( \Ad(k)^{-1}X)\d k=\vol(\mathcal{I}_{\infty}) \mathbbm{1}_{\wp_\infty^{-d_\infty-1}\Omega_{t_{\infty}}}(X). $$
Therefore, the discussion for $v\in S-\{\infty\}$ applies for $v=\infty$ too. The value of $J_{o}^{\ggg, \xi}(\hat{\varphi})$ equals \begin{equation}\vol(\mathcal{I}_{\infty})J_{o}^{\ggg, \xi}(\mathbbm{1}_{D})=
\vol(\mathcal{I}_{\infty})\int_{G_{X}(F)\backslash  G(\AAA)}\Gamma_{M}(w\xi, x) \mathbbm{1}_{D}(\Ad(x^{-1})  X ) \d x.   \end{equation}
By this, we have reduced it to a case dealt with in \cite{C-L1}. In fact, the proof of \cite[11.9.1]{C-L1}
shows that $J_{o}^{\ggg, \xi}(\hat{\varphi})$ equals 
\begin{equation}\vol(\mathcal{I}_{\infty})
\int_{G_{X}(\AAA)\backslash  G(\AAA)}\mathrm{w}_{M}^{\xi}(x)\mathbbm{1}_{D}(\Ad(x^{-1})  X ) \d x, \end{equation}
where \(\mathrm{w}_{M}^{\xi}(x)\) equals the number of points of the set $[\xi]_{M}+X_{\ast}(M)$ contained in the convex envelope of the points $-H_P(x)$ for all $P\in \mathcal{P}(M)$.  Finally, we obtain the result needed by Corollary 11.14.3 of  \cite{C-L1}. 
\end{proof}


\begin{coro}\label{indepxi}
The value of $J^{\ggg, \xi}_{nil}(\varphi)$ is independent of $\xi$ as long as $\xi$ stays in general position in the sense of Definition \ref{gepo}. 
\end{coro}

\begin{coro}\label{corollary}
Let $o\in \widetilde{\mathcal{R}}_G(\mathbb{F}_q)$ be the image of \[(t_v)_{v\in S}\in \prod_{v\in S}\ttt_v^{\reg}(\kappa_v)\] in $\widetilde{\mathcal{R}}_G(\mathbb{F}_q)$. 
Then we have 
\[\frac{1}{\vol(\mathcal{I}_{\infty})}J^{\ggg, \xi}_{nil}(\varphi)=q^{-\frac{1}{2}(\dim \mathcal{M}-\dim\mathcal{R})}|\widetilde{\mathcal{M}}^{\xi}_G(o)(\mathbb{F}_q)|. \]
\end{coro}
\begin{proof}

By the identities \eqref{method1} and \eqref{method2}, we have
\[J^{\ggg, \xi}(\varphi) =\sum_{o\in \mathcal{E}}J^{\ggg, \xi}_{o}(\hat{\varphi}).  \]
Therefore by Proposition \ref{reg-ss}, $J^{\ggg,\xi}(\varphi)$ equals the sum of \[\vol(\mathcal{I}_{\infty})   q^{ -\frac{1}{2}(\dim \ggg (2g-2+\deg S)-  \dim \ttt \deg S)}\frac {\vol(G_{X}(F)\backslash G_{X}(\AAA)^{1})}{\vol(\ago_{M}/X_{*}(M) )}J_{M}^{\ggg}(X, \mathbbm{1}_{D}), \]
over all classes $[X]\in \mathcal{E}$ that can be represented by a semisimple element $X\in \ggg(F)$ such that for every $v\in S$, it can be $G(F_v)$-conjugate into \[\wp_{v}^{-d_v-1}\Omega_{t_v},\] and for $v\notin S$, it can be $G(F_v)$-conjugate into \[\wp_{v}^{-d_v}\ggg(\ooo_v). \] 
For each class above, we have chosen in  Corollary \ref{533} a representative $X$ so that (1) the group $M$ is a semistandard Levi subgroup that is also the minimal Levi subgroup such that $X\in \mmm(F)$; (2) there is an element $m\in M(F_\infty)$ satisfying
\begin{equation}
\Ad(m^{-1})(X)\in \wp_{v}^{-d_v-1}({t}_{\infty}+\mathfrak{I}_{\infty+}). \end{equation}
This last condition is equivalent to the existence of $m\in M(F_\infty)$ such that \begin{equation}
\Ad(m^{-1})(X)\in \wp_{v}^{-d_v-1}\Omega_{{t}_{\infty}}. \end{equation}
Following the diagram \eqref{41222} and \eqref{41333}, it is equivalent further that $(\chi_M(X), t_\infty)\in \widetilde{\mathcal{A}}_{M}(\mathbb{F}_q)$ and the image of $\chi_M(X)$ in $\widetilde{\mathcal{R}}(\mathbb{F}_q)$ 
 is the same as the image of $(t_v)_{v\in S}$ in $\widetilde{\mathcal{R}}(\mathbb{F}_q)$. The fact that $X$ is elliptic in $\mmm(F)$ implies that  $(\chi_M(X), t_\infty)\in \widetilde{\mathcal{A}}_{M}^{\mathrm{ell}}(\mathbb{F}_q)$ (Proposition \ref{423}).

Now we compare with Theorem \ref{C-L1}. It remains to verify that the dimension is correct. 
This is clear since we have (\cite[4.13.4]{Ngo}) \[\dim \mathcal{M}_G= \dim \ggg\deg D,   \]
and \[\dim \mathcal{R}_G=\deg S\dim \ttt.\] 
\end{proof}
\begin{remark}\normalfont
When $\deg S>\max\{2-2g,0\}$, the residue morphism $\widetilde{\res}_G: \widetilde{\mathcal{M}}_G\rightarrow\widetilde{\mathcal{R}}_G$ is faithfully  flat. In this case we have \[\dim\widetilde{\mathcal{M}}^{\xi}_G(o) = \dim \mathcal{M}_G- \dim \mathcal{R}_G. \]
\end{remark}

\subsubsection{Proof of Theorem \ref{independent}}\label{proofind}
Let us take our attention away from the proof of the main theorem and return to Theorem \ref{independent}. 
We caution first that the proof is not circular as it relies only on results of Section \ref{4.1}-\ref{4.2}.

Namely we should show that the groupoid cardinality of $|\widetilde{\mathcal{M}}^{\xi}_G(o_G)(\mathbb{F}_q)|$ is independent of the choice of regular $o_G\in \widetilde{\mathcal{R}}_G(\mathbb{F}_q)$. By  Corollary \ref{corollary}, it amounts to show that $J^{\ggg,\xi}_{nil}(\varphi)$ is independent of $(t_v)_{v\in S}$ in the definition of the function $\varphi$. However, this is obvious from its definition and Springer's hypothesis:  
The restriction to nilpotent elements of the function $\varphi_{\infty}=  \mathbbm{1}_{ \mathfrak{I}_{\infty}}\chi, $ does not depend on the character $\chi$; the function $\varphi_{v}$ for $v\in S-\{\infty\}$ is the Fourier transform of the support function of the conjugacy class of $t_v$. Springer's hypothesis (\cite[Theorem A.1]{KV}) says that its restriction to nilpotent elements depends only on the conjugacy class of the torus that contains $t_v$.

\subsection{Comparison of truncated traces between $G$ and $\ggg$}\label{finalsec}

\subsubsection{}
We will relate $J^{G,\xi}(e_{\rho})$ with some truncated trace of Lie algebras by Theorem \ref{expansion} and Theorem \ref{expansion'}. 
To fulfill the hypothesis of these theorems, let us introduce another test function. 

Let $\widetilde{e}_{\rho}$ be the test function that coincides with $e_{\rho}$ at every place $v$ except at $v=\infty$, where we define \begin{equation}\label{terho}\widetilde{e}_{\rho, \infty}(x)=\begin{cases}  \theta_{\infty}(\overbar{x}^{-1}), \quad x\in \mathcal{I}_{\infty};     \\
0, \quad x\notin \mathcal{I}_{\infty};
\end{cases}
\end{equation}
here $\overbar{x}$ denotes the reduction mod-$\wp_\infty$ of $x$. 

\begin{prop}\label{Frobenius}
For any $\xi\in \ago_{B}$, we have
$$ J^{G,\xi}(e_\rho) =  \frac{1}{\vol(\mathcal{I}_\infty)} J^{G,\xi}(\widetilde{e}_\rho). $$
\end{prop}
\begin{proof}

By Frobenius reciprocity, for any irreducible smooth representation $\pi$ of $G(F_\infty)$, we have
$$ \Hom_{G(\ooo_\infty)}(\rho, \pi|_{G(\ooo_\infty)}) = \Hom_{\mathcal{I}_{\infty}}(\theta_\infty, \pi |_{\mathcal{I}_{\infty}}  ).$$
The result then follows from Corollary \ref{mul1} and Proposition \ref{mul2}.
\end{proof}

\subsubsection{A partition of $T(\mathbb{F}_q)$}
To proceed, we need to define a partition of the set $T(\mathbb{F}_q)$ following \cite[5.2, 5.6]{Reeder}.

Let $I(T)$ be the set defined by $$I(T):=\{G_s^{0}| s\in T(\mathbb{F}_q) \}. $$
Note that each element of $I(T)$ is determined by a subset of the roots of $T$ in $G$. 
For any $\iota\in I(T)$,  let $G_\iota$ be the corresponding connected centralizer, and let 
$$\mathcal{S}_\iota:=\{s\in T(\mathbb{F}_q)| G_s^{0}=G_{\iota}\}.$$
Note that $T(\mathbb{F}_q)$ is finitely partitioned as \begin{equation}\label{partition}T(\mathbb{F}_q)=\coprod_{\iota\in I(T)}\mathcal{S}_{\iota}. \end{equation}
For any $\iota\in I(T)$, let $$Z_{\iota}:=Z(G_\iota). $$
We have $$ \mathcal{S}_\iota\subseteq Z_\iota(\mathbb{F}_q)\subseteq T(\mathbb{F}_q), $$
and $$ G_\iota= C_G(Z_\iota)^{0}. $$

We define a partial order on the set $I(T)$ by the relation: 
$$ \iota'\leq \iota \quad \iff \quad G_{\iota}\subseteq G_{\iota'}. $$
Equivalently we have $$\iota'\leq \iota \quad \iff \quad Z_{\iota'}\subseteq Z_{\iota}. $$
Let \[I(T)^{\mathrm{ell}}\] be the subset of $I(T)$ consisting of elliptic elements. 


\subsubsection{Proof of Theorem \ref{main}}
Now we study $J^{G,\xi}(\widetilde{e}_{\rho})$. 

By Theorem \ref{expansion}, $J_{o}^{G,\xi}(\widetilde{e}_{\rho})$ vanishes unless $o$ is represented 
by an elliptic element $s\in T(\mathbb{F}_q)$. In this case, we have
\begin{equation}\label{tpp}J^{G, \xi}_{[s]}(\widetilde{e}_{\rho})=\frac{\vol(\mathcal{I}_\infty)}{\vol(\mathcal{I}_{\infty, s})  }\frac{1}{|\pi_0(G_{s})|}\sum_{w}J^{G_s^{0}, w\xi}_{[s]}(\widetilde{e}_{\rho}^{w^{-1}})  ,   \end{equation}
where $w$ is taken over the set of representatives of $W^{(G_\sigma^0, T)}\backslash W$ which sends positive roots of $(G_s^0, T)$ to positive roots. 
Note that for such a $w$, the group $w\mathcal{I}_\infty w^{-1}\cap G_{s}^0(F_\infty)$ is the standard Iwahori subgroup $\mathcal{I}_{\infty, s}$ of $G_{s}^0(F_\infty)$. From Definition \eqref{terho}, we have
 \begin{equation}
\widetilde{e}_{\rho, \infty}^{w^{-1}}(x)=\begin{cases}\theta_{\infty}(w^{-1}\overbar{x}^{-1}w), \quad x\in \mathcal{I}_{\infty, s};     \\
0, \quad x\notin \mathcal{I}_{\infty, s}.
\end{cases}
    \end{equation}

We have:
\begin{equation}\label{trace}\frac{1}{\vol(\mathcal{I}_{\infty})}J^{G,\xi}(\widetilde{e}_{\rho})=\sum_{s\in T(\mathbb{F}_q)^{\mathrm{ell}}/\sim}\frac{1}{\vol(\mathcal{I}_{\infty, s})  } \frac{1}{|\pi_0(G_s)|}\sum_{w} J^{G_s^{0},w \xi}_{[s]}({}\widetilde{e}_{\rho}^{w^{-1}}),  \end{equation}
where the first sum of the right-hand side is taken over the set of conjugacy classes of elliptic elements that can be represented by an element $s\in T(\mathbb{F}_q)$.

We are going to arrange the sum in \eqref{trace} following the partition on $T(\mathbb{F}_q)$ introduced earlier. 
Note that for a semisimple element $s\in T(\mathbb{F}_q)$, we have \[|T(\mathbb{F}_q)\cap  [s]|=\frac{|W|}{|W_s|},\]
 where we denote the Weyl group of $(G_s, T)$ by $W_s$ (similarly we introduce $W_s^0$).   
 Hence if we change the sum over conjugacy classes in $T(\mathbb{F}_q)^{\mathrm{ell}}$ to the sum over $T(\mathbb{F}_q)^{\mathrm{ell}}$ itself, each summand indexed by $s\in T(\mathbb{F}_q)^{\mathrm{ell}}$ is repeated $|W|/|W_s|$ times. Since $W_s^{0}$ depends only on $\iota\in I(T)^{\mathrm{ell}}$, we can use $W_\iota$ to denote $W_s^{0}$, similarly we introduce $\vol(\mathcal{I}_{\infty, \iota})$. As $|W_{s}/W_{s}^{0}|$ equals $|\pi_{0}(G_s)|$, we use the partition \eqref{partition} to rewrite \(\frac{1}{\vol(\mathcal{I}_{\infty})}J^{G,\xi}(\widetilde{e}_{\rho})\) as 
 \begin{equation}  \label{trace"}\sum_{\iota \in I(T)^{\mathrm{ell}}}\frac{1}{\vol(\mathcal{I}_{\infty, \iota})  } \frac{|W_\iota|}{|W|}\sum_{s\in \mathcal{S}_{\iota}} \sum_{w} J^{G_\iota,w \xi}_{[s]}(\widetilde{e}_{\rho}^{w^{-1}}).  \end{equation}

For each place $v\in S-\{\infty\}$, the character formula of Deligne-Lusztig (\cite[Theorem 4.2]{DL}) says that for each semisimple element $s\in G(\kappa_v)$ and each unipotent element $u$ that commutes with $s$ we have: 
$$  R_{T_v }^{G}(\theta_v^{-1})(su)=  \frac{1}{|G_s^{0}(\kappa_v)|}  \sum_{\{     \gamma \in G(\kappa_v)\mid s\in \gamma T_v  \gamma^{-1}   \}}     \theta_v(\gamma^{-1}s^{-1}\gamma)  R_{\gamma T_v\gamma^{-1}}^{G_{s}^0}(1)(u). $$
For every $s\in T(\mathbb{F}_q)^{\mathrm{ell}}$, we define
 \begin{equation}{N}(s, T_v):=G_{s}^{0}(\kappa_v) \backslash \{     \gamma \in G(\kappa_v)\mid  \gamma T_v  \gamma^{-1} \subseteq G_s^0  \}. \end{equation}
 Note that if $T_v$ is split then \[{N}(s, T_v)\cong W_\iota \backslash W .  \]
It depends only on $\iota\in T(\mathbb{F}_q)^{\mathrm{ell}}$ that contains $s$. Therefore it can be denoted by ${N}_G(\iota, T_v)$. 
With this notation, we can rewrite Deligne-Lusztig's character formula by grouping together those $\gamma$ such that $\gamma T_v \gamma^{-1}$ are conjugate in $G_{\iota}$: 
\begin{equation} R_{T_v}^{G}(\theta_v^{-1})(su)= \sum_{\overbar{\gamma}\in {N}(\iota, T_v)}R_{\gamma T_v\gamma^{-1}}^{G_\iota}(1)(u)  {\theta_v^{\gamma^{-1}}(s^{-1})}. \end{equation}
Recall that $R_{\gamma T_v\gamma^{-1}}^{G_\iota}(1)$ depends only on the conjugacy class of $\gamma T_v\gamma^{-1}$ in  $G_\iota$, so the above expression is well-defined.

We define $$N(\iota, T_\bullet )=W_\iota\backslash W\times \prod_{v\in S-\{\infty\}}N(\iota, T_v).$$ 
For each $\gamma_{\bullet}\in N(\iota, T_{\bullet})$, let $ h_{\gamma_{\bullet}}\in  C_c^{\infty}(G_\iota(\AAA))$ be the function that is the tensor product over all places of $F$ of the function
$$\mathbbm{1}_{G_\iota(\ooo_v)},$$ for each place $v$ outside $S$, with 
\begin{equation}
h_{v,\gamma}(x)=\begin{cases}\epsilon_{\kappa_v}(T_v)\epsilon_{\kappa_v}(G_\iota) R_{\gamma^{-1}T_{v}\gamma}^{G_\iota}(1)(\overbar{x}),  \quad x\in G_\iota(\mathcal{O}_v);\\
0 , \quad x\notin G_\iota(\mathcal{O}_v);
\end{cases}
\end{equation}
for each place $v\in S-{\infty}$, and with the characteristic function
$$ \mathbbm{1}_{ \mathcal{I}_{\infty, \iota}}$$
for the place $\infty$. 
Let $\iota\in I(T)^{\mathrm{ell}}$. Since semisimple and unipotent parts of an element in $G_\iota(\ooo_v)$ are in $G_\iota(\ooo_v)$ as well (Proposition \ref{JordanC}), for any $s\in \mathcal{S}_\iota$, $x\in G_{\iota}(F_v)$ and $u\in \mathcal{U}_{G_\iota}(F_v)$, the condition 
$$x^{-1}sux\in G(\ooo_v)  \quad \text{(resp.      \( x^{-1}sux\in \mathcal{I}_{\infty} \), for  \(v=\infty\))}, $$ is equivalent to
$$x^{-1}ux \in  \mathcal{I}_{\infty}\quad \text{(resp.      \( x^{-1}ux\in \mathcal{I}_{\infty} \))}.   $$
So we have $$ \widetilde{e}_{\rho}(x^{-1}sux)= \sum_{\gamma_{\bullet}\in N(\iota, T_{\bullet}) }\theta_{\gamma_\bullet}(s)h_{\gamma_\bullet}(x^{-1}ux), $$
where 
\begin{equation}\theta_{\gamma_{\bullet}}(s)= \prod_{v\in S}  \theta_v^{\gamma_v^{-1}}(s^{-1}).  \end{equation}
It follows then
\begin{equation}\label{tt''}
\sum_{w} J^{G_s^0, w\xi }_{[s]}(\widetilde{e}_{\rho}^{w^{-1}}) =  \sum_{\gamma_{\bullet}\in {N}_G(\iota, T_{\bullet})}J_{uni}^{G_s^0, w\xi}( h_{\gamma_{\bullet}}) \theta_{\gamma_{\bullet}}(s). \end{equation}

For each $\gamma_{\bullet}\in N(\iota, T_{\bullet})$, 
we define the function $f_{\gamma_{\bullet}}\in C_c^{\infty}(\ggg_\iota(\AAA))$ to be the tensor product of the function
$$\mathbbm{1}_{\ggg_\iota(\ooo_v)},$$ for places $v$ outside $S$, with 
\begin{equation*}
f_{v,\gamma}( x )=\begin{cases} \epsilon_{\kappa_v}(T_v)\epsilon_{\kappa_v}(G_\iota) R_{\gamma^{-1}T_{v}\gamma}^{G_\iota}(1)(\overbar{l^{-1}(x)}),  \quad x\in  \mathcal{N}_{\ggg_\iota}(\ooo_v);\\
0 , \quad x\notin \mathcal{N}_{\ggg_\iota}(\ooo_v);
\end{cases}
\end{equation*}
and with the function
$$ \mathbbm{1}_{ \mathfrak{I}_{\infty, \iota}}.\chi $$
for the place $\infty$. In the above, $l: \mathcal{U}_{G_\iota}\rightarrow \mathcal{N}_{\ggg_\iota}$ is the isomorphism fixed in Proposition \ref{ch} for each place $v\in S-{\infty}$ and $\chi:  \mathfrak{I}_{\infty, \iota}\rightarrow \ttt(\kappa_\infty)\rightarrow \mathbb{C}^{\times}$ is any fixed additive character. 
Note that this function is well defined since $l^{-1}$ sends $\ggg_\iota(\mathcal{O}_v) \cap \mathcal{N}_{\ggg_\iota}(F)$ to $G_\iota(\mathcal{O}_v) \cap \mathcal{U}_{G_\iota}(F)$. 
By definition, we have \begin{equation} f_{\gamma_{\bullet}}\circ l|_{\mathcal{U}_{G_\iota}(\AAA)   }=h_{\gamma_{\bullet}}|_{\mathcal{U}_{G_\iota}(\AAA)   }. \end{equation}

By Proposition \ref{222}, we have
\begin{equation}\label{ss''}
J_{uni}^{G_\iota, w\xi}( h_{\gamma_{\bullet}})=J_{nil}^{\ggg_\iota, w\xi}( f_{\gamma_{\bullet}}),
\end{equation}
where we have also used the fact that $l$ is $G$-equivariant and it sends $\mathcal{I}_{\infty}$ to $\mathfrak{I}_{\infty}$  and $\mathcal{I}_{\infty,+}$ to  $\mathfrak{I}_{\infty,+}$ respectively. In fact, since  $l$ is defined over $\mathbb{F}_q$, we have \[\overbar{l(x)}= l(\overbar{x}), \quad \forall x\in \mathcal{U}_G(\ooo_\infty). \] 
We have seen in Corollary \ref{indepxi} that $J_{nil}^{\ggg_\iota, w\xi}( f_{\gamma_{\bullet}})$ is independent of $\xi$ whenever it is in general position. In particular, we can change $w\xi$ to $\xi$.  
Therefore, we deduce from \eqref{trace"}, \eqref{tt''} and \eqref{ss''} that \(\frac{1}{\vol(\mathcal{I}_{\infty})}J^{G,\xi}(\widetilde{e}_{\rho})\) is equal to
\begin{equation}\label{resul} \sum_{\iota \in I(T)^{\mathrm{ell}}}  \frac{1}{\vol(\mathcal{I}_{\infty, \iota})  }  \sum_{\gamma_{\bullet}\in N(\iota, T_\bullet) }  \left( \frac{|W_\iota|}{|W|} \sum_{s\in \mathcal{S}_{\iota}} \theta_{\gamma_\bullet}(s)\right)  J^{\ggg_\iota, \xi}_{nil}( f_{\gamma_\bullet}) .    \end{equation}

Note the set $I(T)^{\mathrm{ell}}$ admits an action of $W$ induced by the conjugation action of $W$ on $T(\mathbb{F}_q)^{\mathrm{ell}}$. The summand in \eqref{resul} does not depend on the choice of a representative of the conjugacy class of $\iota$. 
For each $\iota\in I(T)^{\mathrm{ell}}$, let $C_W(\iota)$ be the stabilizer of $\iota$ under this action. We have $W_\iota\subseteq C_W(\iota)$. 
Let \[I \] be a set of representatives for the quotient $I(T)^{\mathrm{ell}}/W$. We put together the summand in one conjugacy class, then \eqref{resul} becomes
\begin{equation}\label{resul2} \sum_{\iota \in I}  \frac{1}{\vol(\mathcal{I}_{\infty, \iota})  }  \sum_{\gamma_{\bullet}\in N(\iota, T_\bullet) }  \left( \frac{|W_\iota|}{|C_W(\iota)|} \sum_{s\in \mathcal{S}_{\iota}} \theta_{\gamma_\bullet}(s)\right)  J^{\ggg_\iota, \xi}_{nil}( f_{\gamma_\bullet}) .    \end{equation}

Note that $N(\iota, T_\bullet)$ admits a left $C_W(\iota)$-action by acting on the factor over $\infty$ (since $W_\iota$ is normal in $C_W(\iota)$). 
Clearly for any $w\in C_W(\iota)$ we have \[J^{\ggg_\iota, \xi}_{nil}( f_{w\gamma_\bullet}) = J^{\ggg_\iota, \xi}_{nil}( f_{\gamma_\bullet}) . \]
Observing that for any $(\gamma_v)_{v\in S}\in N(\iota, T_\bullet)$, we have 
 \[\theta_{w(\gamma_v)_{v\in S}}( s )  = \theta_{(\gamma'_v)_{v\in S}}( w^{-1}sw ),\]
where $\gamma'_v = w^{-1}\gamma_v w$ if $v\neq \infty$ and $\gamma'_\infty = \gamma_\infty$. 
Moreover any $w\in C_W(\iota)$ permutes 
\[\prod_{v\in S-\{\infty\}}N(\iota, T_v) ,\] since $C_{W}(\iota)$ normalizes $G_\iota$. 
So we can break the sum over $N(\iota, T_\bullet)$ into a sum over the orbits of the $C_W(\iota)$-action on $N(\iota, T_\bullet)$, and we may write \(\frac{1}{\vol(\mathcal{I}_{\infty})}J^{G,\xi}(\widetilde{e}_{\rho})\) as 
\begin{equation}\label{resultat} \sum_{\iota \in I} \sum_{\gamma_{\bullet}\in C_W(\iota) \backslash N(\iota, T_\bullet )}  \left( \frac{|W_\iota|}{|C_W(\iota)|}\sum_{w\in W_\iota\backslash C_W(\iota)}  \sum_{s\in \mathcal{S}_{\iota}} \theta_{\gamma_\bullet}(w^{-1}sw)\right)   \frac{1}{\vol(\mathcal{I}_{\infty, \iota})  }   J^{\ggg_\iota, \xi}_{nil}( f_{\gamma_\bullet}) .    \end{equation}
Now the last thing is to notice that $C_W(\iota)$ permutes $\mathcal{S}_\iota$, therefore its action does not change the sum of $\theta_{\gamma_\bullet}$ over $\mathcal{S}_\iota$.
 We conclude that  
\(\frac{1}{\vol(\mathcal{I}_{\infty})}J^{G,\xi}(\widetilde{e}_{\rho})\)  equals 
\begin{equation}\label{final} \sum_{\iota \in I} \sum_{\gamma_{\bullet}\in C_W(\iota) \backslash N(\iota, T_\bullet )}  \left(   \sum_{s\in \mathcal{S}_{\iota}} \theta_{\gamma_\bullet}(s)\right)   \frac{1}{\vol(\mathcal{I}_{\infty, \iota})  } J^{\ggg_\iota, \xi}_{nil}( f_{\gamma_\bullet}) .    \end{equation}

Since $f_{\gamma_\bullet}$ coincides with the function $\varphi$ used in Section \ref{s532} by Springer's hypothesis (Proposition \ref{ft}), we can apply Corollary \ref{corollary} to deduce that \(\frac{1}{\vol(\mathcal{I}_{\infty})}J^{G,\xi}(\widetilde{e}_{\rho})\) equals 
\begin{equation}\label{finalll} \sum_{\iota \in I} \sum_{\gamma_{\bullet}\in C_W(\iota) \backslash N(\iota, T_\bullet )}  \left(   \sum_{s\in \mathcal{S}_{\iota}} \theta_{\gamma_\bullet}(s)\right)  q^{-\frac{1}{2}(\dim\mathcal{M}_{G_\iota}-\dim\mathcal{R}_{G_\iota} )}  |\widetilde{\mathcal{M}}_{G_\iota}^{\xi}(o_{\gamma_\bullet})(\mathbb{F}_q)|  ,\end{equation}
where we fix a $G$-regular family $(t_v)_{v\in S}$ with $t_v\in \ttt_v(\kappa_v)$, and $o_{\gamma_\bullet}$ denotes the image of $(\gamma_v^{-1}t_v \gamma_v)_{v\in S}$ in $\widetilde{\mathcal{R}}_{G_\iota}(\mathbb{F}_q)$.

\subsubsection{} We are going to write \eqref{finalll} in a more intrinsic way. 

Let $\iota\in I(T)^{\mathrm{ell}}$. Let us denote $H=G_\iota$. 
We have a finite morphism: \[\pi_{H,G}: \car_{H}\longrightarrow \car_G .\]
We have the following result. 
\begin{lemm}
Let $x_G$ be any regular element in $\ttt_v(\kappa_v)$. Let $o_G$ be the image of $x_G$ under the projection map $\chi_G: \ggg\rightarrow \car_G$. There is a bijection between the set \[N(\iota, T_v),  \] and the set \[\pi_{H,G}^{-1}(o_G)(\kappa_v). \] 
\end{lemm}
\begin{proof}
Recall that for $H=G_\iota$, we have
\[N(\iota, T_v)= H(\kappa_v) \backslash \{  \gamma\in G(\kappa_v) \mid \gamma    T_v \gamma^{-1}\subseteq H \} . \] 

Let \[\pi_{H,G}^{-1}(o_G)(\kappa_v)=\{o_1, \ldots, o_m\}. \] 
Since we are in a very good characteristic and the group $G$ split, a Kostant section exists (see \cite{Riche}). Therefore each element in the fiber can be represented by a regular semisimple element in $\hhh(\kappa_v)$ denoted by $x_i$. 
Note that the elements $x_i$ are two by two non-conjugate in $\hhh$. Otherwise, they will have the same images in $\car_H(\kappa_v)$. However, they are all geometrically conjugate to $x_G$ in $\ggg(\kappa_v)$, hence $G(\kappa_v)$-conjugate. Let $\gamma_i\in  G(\kappa_v)$ be an element such that $\ad(\gamma_i)^{-1}x_i=x_G$. We have $\gamma_i^{-1}G_{x_i}\gamma_i=G_{x_G}=T_v$. Since ${x_i}$ is $G$-regular, $H_{x_i}\subseteq G_{x_i}$ are tori of the same dimension. They thus coincide and $G_{x_i}\subseteq H$. Therefore $\gamma_i$ represents an element in $N(\iota, T_v)$. Different choices of $\gamma_i$ induce the same element in $N(\iota, T_v)$. 
 
 From the above arguments, we obtain a map from $\pi_{H,G}^{-1}(o_G)(\kappa_v)$ to $N(\iota, T_v)$. It is injective. In fact, if $\gamma_i$ and $\gamma_j$ induce the same image in $N(\iota, T_v)$, then they differ by an element in $H(\kappa_v)$ and $x_i$ and $x_j$ will be $H(\kappa_v)$-conjugate as elements in $\hhh(\kappa_v)$. The map is also surjective. Given any $\overbar{\gamma}\in N(\iota,  T_v)$ represented by $\gamma\in G(\kappa_v)$, we have $\ad(\gamma)x_G\in \hhh(\kappa_v)$. Its image in $\car_H(\kappa_v)$ lies in the fiber of $\pi_{H,G}$ in $o_G$. 
\end{proof}

Recall that we have defined in Definition \ref{coendo} the notion of split elliptic coendoscopic groups. In appendix \ref{A}, we give a classification of them in terms of the root system of $G$. In particular, the set $I$, which we have fixed as a set of representatives of $I(T)^{\mathrm{ell}}/W$, is in bijection with the conjugacy classes of split elliptic coendoscopic groups of $G$ defined over $\mathbb{F}_q$. 
Let $H$ be a split elliptic coendoscopic group of $G$; we have an obviously defined finite morphism,  denoted by $\widetilde{\pi}_{H, G}$. It is independent of the choice of a representative in the conjugacy class of $H$:
\begin{equation}\widetilde{\pi}_{H,G}: \widetilde{\mathcal{R}}^G_H\longrightarrow  \widetilde{\mathcal{R}}_G , \end{equation}
where $\widetilde{\mathcal{R}}^G_H$ is the open subscheme of $\widetilde{\mathcal{R}}_H$ formed of $G$-regular elements at $\infty$. 

Let $o_G= (o_v)_{v\in S}\in \widetilde{\mathcal{R}}_G(\mathbb{F}_q)$.  Suppose that every $o_v$ is regular. Recall that the groupoid cardinality $| \widetilde{\mathcal{M}}_{G}^{\xi}(o_G)(\mathbb{F}_q)|$ depends only on the family of  $\kappa_v$-torus that contains $o_v$ $(v\in S)$. In particular, it does not depend on $o_\infty\in \ttt_\infty^{\reg}(\kappa_\infty)$. Therefore we can write \eqref{finalll} in the following form: 
\begin{equation}\label{536}
\sum_{H}q^{-\frac{1}{2}(\dim\mathcal{M}_{H}-\dim\mathcal{R}_{H} )} \sum_{o \in \widetilde{\pi}_{H, G}^{-1}(o_G)(\mathbb{F}_q)}n_{H,o} |\widetilde{\mathcal{M}}^{\xi}_H(o)(\mathbb{F}_q)|.   
\end{equation}
where the sum is taken over split elliptic coendoscopic group $H$ of $G$ and if $H$ is represented by $G_\iota$ for $\iota\in I(T)^{\mathrm{ell}}$, $o=o_{\gamma_{\bullet}}$ for $\gamma_\bullet \in \prod_{v\in S-\{\infty\}}N(\iota,T_\bullet)$, then 
\begin{equation}\label{537} n_{H,o}= \sum_{w\in C_W(\iota)\backslash W}\sum_{s\in \mathcal{S}_\iota}\theta_{\gamma_\bullet}(s)\theta_{\infty}(w^{-1}sw).  \end{equation}

\subsubsection{Properties of the numbers $n_{H,o}$.}\label{nho}
We will prove the properties about the numbers $n_{H,o}$ stated in Theorem \ref{best}. 

\begin{prop}
For any $H,o$ appeared in the sum \eqref{536}, the numbers $n_{H,o}$ given by \eqref{537} are integers. 
We have \[n_{G,o_G}=|Z_G(\mathbb{F}_q)|.\]
We know that there is a bound $C$ depending only on the root system of $G$ and on $\deg S$, in particular it is independent of $q$, such that
\[\sum_{H,o}|n_{H,o}|\leq C . \]

Suppose $S\subseteq X(\mathbb{F}_q)$, every $T_v$ is split, and $q-1$ is divisible enough (an explicit hypothesis is given in \ref{car}).  Then every split elliptic coendoscopic group $H$ defined over $\overbar{\mathbb{F}}_q$ is defined over $\mathbb{F}_q$ and the fiber $\widetilde{\pi}_{H,G}^{-1}(o_G)$ is a constant scheme over $\mathbb{F}_q$.
Assume that $\theta_v$ is obtained as the composition of an algebraic character in $X^{*}(T_v)$ with a character $\psi: \overbar{\mathbb{F}}_q\rightarrow \mathbb{C}^{\times}$.  Then
 $n_{H,o}$ is unchanged if the field $\mathbb{F}_q$ is replaced by $\mathbb{F}_{q^n}$.

\end{prop}
\begin{proof}
By Lemma 5.1 of \cite{Reeder}, we have 
$${Z}_{\iota}(\mathbb{F}_q)=\coprod_{\iota'\leq \iota}\mathcal{S}_{\iota'}. $$
Note that $\theta_{\gamma_\bullet}$ is a character that can be defined over the set $Z_{\iota}(\mathbb{F}_q)$.
Let $\gamma_\bullet \in N(\iota, T_{\bullet})$. 
Let $\mu$ be the Möbius function of the poset $(I(T)^{\mathrm{ell}}, \leq)$, which is an integer-valued function. 
By Möbius inversion formula and the above partition, we have
\begin{equation}\label{rational}
\sum_{s\in \mathcal{S}_{\iota}}\theta_{\gamma_\bullet}(s) =   \sum_{\{\iota'\mid \iota'\leq \iota\}}\mu(\iota', \iota)  \sum_{s\in Z_{\iota'}(\mathbb{F}_q)}\theta_{\gamma_\bullet}(s) . \end{equation}
However for each $\iota'\leq \iota$, the sum \[\sum_{s\in Z_{\iota'}(\mathbb{F}_q)}\theta_{\gamma_\bullet}(s)\] is an integer as $Z_{\iota'}(\mathbb{F}_q)$ is a group. This implies that the number under consideration is an integer.

It is clear that \[n_{G,o_G}=|Z_G(\mathbb{F}_q)|.\]
In fact, let $\iota\in I(T)^{\mathrm{ell}}$ be an element such that $G_\iota=G$, then $\iota$ is the minimal element in $I(T)^{\mathrm{ell}}$, i.e., the one is given by elements in the center of $G$. In this case, $C_W(\iota)=W$ and $S_\iota=|Z_G(\mathbb{F}_q)|. $ The result follows from that the product of the restriction of every $\theta_v$ to $Z_G(\mathbb{F}_q)$ is trivial.

Concerning the bound of the sum of the absolute values of all $n_{H,o}$, it is clear by the following lemma. 
\begin{lemm}
We have the following bounds 
\[(i)\quad |N(\iota, T_\bullet)| \leq |W^{G}|^{|S|}|W^{G_\iota}|^{|S|},  \]
and 
\[(ii)\quad |\sum_{s\in \mathcal{S}_{\iota}} \theta_{\gamma_\bullet}(s)|\leq  |Z_{G_\iota}(\overbar{\mathbb{F}}_q)|.  \]
\end{lemm}
\begin{proof}
Conjugacy classes of tori defined over $\kappa_v$ in $G_\iota$ can be parametrized by $H^{1}(\kappa_v, W_\iota)$, whose cardinality is trivially bounded by $|W_\iota|$, hence we obtain the first inequality. 

The second inequality is also clear:
\[|\sum_{s\in S_{\iota}} \theta_{\gamma_\bullet}(s)|\leq |S_{\iota}|\leq |Z_{G_\iota}(\mathbb{F}_q)|\leq  |Z_{G_\iota}(\overbar{\mathbb{F}}_q)|. \]
Note that $|Z_{G_\iota}(\overbar{\mathbb{F}}_q)|$ is finite since $G_{\iota}$ is semisimple for $\iota\in I(T)^{\mathrm{ell}}$. In fact, it equals \[|\mathbb{Z}\hat{\Delta}_{B_\iota}/ X^{*}(T)|, \] 
since we are in a very good characteristic.  
\end{proof}

Let us prove the second part of the proposition. 

We will see in \ref{rationalite} of that if $q-1$ is divisible enough, then every split elliptic coendoscopic group $H$ defined over $\overbar{\mathbb{F}}_q$ is defined over $\mathbb{F}_q$ and $Z_H(\mathbb{F}_q)=Z_H(\overbar{\mathbb{F}}_q)$.
If every $S\subseteq X(\mathbb{F}_q)$ and every $T_v$ is split. We may suppose $T_v=T$.  Note that the fiber of $\widetilde{\mathcal{R}}_H\rightarrow\widetilde{\mathcal{R}}_G$ over $o_G$ is finite union of $\Spec(\mathbb{F}_q)$.  
In this case, \[N(\iota, T_v)=G_\iota(\mathbb{F}_q) \backslash \{     \gamma \in G(\mathbb{F}_q)\mid  \gamma T  \gamma^{-1} \subseteq G_\iota  \}\cong W_{\iota}\backslash W. \] 
In particular, it is unchanged under the base change from $\mathbb{F}_q$ to $\mathbb{F}_{q^n}$. Moreover, by our assumption, $Z_{G_\iota}(\mathbb{F}_{q^n})=Z_{G_\iota}(\overbar{\mathbb{F}}_{q^n})=Z_{G_\iota}(\mathbb{F}_q)$. To show that $n_{H,o}$ are not changed, by the expression \eqref{537}, it is sufficient to notice that for each  $\gamma_\bullet\in N(\iota, T_\bullet)$, the character $\theta_{\gamma_\bullet}$ of $Z_{G_\iota}(\mathbb{F}_q)$ defined by \begin{equation}\label{tgs}
\theta_{\gamma_\bullet}(s)=\prod_{v\in S}\theta_v(\gamma^{-1}_v s \gamma_v),  \end{equation}
coincides with that of $Z_{G_\iota}(\mathbb{F}_{q^n})$ when the field $\mathbb{F}_q$ is changed to $\mathbb{F}_{q^n}$. In fact, representing each $\gamma_v$ by an element in $W$, the formula \eqref{tgs} defines a character of $T({\mathbb{F}}_q)$ of the form $\psi\circ\chi$ for a character $\chi\in X^{*}(T)$ which is unchanged under a field extension. In particular, its restrictions to $Z_{G_\iota}(\mathbb{F}_{q})$ is unchanged. 
 \end{proof}

\appendix

\section{Classification of split elliptic coendoscopic groups}\label{A}
This appendix shows how to find all split elliptic coendoscopic groups. Recall that $G$ is a semisimple algebraic group defined over $\mathbb{F}_q$ so that we are in very good characteristic. We have fixed a split maximal torus $T$ and a Borel subgroup $B$ containing $T$. 

We will use the identifications below following \cite[Section 5]{DL}. 
We choose and fix isomorphisms \[{\overbar{\mathbb{F}}_q}^{\times}\xrightarrow{\sim}(\mathbb{Q}/\mathbb{Z})_{p'}, \]
and \[\text{(roots of unity of order prime to $p$ in $\mathbb{C}^{\times})$} \xrightarrow{\sim} (\mathbb{Q}/\mathbb{Z})_{p'},  \]
where $(\mathbb{Q}/\mathbb{Z})_{p'}$ is the maximal prime to $p$ subgroup of $\mathbb{Q}/\mathbb{Z}$. 
We have (\cite[(5.2.1)]{DL})
 \begin{equation}\label{characters}
 T(\overbar{\mathbb{F}}_q)\cong X_*(T)\otimes (\mathbb{Q}/\mathbb{Z})_{p'}\hookrightarrow  \ago_{T,\mathbb{Q}}/ X_{*}({T})  , \end{equation}
where $$\ago_{T,\mathbb{Q}}=X_*(T)\otimes \mathbb{Q}. $$
We also have a short exact sequence (\cite[(5.2.3)*]{DL}): 
\[
0 \longrightarrow T(\mathbb{F}_q)^{\vee}\longrightarrow X^{*}(T)\otimes (\mathbb{Q}/\mathbb{Z})_{p'}\xrightarrow{\tau-1} X^{*}(T)\otimes (\mathbb{Q}/\mathbb{Z})_{p'}\longrightarrow 0, \]
where $\tau$ is the Frobenius automorphism in $\Gal(\overbar{\mathbb{F}}_q|\mathbb{F}_q)$. Since $T$ is split, $\tau$ acts trivially on $X^{*}(T)$ and acts by multiplication by $q$ on $(\mathbb{Q}/\mathbb{Z})_{p'}$. 
Using \begin{equation}\label{cocharacters}
X^{*}(T)\otimes (\mathbb{Q}/\mathbb{Z})_{p'}\hookrightarrow \ago_{T,\mathbb{Q}}^{*}/X^{*}{(T)}, \end{equation}
where 
$$\ago_{T,\mathbb{Q}}^{*}=X^*(T)\otimes \mathbb{Q}, $$
we will view a character $\theta$ of $T(\mathbb{F}_q)$ as an element in $\ago_{T,\mathbb{Q}}^{*}/X^{*}(T)$.

\begin{aprop}\label{A1}
Suppose the root system $\Phi(G, T)$ is irreducible. Then for any $s\in T(\overbar{\mathbb{F}}_q)-Z_G(\overbar{\mathbb{F}}_q)$, $s$ is elliptic if and only if $\Phi(G_{s}^{0}, T)$ is a maximal closed subroot system of $\Phi(G,T)$ of the same rank as $\Phi(G, T)$. 
\end{aprop}
\begin{proof}
Recall that 
\[\Phi(G_{s}^{0}, T)=\{ \alpha\in \Phi(G,T)\mid \alpha(s)=1    \},   \]
so it is clear that $\Phi(G_{s}^{0}, T)$ is always a closed subroot system of $\Phi(G,T)$. 

Since $G$ is semisimple, by definition $s\in T(\overbar{\mathbb{F}}_q)$ is elliptic if and only if $ G_s^{0}$ is semisimple. If $\Phi(G_{s}^{0}, T)$ is a closed subroot system of $\Phi(G,T)$ of the same rank as $\Phi(G, T)$, then $Z_{G_s^0}^0$ has dimension $0$ by rank reasons, and $s$ is hence elliptic. 

Consider the ``only if" part of the assertion. We suppose that $s$ is elliptic. 
If the rank of $\Phi(G_s^0, T)$ is smaller than that of $\Phi(G, T)$, then by rank reasons $Z_{G_s^0}^0$ will contain a non-trivial torus, a contraction. Finally let us prove that $\Phi(G_{s}^{0}, T)$ is maximal.

Under the identification \eqref{characters}, $s$ can be represented by some element $t\in \ago_{T,\mathbb{Q}}$. 
For any $\alpha\in \Phi(G,T)$, the condition $\alpha(s)=1$ is equivalent to that \[\langle \alpha, t \rangle\in \mathbb{Z}.  \]

Recall that the fundamental alcove for the irreducible root system $\Phi(G,T)$ with respect to the base $\Delta_B$ in $\ago_{T,\mathbb{Q}}$ is given by
\[\mathcal{A}= \{r \in \ago_{T,\mathbb{Q}}\mid  0< \langle r, \alpha  \rangle<1, \; \forall \alpha\in \Delta_B \cup \{\alpha_0\} \} ,   \]
where $\alpha_0$ is the highest root for the base $\Delta_B$. 
Replacing $t$ by an element in its $W\ltimes\mathbb{Z}\Delta_B^{\vee}$-orbit only changes $\Phi(G_s^0, T)$ to a $W$-conjugation. We may suppose that $t\in \overbar{\mathcal{A}}$, the closure of $\mathcal{A}$ in $\ago_{T,\mathbb{Q}}$. 
For each subset $I$ of $\Delta_B\cup \{-\alpha_0\},$  
let $\mathcal{A}_I$ be the set of $r\in \ago_{T,\mathbb{Q}}$ satisfying  \[\begin{cases}  \langle\alpha_0, r\rangle =1, \quad \text{ if }-\alpha_0\in I;\\
 \langle\alpha_0, r\rangle <1, \quad \text{ if } -\alpha_0\notin I; \\
 \langle\alpha, r\rangle =0, \quad  \alpha \in I \cap \Delta_B ; \\
 \langle\alpha, r\rangle >0, \quad \alpha \in \Delta_B, \text{ and }  \alpha\notin I . 
 \end{cases}      \]
We have $$\overbar{\mathcal{A}}=  \coprod_{I}\mathcal{A}_I, $$
where $I$ runs through the proper subsets of $\Delta_B\cup \{-\alpha_0\}$. One knows that if $t\in \mathcal{A}_I$, then the set $\Phi(G_s^0, T)$ has $I$ as a base (\cite[Proposition, III.12-4]{Kane}). Moreover,   the rank of $\Phi(G_s^0, T)$ equals the rank of $\Phi(G,T)$, if and only if the set $\Delta_B\cup \{\alpha_0\}-I$  has cardinality $1$. 

It is shown by the discussions in \cite[p.139, III.12-3]{Kane} that, any maximal subroot system in $\Phi(G,T)$ is of the form $\Phi(G_{s'}^0, T)$ for some $s'\in T(\overbar{\mathbb{F}}_q)-Z_G(\overbar{\mathbb{F}}_q)$. If $\Phi(G_{s}^0, T)\subseteq \Phi(G_{s'}^0, T)$, then by the above proof, we can only have $\Phi(G_{s}^0, T)= \Phi(G_{s'}^0, T)$ otherwise $\Phi(G_{s'}^0, T)$ must contain $\Delta_B$ since both of them have rank equal to that of $\Phi(G,T)$ but this contradicts the fact that $s\notin Z_G(\overbar{\mathbb{F}}_q)$. This shows that $\Phi(G_{s}^0, T)$ is maximal. \end{proof}

The Borel-de Siebenthal theorem classifies the maximal closed sub-root systems in an irreducible root system (see, for example, \cite[Theorem, III.12-1]{Kane}). 
Let $\Delta=\{\alpha_1, \alpha_2, \ldots, \alpha_{r}\}$ be a base of an irreducible reduced root system $\Phi$. Let $\alpha_0$ be the highest root with respect to the base.
Expand 
\begin{equation}\alpha_0= \sum_{i=1}^{r} h_i\alpha_i. \end{equation}
Then Borel-de Siebenthal theorem says that, up to $W$-equivalence, a maximal closed sub-root system of $\Phi$ has as base
\begin{enumerate}
\item
$\{\alpha_1, \ldots,  \hat{\alpha}_i, \ldots, \alpha_r \}$ if $h_i=1$.
\item
$\{-\alpha_0, \alpha_1, \ldots,  \hat{\alpha_i}, \ldots, \alpha_r \}$ (where $\hat{\alpha_i}$ means that $\alpha_i$ is omitted) if $h_i$ is a prime.
\end{enumerate}
In particular, it is the second case that has the same rank as $\Phi$.  The coefficients of the highest root for the base $\Delta$ for each Dynkin diagram are summarized in the table \ref{chr}, where $\bullet$ corresponds to the root $-\alpha_0$. 

\begin{table}
\centering
\begin{tabular}{l l l l l l l l l}
\begin{tikzpicture}[inner sep=1.5pt,outer sep=0pt]
\node (t) at (-4,0) {\(A_r:\)};
\node [circle, fill=black, radius=2pt] (ext) at (0,-1) {};
\node [circle,radius=2pt,draw,label=1] (A) at (-3,0) {};
\node [circle,radius=2pt,draw,label=1] (B) at (-2,0) {};
\node [circle, radius=2pt,draw,label=1] (C) at (-1,0) {};
\node  (dot) at (0,0) {};
\node [circle,radius=2pt, draw,label=1] (D) at (1,0) {};
\node [circle, radius=2pt,draw,label=1] (E) at (2,0) {};
\node [circle,radius=2pt, draw,label=1] (F) at (3,0) {};
\draw (A) -- (B);
\draw (B) -- (C);
\draw[dashed] (C) -- (dot);
\draw[dashed] (dot) -- (D);
\draw (D) -- (E);
\draw (E) -- (F);
\draw (ext)--(A);
\draw (ext)--(F);
\end{tikzpicture}
\\
\begin{tikzpicture}[inner sep=1.5pt,outer sep=0pt]
\node (t) at (-4,0) {\(B_r:\)};
\node [circle, fill=black, radius=2pt] (ext) at (-3, 0.5) {};
\node [circle,radius=2pt,draw,label=1] (E) at (-3,-0.5) {};
\node [circle,radius=2pt,draw,label=2] (F) at (-2,0) {};
\node [circle, radius=2pt,draw,label=2] (A) at (-1,0) {};
\node  (dot) at (0,0) {};
\node [circle,radius=2pt, draw,label=2] (p) at (1,0) {};
\node [circle, radius=2pt,draw,label=2] (B) at (2,0) {};
\node  (C) at (2.5,0) {};
\node [circle,radius=2pt, draw,label=2] (D) at (3,0) {};
\draw (ext)--(F);
\draw (E) -- (F);
\draw (F) -- (A);
\draw[dashed] (A) -- (dot);
\draw[dashed] (dot) -- (p);
\draw (p) -- (B);
\draw[double] (C.west) - - (D);
\draw[-{Classical TikZ Rightarrow[black,length=0.3em,]}, double,] (B) -- (C.east);
\end{tikzpicture}
\\
\begin{tikzpicture}[inner sep=1.5pt,outer sep=0pt]
\node (t) at (-5,0) {\(C_r:\)};
\node [circle, fill=black, radius=2pt] (ext) at (-4, 0) {};
\node [circle,radius=2pt,draw,label=2] (E) at (-3,0) {};
\node [circle,radius=2pt,draw,label=2] (F) at (-2,0) {};
\node [circle, radius=2pt,draw,label=2] (A) at (-1,0) {};
\node  (dot) at (0,0) {};
\node [circle,radius=2pt, draw,label=2] (p) at (1,0) {};
\node [circle, radius=2pt,draw,label=2] (B) at (2,0) {};
\node  (C) at (2.5,0) {};
\node  (C') at (-3.5,0) {};
\node [circle,radius=2pt, draw,label=1] (D) at (3,0) {};
\draw[double] (ext)--(E);
\draw (E) -- (F);
\draw (F) -- (A);
\draw[dashed] (A) -- (dot);
\draw[dashed] (dot) -- (p);
\draw (p) -- (B);
\draw[-{Classical TikZ Rightarrow[black,length=0.3em,reversed]}, double,] (E) -- (C'.east);
\draw[double] (C.west) - - (D);
\draw[-{Classical TikZ Rightarrow[black,length=0.3em, reversed]}, double,] (B) -- (C.east);
\end{tikzpicture}
\\
\begin{tikzpicture}[inner sep=1.5pt,outer sep=0pt]
\node (t) at (-4,0) {\(D_r:\)};
\node [circle, fill=black, radius=2pt] (ext) at (-3,-0.5) {};
\node [circle,radius=2pt,draw,label=1] (A) at (-3,0.5) {};
\node [circle,radius=2pt,draw,label=2] (B) at (-2,0) {};
\node [circle, radius=2pt,draw,label=2] (C) at (-1,0) {};
\node  (dot) at (0,0) {};
\node [circle,radius=2pt, draw,label=2] (D) at (1,0) {};
\node [circle, radius=2pt,draw,label=2] (E) at (2,0) {};
\node [circle,radius=2pt, draw,label=right:1] (F1) at (3,0.5) {};
\node [circle,radius=2pt, draw,label=right: 1] (F2) at (3,-0.5) {};
\draw (ext)--(B);
\draw (A) -- (B);
\draw (B) -- (C);
\draw[dashed] (C) -- (dot);
\draw[dashed] (dot) -- (D);
\draw (D) -- (E);
\draw (E) -- (F1);
\draw (E) -- (F2);
\end{tikzpicture}
\\
\begin{tikzpicture}[inner sep=1.5pt,outer sep=0pt]
\node (t) at (-4,0) {\(E_6:\)};
\node [circle,radius=2pt,draw,label=1] (A) at (-3,0) {};
\node [circle,radius=2pt,draw,label=2] (B) at (-2,0) {};
\node [circle, radius=2pt,draw,label=3] (C) at (-1,0) {};
\node [circle,radius=2pt, draw,label=2] (D) at (0,0) {};
\node [circle, radius=2pt,draw,label=1] (E) at (1,0) {};
\node [circle, radius=2pt,draw,label=right:{2}] (F) at (-1,-1) {};
\node [circle, fill=black, radius=2pt] (ext) at (-1,-2) {};
\draw (A) -- (B);
\draw (B) -- (C);
\draw (C) -- (D);
\draw (D) -- (E);
\draw (C) -- (F);
\draw (ext)--(F);
\end{tikzpicture}
\\
\begin{tikzpicture}[inner sep=1.5pt,outer sep=0pt]
\node (t) at (-5,0) {\(E_7:\)};
\node [circle, fill=black, radius=2pt] (ext) at (-4,0) {};
\node [circle,radius=2pt,draw,label=1] (A) at (-3,0) {};
\node [circle,radius=2pt,draw,label=2] (B) at (-2,0) {};
\node [circle, radius=2pt,draw,label=3] (C) at (-1,0) {};
\node [circle,radius=2pt, draw,label=4] (D) at (0,0) {};
\node [circle, radius=2pt,draw,label=3] (E) at (1,0) {};
\node [circle, radius=2pt,draw,label=2] (F) at (2,0) {};
\node [circle, radius=2pt,draw,label=right:{2}] (G) at (0,-1) {};
\draw (ext)--(A);
\draw (A) -- (B);
\draw (B) -- (C);
\draw (C) -- (D);
\draw (D) -- (E);
\draw (E) -- (F);
\draw (D) -- (G); 
\end{tikzpicture}
\\
\begin{tikzpicture}[inner sep=1.5pt,outer sep=0pt]
\node (t) at (-6,0) {\(E_8:\)};
\node [circle, fill=black, radius=2pt] (ext) at (-5,0) {};
\node [circle,radius=2pt,draw,label=2] (Z) at (-4,0) {};
\node [circle,radius=2pt,draw,label=3] (A) at (-3,0) {};
\node [circle,radius=2pt,draw,label=4] (B) at (-2,0) {};
\node [circle, radius=2pt,draw,label=5] (C) at (-1,0) {};
\node [circle,radius=2pt, draw,label=6] (D) at (0,0) {};
\node [circle, radius=2pt,draw,label=4] (E) at (1,0) {};
\node [circle, radius=2pt,draw,label=2] (F) at (2,0) {};
\node [circle, radius=2pt,draw,label=right:{3}] (G) at (0,-1) {};
\draw (ext)--(Z);
\draw (Z) -- (A);
\draw (A) -- (B);
\draw (B) -- (C);
\draw (C) -- (D);
\draw (D) -- (E);
\draw (E) -- (F);
\draw (D) -- (G); 
\end{tikzpicture}
\\
\begin{tikzpicture}[inner sep=1.5pt,outer sep=0pt]
\node (t) at (-2,0) {\(F_4:\)};
\node [circle, fill=black, radius=2pt] (ext) at (-1,0) {};
\node [circle,radius=2pt, draw,label=2] (p) at (0,0) {};
\node [circle, radius=2pt,draw,label=3] (B) at (1,0) {};
\node  (C) at (1.5,0) {};
\node [circle,radius=2pt, draw, label=4] (D) at (2,0) {};
\node [circle, radius=2pt,draw,label=2] (E) at (3,0) {};
\draw (ext)--(p);
\draw (D) -- (E);
\draw (p) -- (B);
\draw[double] (C.west) - - (D);
\draw[-{Classical TikZ Rightarrow[black,length=0.3em]}, double] (B) - - (C.east);
\end{tikzpicture}
\\
\begin{tikzpicture}[inner sep=1.5pt,outer sep=0pt]
\node (t) at (-1,0) {\(G_2:\)};
\node [circle, fill=black, radius=2pt] (ext) at (0,0) {};
\node [circle, radius=2pt,draw,label=2] (B) at (1,0) {};
\node  (B1) at (1,0.05) {};
\node  (B2) at (1, -0.05) {};
\node  (C) at (1.5,0) {};
\node  (C1) at (1.5,0.05) {};
\node  (C2) at (1.5, -0.05) {};
\node  (D1) at (2,0.05) {};
\node  (D2) at (2, -0.05) {};
\node [circle,radius=2pt, draw,label=3] (D) at (2,0) {};
\draw (ext)--(B);
\draw (C1) -- (B1);
\draw (C2) -- (B2);
\draw (C1) -- (D1);
\draw (C2) -- (D2);
\draw (C.east) - - (D);
\draw[-{Classical TikZ Rightarrow[black,length=0.3em]}] (B) -- (C.east);
\end{tikzpicture}
\\
\end{tabular}
\caption{Coefficients of the highest root}
\label{chr}
\end{table}

Using the Borel-de Siebenthal theorem, we can determine elliptic elements for simple groups as follows.
\begin{aprop}\label{A2021}
Suppose that $G$ is a split simple reductive group defined over $\mathbb{F}_q$ in very good characteristics.   
There is a bijection between the set of conjugacy classes of split elliptic coendoscopic groups of $G$ defined over $\overbar{\mathbb{F}}_q$ different from $G$ and the subset of $\Delta_B$ consisting of those $\alpha$ so that $h_\alpha$ is prime. 
\end{aprop}
\begin{proof}
We continue to use the constructions in the proof of Proposition \ref{A1}.

Let $H$ be a split elliptic coendoscopic group of $G$ which is the connected centralizer of an elliptic element $s\in T(\overbar{\mathbb{F}}_q)$. By Proposition \ref{A1}, and Borel-de Siebenthal theorem, after a $W$ conjugation if necessary, there exists an $i$ ($0\leq i\leq r$) such that 
$$ Z_{G_{s}^0}(\overbar{\mathbb{F}}_q) = \{s\in T(\overbar{\mathbb{F}}_q) | \alpha_j(s)=1, j\neq i\}.  $$
The element $s$ belongs to $Z_G(\overbar{\mathbb{F}}_q)$ if and only if $i=0$. Otherwise, the index $i$ satisfies that $1\leq i\leq r$ and that $h_i$ is a prime number. Moreover, we must have
$$ \alpha_i(s)\neq 1,$$
and for any $0\leq j\leq r$, such that $j\neq i$ we have 
$$ \alpha_j(s)= 1.$$
Since the Weyl group acts simply transitively on the set of bases of a root system, we get a well-defined map sending $H$ to $\alpha_i$. It is clear from the construction that this map is injective. 

To see that it is surjective. Let $\alpha_i$ be a root in $\Delta_B$ so that $h_i$ is a prime number. 
Let $\hat{\Delta}_B^{\vee}=\{ \varpi_{\alpha_1}^{\vee},  \varpi_{\alpha_2}^{\vee}, \ldots,  \varpi_{\alpha_r}^{\vee} \}$ be the set of fundamental co-weights, i.e., a base in $\ago_{T,\mathbb{Q}}$ that is dual to $\Delta_B$. Since we are in very good characteristic, the order of the projection of any element of $\hat{\Delta}_B^{\vee}$ in $\ago_{T,\mathbb{Q}}/X_{*}(T)$ is not divisible by $p$, and we have $p\nmid h_i$ for every $i$. 
Consider \[t =\frac{1}{h_i}\varpi_{\alpha_i}^{\vee}\in \ago_{T,\mathbb{Q}}, \]
then the image of $t$ in $\ago_{T,\mathbb{Q}}/X_{*}(T)$ lies in $X_*(T)\otimes (\mathbb{Q}/\mathbb{Z})_{p'}$. Since $T(\overbar{\mathbb{F}}_q)\cong X_*(T)\otimes (\mathbb{Q}/\mathbb{Z})_{p'}$, we see that $t$ defines an element $s\in T(\overbar{\mathbb{F}}_q)$. It is clear that $G_{s}^0$ is a split elliptic coendoscopic group of $G$, since clearly \[\alpha_j(s)= 1, \quad \text{$j\neq i$}, \] 
\[\alpha_i(s)\neq 1,\]
and \[\alpha_0(s)=1.\] 
By Borel-de Siebenthal theorem, we conclude that $\{-\alpha_0, \alpha_1, \ldots,  \hat{\alpha_i}, \ldots, \alpha_r \}$ is a base of the root system of $(G_{s}^0,T)$. Therefore, $H$ is a split elliptic coendoscopic group of $G$. 
\end{proof}

\begin{aprop}\label{rationalite}
Suppose that the hypothesis on $q-1$ in the following table \ref{car} is satisfied for all simple factors of $G^{ad}$. Then every elliptic split coendoscopic group of $G_{\overbar{\mathbb{F}}_q}$ is defined over $\mathbb{F}_q$. If furthermore, the hypothesis in the following table \ref{car2} is satisfied, then \[Z_H(\mathbb{F}_q)=Z_H(\overbar{\mathbb{F}}_q)\] for every elliptic split coendoscopic group $H$ (including $G$) of $G$.
\begin{table}[H] 
\centering
$A_r: $ no requirement since there is no non-trivial elliptic element; 
$B_r:  4 | q-1$, |
$C_r:   4 | q-1$; ($r$ even), $8| q-1$ ($r$ odd),
$D_r: 4 | q-1$;  
$E_6:  18| q-1$; 
$E_7:  12| q-1$; 
$E_8:   30| q-1$; 
$F_4: 6| q-1$; 
$G_2: 6| q-1$. 
\caption{Hypothesis on \(q-1\)}
\label{car}
\end{table}
\begin{table}[H]  
\centering
$A_r$: $r+1| q-1$;  $B_r:  4 | q-1$,  $C_r:   8 | q-1$;  $D_r: 8 | q-1$;  $E_6:  18| q-1$;  $E_7:  24| q-1$;  $E_8:   90| q-1$;  $F_4: 6| q-1$;  $G_2: 6| q-1$.  
\caption{Further hypothesis on \(q-1\)}
\label{car2}
\end{table}

\end{aprop}
\begin{proof}
We continue to use constructions in the precedent propositions. Suppose first that $G$ is simple. 

Observe that under the embedding $T(\overbar{\mathbb{F}}_q)\hookrightarrow \ago_{T,\mathbb{Q}}/X_{*}(T)$, $T(\mathbb{F}_q)$ is sent to $(\ago_{T,\mathbb{Q}}/X_{*}(T))_{q-1}$, the $(q-1)$-torsion points in $\ago_{T,\mathbb{Q}}/X_{*}(T)$. Hence the element $s\in T(\overbar{\mathbb{F}}_q)$ induced by a $t=\frac{1}{h_i}\varpi_{\alpha_i}^{\vee}$ belongs to $T(\mathbb{F}_q)$ if and only if 
\begin{equation}\frac{1}{h_i}\varpi_{\alpha_i}^{\vee}\in \frac{1}{q-1}X_{*}(T). \end{equation}
A sufficient condition for this to hold is when 
\[\frac{1}{h_i}\varpi_{\alpha_i}^\vee \in  \frac{1}{q-1}\mathbb{Z}\Delta_B^{\vee}. \]
This holds if \begin{equation}\frac{q-1}{h_i}\mathbb{Z}\hat{\Delta}_B^\vee \subseteq \mathbb{Z}\Delta_B^{\vee}, \quad \forall  h_i \text{ prime }. \end{equation}

Since if the root system $\Phi(G, T)$ is of type $A, B, C, D, E, F, G$, then its dual root system $(\Phi(G, T)^{\vee}, \Delta_B^{\vee})$ is of type $A, C, B, D, E, F, G$. 
The result is clear from the previous discussion, and the structure of $\mathbb{Z}\hat{\Delta}_B/\mathbb{Z}\Delta_B$ is listed as follows.  
 \begin{table}[H]
\centering
\begin{tabular}{|c|c|c|c|c|c|c|c|c|c|c|}
\hline
 $A_r$& $B_r$ & $C_r$ &   $D_{2r}$ & $ D_{2r+1}$ & $E_{6}$ & $E_7$& $E_8$& $F_4$&$G_2$ \\
\hline
$\mathbb{Z}/(r+1)\mathbb{Z}$   &$\mathbb{Z}/2 \mathbb{Z}$    &    $\mathbb{Z}/2 \mathbb{Z}$  & $\mathbb{Z}/2 \mathbb{Z}\times \mathbb{Z}/2 \mathbb{Z}$  & $\mathbb{Z}/4 \mathbb{Z}$  &$\mathbb{Z}/3 \mathbb{Z}$   & $\mathbb{Z}/2 \mathbb{Z}$  & $\{1\}$  & $\{1\}$  & $\{1\}$ \\
\hline
\end{tabular}
\caption{Structures of $\mathbb{Z}\hat{\Delta}_B/\mathbb{Z}\Delta_B$}
\end{table}

Now we consider the second statement about the center $Z_{H}$. Since $H$ is semisimple, the group $Z_H(\overbar{\mathbb{F}}_q)\cong X_{*}(T)/\mathbb{Z}\Delta_{B_H}^{\vee}$ where $B_H=B\cap H$. 
A sufficient condition for $Z_H(\overbar{\mathbb{F}}_q)=Z_H(\mathbb{F}_q)$ is that 
\begin{equation}  (X_{*}(T)/\mathbb{Z}\Delta_{B_H}^{\vee})_{q-1}=X_{*}(T)/\mathbb{Z}\Delta_{B_H}^{\vee},\end{equation}
i.e. all elements in $X_{*}(T)/\mathbb{Z}\Delta_{B_H}^{\vee}$ are $(q-1)$-torsion. 
Again we can use the table above and the classification of split elliptic coendoscopic groups; we leave the details to the reader (our result is not optimal). 
\end{proof}

\end{document}